\documentclass[12pt]{amsart}
\usepackage{graphicx}
\usepackage{amsmath}
\usepackage{amssymb}
\usepackage{setspace}
\newtheorem{thm}{Theorem}[section]
\newtheorem{cor}[thm]{Corollary}

\newtheorem{lem}[thm]{Lemma}
\newtheorem{prop}[thm]{Proposition}
\theoremstyle{definition}
\newtheorem{defn}[thm]{Definition}
\theoremstyle{remark}
\newtheorem{rem}[thm]{Remark}
\theoremstyle{question}
\numberwithin{equation}{section}

\begin{document}

\title[2D Schr\"{o}dinger-Newton equations]{Existence and symmetry of solutions to 2-D Schr\"{o}dinger-Newton equations}

\author{Daomin Cao, Wei Dai, Yang Zhang}
\address{School of Mathematics and Information Science, Guangzhou University, Guangzhou 510405, People¡¯s Republic of China}
\email{dmcao@amt.ac.cn}

\address{School of Mathematical Sciences, Beihang University (BUAA), Beijing 100083, P. R. China, and LAGA, UMR 7539, Institut Galil\'{e}e, Universit\'{e} Sorbonne Paris Cit\'e, 93430 - Villetaneuse, France}
\email{weidai@buaa.edu.cn}

\address{School of Mathematics and Statistics, Central South University, Changsha 410075, People¡¯s Republic of China}
\email{zhangyang@amss.ac.cn}

\thanks{D. Cao was supported by NNSF of China (No. 11831009) and Chinese Academy of Sciences (No. QYZDJ-SSW-SYS021). W. Dai was supported by the NNSF of China (No. 11971049), the Fundamental Research Funds for the Central Universities and the State Scholarship Fund of China (No. 201806025011).}

\begin{abstract}
In this paper, we consider the following 2-D Schr\"{o}dinger-Newton equations
\begin{eqnarray*}
-\Delta u+a(x)u+\frac{\gamma}{2\pi}\left(\log(|\cdot|)*|u|^p\right){|u|}^{p-2}u=b{|u|}^{q-2}u \qquad \text{in} \,\,\, \mathbb{R}^{2},
\end{eqnarray*}
where $a\in C(\mathbb{R}^{2})$ is a $\mathbb{Z}^{2}$-periodic function with $\inf_{\mathbb{R}^{2}}a>0$, $\gamma>0$, $b\geq0$, $p\geq2$ and $q\geq 2$. By using ideas from \cite{CW,DW,Stubbe}, under mild assumptions, we obtain existence of ground state solutions and mountain pass solutions to the above equations for $p\geq2$ and $q\geq2p-2$ via variational methods. The auxiliary functional $J_{1}$ plays a key role in the cases $p\geq3$. We also prove the radial symmetry of positive solutions (up to translations) for $p\geq2$ and $q\geq 2$. The corresponding results for planar Schr\"{o}dinger-Poisson systems will also be obtained. Our theorems extend the results in \cite{CW,DW} from $p=2$ and $b=1$ to general $p\geq2$ and $b\geq0$.
\end{abstract}

\maketitle

{\small {\bf Keywords:} Logarithmic convolution potential, variational methods, Schr\"{o}dinger-Newton equations, Schr\"{o}dinger-Poisson systems, positive solutions, radial symmetry. \\

}

\section{Introduction}
In this paper, we consider the following 2-D Schr\"{o}dinger-Newton equations
\begin{equation}\label{S-N}
  -\Delta u+a(x)u+\frac{\gamma}{2\pi}\left(\log(|\cdot|)*|u|^p\right){|u|}^{p-2}u=b{|u|}^{q-2}u \qquad \text{in} \,\,\, \mathbb{R}^{2},
\end{equation}
where $a\in C(\mathbb{R}^{2})$ is a $\mathbb{Z}^{2}$-periodic function with $\inf_{\mathbb{R}^{2}}a>0$, $\gamma>0$, $b\geq0$, $p\geq2$ and $q\geq 2$. Solutions to Schr\"{o}dinger-Newton equations \eqref{S-N} also solve the following Schr\"{o}dinger-Poisson systems (also called Schr\"{o}dinger-Maxwell systems)
\begin{equation}\label{S-P}
  \left\{\begin{array}{ll}{-\Delta u+a(x) u+\gamma w|u|^{p-2}u=b|u|^{q-2}u} & {\text { in } \mathbb{R}^{2},} \\ {\Delta w=|u|^{p}} & {\text { in } \mathbb{R}^{2}.}\end{array}\right.
\end{equation}

Schr\"{o}dinger-Newton equations \eqref{S-N} and Schr\"{o}dinger-Poisson systems \eqref{S-P} have numerous important applications in physics and model many phenomena in quantum mechanics (see \cite{BL,FL,L,Lions,MPT,P}). In general dimensions $\mathbb{R}^{d}$, solution $u$ to Schr\"{o}dinger-Poisson systems gives a standing (or solitary) wave solution $\psi(x,t)=e^{-i\lambda t}u(x)$ ($\lambda\in\mathbb{R}$) of dynamic Schr\"{o}dinger-Poisson systems of the type
\begin{equation}\label{D-S-P}
  \left\{\begin{array}{ll}{i\partial_{t}\psi-\Delta\psi+E(x)\psi+\gamma w|\psi|^{p-2}\psi=b|\psi|^{q-2}\psi} & {\text { in } \mathbb{R}^{d}\times\mathbb{R},} \\ {\Delta w=|\psi|^{p}} & {\text { in } \mathbb{R}^{d}\times\mathbb{R},}\end{array}\right.
\end{equation}
where $E(x)=a(x)-\lambda$ is a real external potential, the function $w$ represents an internal potential for a nonlocal self-interaction of the wave function $\psi$.

For space dimension $d=3$, \eqref{S-N} and \eqref{S-P} have been extensively studied. In the cases $p=2$, $\gamma>0$ and $b=0$, \eqref{S-N} and \eqref{S-P} were introduced by Pekar \cite{Pekar} in 1954 to describe the quantum mechanics of a polaron at rest and by Choquard in 1976 to describe an electron trapped in its hole. In \cite{P}, Penrose obtained \eqref{S-N} in his discussion about the self gravitational collapse of a quantum-mechanical system. Lieb \cite{L} derived the existence of a unique ground state of \eqref{S-P} which is positive and radially symmetric, by using a minimization argument. Lions proved in \cite{Lions} the existence of infinitely many distinct radially symmetric solutions when $a(x)$ is a nonnegative and radially symmetric potential. In the cases $b\neq0$, there are also a large amount of papers devoted to the study of existence of solutions for \eqref{S-N} and \eqref{S-P}. For results on the cases $\gamma<0$, $b>0$ and $2<q<6$, please refer to \cite{AR,R}. In the cases $\gamma>0$ and $b<0$, \eqref{S-P} represents a Hartree model for crystals (see \cite{Mu}). For more related results in three and higher dimensions, please refer to \cite{BL,BS,CD,CL1,DFHQW,DFQ,DL,DQ,FL,MS,MS1,MZ,TT} and the references therein.

In this paper, we will investigate the existence of ground state solutions and symmetry of positive solutions for \eqref{S-N} and \eqref{S-P} in two dimension case. For space dimension $d=2$, there are very few literature (see \cite{BCS,CS,CSV,CW,DW,Stubbe}). Unlike higher dimension cases $d\geq3$, since the corresponding energy functional is not well-defined on $H^{1}(\mathbb{R}^{2})$, variational methods have rarely been used in the planar case. For the cases that $a\geq0$ is a constant, $\gamma>0$, $p=2$ and $b=0$, Stubbe \cite{Stubbe} set up a variational framework for \eqref{S-N} within a subspace of $H^{1}(\mathbb{R}^{2})$. Among other things, he derived the existence of a unique ground state solution which is a positive spherically symmetric decreasing function, by using strict rearrangement inequalities. For the cases $a\in C(\mathbb{R}^{2})$ is a $\mathbb{Z}^{2}$-periodic function with $\inf_{\mathbb{R}^{2}}a>0$, $\gamma>0$, $b\geq0$, $p=2$ and $q\geq 4$, by developing new ideas and estimates within the variational framework, Cingolani and Weth \cite{CW} proved the existence of ground states and high energy solutions for \eqref{S-N} and \eqref{S-P}. Their key tool is a surprisingly strong compactness condition for Cerami sequences which is not available for the corresponding problem in higher space dimensions. Subsequently, Du and Weth \cite{DW} removed the restriction $q\geq4$ in \cite{CW} and derived the existence of ground states and high energy solutions for \eqref{S-N} and \eqref{S-P}, by exploring the more complicated underlying functional geometry in the case $2<q<4$ with a different variational approach.

Following ideas from \cite{CW,Stubbe}, we will apply variational methods to study the existence of ground state solutions to 2-D Schr\"{o}dinger-Newton equation \eqref{S-N} with general $p\geq2$ and $q\geq 2p$. More precisely, we will consider the following energy functional:
\begin{eqnarray}\label{1-1}
  && I(u):=\frac{1}{2}\int_{\mathbb{R}^2}\left(|\nabla u|^2+a(x)u^2\right)dx+\frac{\gamma}{4p\pi}\int_{\mathbb{R}^2}\int_{\mathbb{R}^2}\log\left(|x-y|\right)|u|^p(x)|u|^p(y)dxdy \\
  \nonumber && \qquad\quad -\frac{b}{q}\int_{\mathbb{R}^{2}}|u|^qdx,
\end{eqnarray}
which is not well-defined on $H^{1}(\mathbb{R}^{2})$. Inspired by Stubbe \cite{Stubbe}, we will consider $I$ in the smaller Banach subspace:
\begin{equation}\label{1-2}
  X:=\left\{ u\in H^1(\mathbb{R}^2):\int_{\mathbb{R}^2}\log{(1+|x|)}|u|^p(x)dx<\infty\right\}=H^1 \cap L^p(d \mu)
\end{equation}
with the Radon measure $d\mu:=\log{(1+|x|)}dx$. By Lemma \ref{lemma1} in Section 2, we can see that the energy functional $I$ defines a $C^1$-functional on $X$, moreover, Schr\"{o}dinger-Newton equation \eqref{S-N} is the corresponding Euler-Lagrange equation for $I$, critical points $u\in X$ of $I$ are strong solutions of \eqref{S-N} in $W^{2,r}(\mathbb{R}^2)$ for all $r\geq 1$, and they are classical solutions in $C^{2}(\mathbb{R}^2)$ if $a$ is H\"{o}lder continuous. Thus $X$ provides a variational framework for \eqref{S-N}. Nevertheless, we should note that the norm of $X$ is not translation invariant, while the energy functional $I$ is invariant under any $\mathbb{Z}^{2}$-translations. Thus there will be more additional difficulties than higher dimension cases $d\geq3$.

By using the methods from \cite{CW}, we can prove that the ground state energy
\begin{equation}\label{1-3}
  c_g:=\inf\left\{I(u): \, u\in X\setminus \{0\}, \, I'(u)=0\right\}
\end{equation}
can be attained, and hence derive the ground state solutions for \eqref{S-N}. Our first main result is the following theorem.
\begin{thm}\label{thm1}
Assume $p\geq2$ and $q\geq 2p$. Then \eqref{S-N} admits a pair of ground state solutions $\pm u \in X\setminus\{0\}$ such that $I(u)=c_g$. Moreover, the restriction of $I$ to the associated Nehari manifold $\mathcal{N}:=\left \{u\in X\setminus \{0\}:\langle I^{\prime}(u),u\rangle=0\right\}$ admits a global minimum and each minimizer $u\in\mathcal{N}$ of $I\big|_{\mathcal{N}}$ is also a (ground state) solution to \eqref{S-N} which doesn't change sign and obeys the mini-max characterization
\begin{equation*}
  I(u)=\inf_{u \in X}\sup_{t\in\mathbb{R}}I(tu).
\end{equation*}
\end{thm}

Next, assuming that $a(x)=a>0$ is a constant, $q\geq 2p-2$ and $q>2$, we investigate the existence of mountain pass and ground state solutions to the Schr\"{o}dinger-Newton equations \eqref{S-N}. To this end, we define the mountain pass value
\begin{equation}\label{mp}
  c_{m p}=\inf _{\gamma \in \Gamma} \max _{t \in[0,1]} I(\gamma(t)), \qquad \Gamma=\left\{\gamma \in C([0,1] ; X) \,|\, \gamma(0)=0, I(\gamma(1))<0\right\}.
\end{equation}
Inspired by ideas from \cite{R} (see also \cite{DW}), we define the auxiliary functionals $J_{k}:\,X\rightarrow\mathbb{R}$ ($k=1,2$) by
\begin{equation}\label{af}
\begin{aligned}
& J_{k}(u):=\int_{\mathbb{R}^{2}}\left(k|\nabla u|^{2}+(k-1)au^{2}-\frac{(kq-2)b}{q}|u|^{q}\right)dx-\frac{\gamma}{4\pi p}\left(\int_{\mathbb{R}^{2}}|u|^{p}dx\right)^{2} \\
&\qquad\quad +\frac{\gamma(kp-2)}{2p\pi}\int_{\mathbb{R}^{2}}\int_{\mathbb{R}^{2}}\log(|x-y|)|u|^{p}(x)|u|^{p}(u)dxdy.
\end{aligned}\end{equation}

Our second main result is the following theorem.
\begin{thm}\label{thm3}
Let $a(x)=a>0$ be a constant, $p\geq2$, $q\geq2p-2$ and $q>2$. For $2\leq p<3$, assume further $q\geq 2p-1$ \emph{or} $q<p+1$.

\smallskip

\noindent(i) We have $c_{mp}>0$, and equation \eqref{S-N} has a pair of solutions $\pm u\in X\setminus\{0\}$ such that $I(u)=c_{mp}$.

\smallskip

\noindent(ii) Equation \eqref{S-N} admits a pair of ground state solutions $\pm u \in X\setminus\{0\}$ such that $I(u)=c_g$.

\smallskip

\noindent(iii) Assume further that $q\geq2p-1$ if $2\leq p<3$, then $c_{g}=c_{mp}$. Moreover, let $k=1$ \emph{if and only if} $p\geq3$ and $k=2$ \emph{if and only if} $2\leq p<3$, then the restriction of $I$ to the set $\mathcal{M}_{k}:=\left \{u\in X\setminus \{0\}\mid J_{k}(u)=0\right\}$ admits a global minimum and each minimizer $u\in\mathcal{M}_{k}$ of $I\big|_{\mathcal{M}_{k}}$ is also a (ground state) solution to \eqref{S-N} which doesn't change sign and obeys the mini-max characterization
\begin{equation*}
  I(u)=\inf_{u \in X\setminus\{0\}}\sup_{t>0}I(u_{t,k}),
\end{equation*}
where $u_{t,k}\in X$ is defined by $u_{t,k}(x):=t^{k}u(tx)$ for any $u\in X\setminus\{0\}$ and $t>0$.
\end{thm}

\begin{rem}\label{rem0}
Being different from \cite{R} and \cite{DW}, besides using the auxiliary functional $J_2$ in the cases $2\leq p<3$, we also define the auxiliary functional $J_1$ and use it to deal with the cases $p\geq3$ and $q\geq2p-2$. Please see the crucial lemma Lemma \ref{lemma13}.
\end{rem}

\begin{rem}\label{rem4}
When $2<p<3$, the cases $p+1\leq q<2p-1$ are not included in our Theorem \ref{thm3}. For $2\leq p<3$, we do not know whether $c_{g}=c_{mp}$ and there exists a ground state solution that may change sign or not in the cases $2p-2\leq q<2p-1$.
\end{rem}

By Lemma \ref{lemma1} in Section 2, we can see that every solution $u\in X$ to Schr\"{o}dinger-Newton equations \eqref{S-N} also gives a solution $(u,w)$ with $w:=\frac{1}{2\pi}\int_{\mathbb{R}^{2}}\log\left(|x-y|\right)|u|^{p}(y)dy$ to Schr\"{o}dinger-Poisson systems \eqref{S-P}. As a consequence of Theorem \ref{thm1}, Theorem \ref{thm3} and Lemma \ref{lemma1}, we derive the existence of ground state solutions for \eqref{S-P}.
\begin{cor}\label{cor1}
Let $p\geq2$, $q\geq 2p-2$ and $q>2$. Assume further that $a(x)=a>0$ is a constant if $q<2p$, and $q\geq 2p-1$ \emph{or} $2p-2\leq q<p+1$ if $2\leq p<3$. Then \eqref{S-P} admits a pair of ground state solutions $(\pm u,w)$ with $\pm u \in X $ satisfying $I(u)=c_g$. Assume further that $q\geq3$ if $2\leq p<\frac{5}{2}$, then the ground state solution $(u,w)$ doesn't change sign. Moreover, $u\in L^{\infty}(\mathbb{R}^{2})$ decays faster than exponential functions and $w\rightarrow+\infty$ as $|x|\rightarrow\infty$, if $a$ is H\"{o}lder continuous, then $(u,w)$ is a classical solution to \eqref{S-P}. The rest of conclusions in Theorem \ref{thm1} and Theorem \ref{thm3} also hold for Schr\"{o}dinger-Poisson systems \eqref{S-P}.
\end{cor}

Since we have already derived the existence of positive solutions to \eqref{S-N} and \eqref{S-P}, it is natural for us to classify positive solutions via their geometric properties. First, by using the method of moving planes (see \cite{CD,CL,CLO,CW,DFHQW,DFQ,DQ,GNN,MZ}), we will study the symmetry property of classical solutions $(u,w)$ to the following generalized Schr\"{o}dinger-Poisson systems
\begin{equation}\label{g-S-P}
  \left\{\begin{array}{ll}{-\Delta u-\gamma w|u|^{p-2}u=f(u)} & {\text { in } \mathbb{R}^{2},} \\ {-\Delta w=|u|^{p}} & {\text { in } \mathbb{R}^{2},}\end{array}\right.
\end{equation}
subject to the conditions
\begin{equation}\label{conditions}
  0<u\in L^{\infty}(\mathbb{R}^{2}) \quad \text{and} \quad w(x)\rightarrow -\infty, \quad \text{as} \,\, |x|\rightarrow\infty.
\end{equation}

We have the following result on symmetry of classical solutions for \eqref{g-S-P} and \eqref{conditions}.
\begin{thm}\label{thm2}
Assume $p\geq2$ and $f: \, \mathbb{R}\rightarrow \mathbb{R}$ is locally Lipschitz with $f(0)=0$. If $p>2$, we assume further that there exist some $\epsilon_{0}>0$ and $a_{0}>0$ small such that
\begin{equation}\label{monotonicity}
  \sup_{0<x\neq y<\epsilon_{0}}\frac{f(x)-f(y)}{x-y}\leq-a_{0}<0.
\end{equation}
Then every classical solution $(u,w)$ of \eqref{g-S-P}, \eqref{conditions} is radially symmetric and strictly decreasing with respect to some $x_{0}\in \mathbb{R}^{2}$.
\end{thm}

\begin{rem}\label{rem1}
If the locally Lipschitz continuous function $f(u)=b|u|^{q-2}u-au$ with $q\geq 2$ and constant $a>0$, then we can derive from Theorem \ref{thm2} the symmetry of classical solutions $(u,w)$ for Schr\"{o}dinger-Poisson systems \eqref{S-P} subject to the conditions \eqref{conditions}.
\end{rem}

As a consequence of Theorem \ref{thm2}, we derive the symmetry of positive solutions for Schr\"{o}dinger-Newton equations \eqref{S-N}.
\begin{cor}\label{cor2}
Assume $p\geq2$, $q\geq 2$ and $a$ is a positive constant in \eqref{S-N}. Then every positive solution $u\in X$ of \eqref{S-N} is radially symmetric and strictly decreasing with respect to some $x_{0}\in \mathbb{R}^{2}$.
\end{cor}

\begin{rem}\label{rem2}
In \cite{CW}, Cingolani and Weth have proved radial symmetry (up to translation) of positive solutions to \eqref{S-N} and \eqref{g-S-P} for $p=2$. Our Theorem \ref{thm2} and Corollary \ref{cor2} extend the symmetry results in \cite{CW} from $p=2$ to general $p\geq2$. The corresponding results for higher dimensions $d\geq3$ can be found in \cite{CD,DFHQW,DFQ,DL,DQ,L,MPT,MZ}. We should note that the 2-D case is quite different from higher dimension cases $d\geq3$, since the logarithmic convolution kernel is sign-changing.
\end{rem}

The rest of this paper is organized as follows. In Section 2, we will set up the variational framework and establish some useful preliminary lemmas and estimates for energy functional $I$ and function space $X$. We establish the key strong compactness condition for Cerami sequences (up to $\mathbb{Z}^{2}$-translations) and the quantitative deformation lemma in Section 3 (see Theorem \ref{Cerami} and Lemma \ref{deformation}). Section 4 is devoted to proving Theorem \ref{thm1}. In Section 5, we will prove Theorem \ref{thm3}. In Section 6, we carry out the proof of Theorem \ref{thm2} and Corollary \ref{cor2}.

\section{Preliminaries}
In this section, we will give some necessary preliminary knowledge for energy functional $I(u)$ and function space $X$. In the following, we assume that $a\in L^{\infty}(\mathbb{R}^{2})$ satisfies $\inf_{\mathbb{R}^{2}}a>0$.

First, we introduce some basic notations. The function space $X$ is a Banach space equipped with the norm
\begin{equation}\label{2-1}
  \|u\|_X:=\|u\|_{H^1(\mathbb{R}^{2})} + \|u\|_{*},
\end{equation}
where
\begin{equation}\label{2-2}
  \|u\|_{*}:=\|u\|_{L^p(d\mu)}=\left(\int_{\mathbb{R}^{2}}\log\left(1+|x|\right)|u|^{p}dx\right)^{\frac{1}{p}}=\left(\int_{\mathbb{R}^{2}}|u|^pd\mu\right)^{\frac{1}{p}},
\end{equation}
and
\begin{equation}\label{2-3}
  \|u\|_{H^1}:=\sqrt{\langle u,u \rangle_{H^{1}}}=\left(\int_{\mathbb{R}^{2}}\left(|\nabla u|^2+a(x)u^2\right)dx\right)^{\frac{1}{2}}
\end{equation}
with the $H^{1}(\mathbb{R}^{2})$ equivalent inner product given by
\begin{equation}\label{2-4}
  \langle u,v \rangle_{H^{1}}:=\int_{\mathbb{R}^{2}}\left(\nabla u \cdot \nabla v+a(x)uv\right)dx.
\end{equation}

Now we define the following three bilinear functionals:
\begin{equation}\label{2-5}
  B_1(f,g):=\int_{\mathbb{R}^{2}}\int_{\mathbb{R}^{2}}\log\left(1+|x-y|\right)f(x)g(y)dxdy,
\end{equation}
\begin{equation}\label{2-6}
  B_2(f,g):=\int_{\mathbb{R}^{2}}\int_{\mathbb{R}^{2}}\log\left(1+\frac{1}{|x-y|}\right)f(x)g(y)dxdy,
\end{equation}
\begin{equation}\label{2-7}
  B_0(f,g):=B_1(f,g)-B_2(f,g)=\int_{\mathbb{R}^{2}}\int_{\mathbb{R}^{2}}\log(|x-y|)f(x)g(y)dxdy.
\end{equation}
By Hardy-Littlewood-Sobolev inequality, $B_2(f,g)\in L^{\infty}(\mathbb{R}^{2})$ and has the following upper bound:
\begin{equation}\label{2-8}
  \left|B_2(f,g)\right|\leq\int_{\mathbb{R}^{2}}\int_{\mathbb{R}^{2}}\frac{1}{|x-y|}\big|f(x)g(y)\big|dxdy\leq C\|f\|_{L^{\frac{4}{3}}(\mathbb{R}^{2})}\|g\|_{L^{\frac{4}{3}}(\mathbb{R}^{2})}.
\end{equation}
Correspondingly, we define the following functionals associated to the above bilinear forms:
\begin{equation}\label{2-9}
  V_1(u):=B_1(|u|^p,|u|^p)=\int_{\mathbb{R}^{2}}\int_{\mathbb{R}^{2}}\log\left(1+|x-y|\right)|u|^p(x)|u|^p(y)dxdy,
\end{equation}
\begin{equation}\label{2-10}
  V_2(u):=B_2(|u|^p,|u|^p)=\int_{\mathbb{R}^{2}}\int_{\mathbb{R}^{2}}\log\left(1+\frac{1}{|x-y|}\right)|u|^p(x)|u|^p(y)dxdy,
\end{equation}
\begin{equation}\label{2-11}
  V_0(u):=B_0(|u|^p,|u|^p)=\int_{\mathbb{R}^{2}}\int_{\mathbb{R}^{2}}\log(|x-y|)|u|^p(x)|u|^p(y)dxdy.
\end{equation}
As a consequence of \eqref{2-8}, we infer immediately the following boundedness for $V_2(u)$:
\begin{equation}\label{2-12}
  |V_2(u)| \leq C \||u|^p \|_{L^{\frac{4}{3}}}^2= C \|u \|_{L^{\frac{4p}{3}}}^{2p},
\end{equation}
which indicates that $V_{2}$ is well-defined for $u\in L^{\frac{4p}{3}}(\mathbb{R}^{2})$. Since $p\geq 2$ and $H^1(\mathbb{R}^2)\hookrightarrow L^r(\mathbb{R}^2)$ for each $2\leq r<+\infty$, it follows from \eqref{2-12} that $V_2(u)$ is well-defined for $H^1$ functions.

Using these notations, we can rewrite the energy functional in the following form:
\begin{equation}\label{2-13}
  I(u)=\frac{1}{2}\|u\|_{H^1(\mathbb{R}^{2})}^2+\frac{\gamma}{4p\pi}V_0(u)-\frac{b}{q}\|u\|_{L^q(\mathbb{R}^{2})}^q.
\end{equation}

Next, we will give several useful lemmas.

\begin{lem}\label{epsilon}
Given any $0<p<+\infty$. For any $\epsilon>0$, there is a constant $C_{\epsilon,p}>0$ such that, for all $a, \, b\in\mathbb{C}$,
\begin{equation*}
  \left||a+b|^p-|b|^p\right|\leq \epsilon|b|^p+C_{\epsilon,p}|a|^p.
\end{equation*}
\end{lem}
The proof of Lemma \ref{epsilon} is elementary, see \cite{LL} for example.

\begin{lem}\label{vanish} (P. L. Lions, \cite{Lions1}, 1984)
Let $r>0$ and $2\leq q<2^*$ ($2^{*}:=\frac{2d}{d-2}$ if $d\geq3$, $2^{*}=+\infty$ if $d=2$). If $(u_n)_{n}$ is bounded in $H^1(\mathbb{R}^d)$ and
\begin{equation*}
  \lim_{n\rightarrow\infty}\sup_{y \in \mathbb{R}^d}\int_{B_r(y)}|u_n|^qdx=0,
\end{equation*}
then $u_n\rightarrow0$ in $L^p(\mathbb{R}^d)$ for all $2<p<2^*$.
\end{lem}
For the proof of Lemma \ref{vanish}, we refer to \cite{Lions1} or \cite{W}.

\begin{lem}[Basic properties for the functionals $I$, $V_{i}$ and the function space $X$]\label{lemma0}
We have the following properties: \\
(1)\, $X=H^1\cap L^p(d\mu)$ is compactly embedded in $L^s(\mathbb{R}^2)$ for all $s\in[p,+\infty)$; \\
(2)\, The functionals $V_1$, $V_2$, $V_0$, $I\in C^{1}(X,\mathbb{R})$, moreover,
\begin{equation*}
  \big\langle V_i^{\prime}(u),v \big\rangle=2p B_i(|u|^p,|u|^{p-2}uv)
\end{equation*}
for every $u,v\in X$ and $i=0,1,2$; \\
(3)\, $V_1$ is weakly lower semi-continuous on $H^1(\mathbb{R}^2)$, $I$ is weakly lower semi-continuous on $X$ and lower semi-continuous on $H^1(\mathbb{R}^2)$.
\end{lem}
\begin{proof}
\emph{Proof of (1):} By Rellich's compact embedding theorem (see \cite{B}), we only need to show the uniformly $L^{s}$-integrability for any bounded sequence $\{u_n\}\subset X$. Indeed, since there exists a $M>0$ such that
\begin{equation}\label{2-14}
  \int_{|x|\geq R}\log(1+R)|u_n|^p dx\leq\int_{|x|\geq R}\log(1+|x|)|u_n|^p dx\leq M,
\end{equation}
one can infer that for any $\epsilon>0$, there exists a $R(\epsilon)>0$ sufficiently large such that, for any $R>R(\epsilon)$,
\begin{equation}\label{2-15}
  \int_{|x| \geq R}|u_n|^p dx \leq \frac{M}{\log (1+R)}\leq\epsilon,
\end{equation}
from which the uniformly $L^{s}$-integrability follows immediately. Hence, the embedding $X\hookrightarrow \hookrightarrow L^s(\mathbb{R}^{2})$ is compact for all $s\in[p,+\infty)$.

\emph{Proof of (2):} We first show that $V_1$ is well-defined on the function space $X$:
\begin{align}\label{2-16}
&\quad\,\, \big|V_1(u)\big|=\left|\int_{\mathbb{R}^{2}}\int_{\mathbb{R}^{2}}\log(1+|x-y|)|u|^p(x)|u|^p(y)dxdy\right| \\
                  \nonumber &\leq  \left| \int_{\mathbb{R}^{2}}\int_{\mathbb{R}^{2}} \log\left((1+|x|)(1+|y|)\right)|u|^p(x)|u|^p(y)dxdy \right| \\
                  \nonumber &=\left| \int_{\mathbb{R}^{2}}\int_{\mathbb{R}^{2}} \log(1+|x|)|u|^p(x)|u|^p(y)dxdy     \right|
                   +\left| \int_{\mathbb{R}^{2}}\int_{\mathbb{R}^{2}} \log(1+|y|)|u|^p(x)|u|^p(y)dxdy     \right| \\
                  \nonumber &=2\|u\|_*^p \|u\|_{L^p(\mathbb{R}^{2})}^p\leq C\|u\|^{2p}_{X}.
\end{align}
Now, letting $u_n\to u$ in $X$ as $n\rightarrow+\infty$, we get
\begin{align}\label{2-17}
&\quad\, \left|V_1(u_n)-V_1(u)\right| \\
\nonumber &=\left|\int_{\mathbb{R}^{2}}\int_{\mathbb{R}^{2}}\log(1+|x-y|)\big(|u_n|^p(x)|u_n|^p(y)-|u|^p(x)|u|^p(y)\big)dxdy\right| \\
                           \nonumber &\leq\int_{\mathbb{R}^{2}}\int_{\mathbb{R}^{2}}\log\left((1+|x|)(1+|y|)\right)\big||u_n|^p(x)|u_n|^p(y)-|u|^p(x)|u|^p(y)\big|dxdy \\
                           \nonumber &\leq \int_{\mathbb{R}^{2}}\int_{\mathbb{R}^{2}}\log(1+|x|)\big||u_n|^p(x)|u_n|^p(y)-|u|^p(x)|u|^p(y)\big|dxdy \\
                           \nonumber &\quad+\int_{\mathbb{R}^{2}}\int_{\mathbb{R}^{2}}\log(1+|y|)\big||u_n|^p(x)|u_n|^p(y)-|u|^p(x)|u|^p(y)\big|dxdy \\
                           \nonumber &=: I^{1}_{n}+I^{2}_{n}.
\end{align}
For the sequence $I^{1}_{n}$, one has, as $n\rightarrow+\infty$,
\begin{align}\label{2-18}
& I^{1}_{n}= \int_{\mathbb{R}^{2}}\int_{\mathbb{R}^{2}}\log(1+|x|)\big||u_n|^p(x)|u_n|^p(y)-|u|^p(x)|u|^p(y)\big|dxdy \\
\nonumber &\leq \int_{\mathbb{R}^{2}}\int_{\mathbb{R}^{2}} \log (1+|x|) |u_n|^p(x) \big| |u_n|^p(y)-|u|^p(y)\big| dxdy\\
\nonumber &\quad +\int_{\mathbb{R}^{2}}\int_{\mathbb{R}^{2}} \log (1+|x|) |u|^p(y) \big| |u_n|^p(x)-|u|^p(x)\big| dxdy \\
\nonumber &\leq p\|u_n\|_{*}^{p}\left(\|u_{n}\|_{L^p}^{p-1}+\|u\|_{L^p}^{p-1}\right)\|u_n-u\|_{L^p} \\
\nonumber &\quad +p\left(\|u_{n}\|_{*}^{p-1}+\|u\|_{*}^{p-1}\right)\|u_n-u\|_{*}\|u\|_{L^p}^{p}\longrightarrow 0.
\end{align}
Through similar calculations, we also have $I^{2}_{n}\rightarrow0$ as $n\rightarrow+\infty$. Thus we get $V_1(u_n) \to V_1(u)$ and the functional $V_1$ is continuous on $X$.

The Gateaux derivative of $V_{1}$ at $u\in X$ is given by
\begin{eqnarray}\label{2-19}
  && \left\langle V_1^{\prime}(u),v \right\rangle=2p B_1(|u|^p,|u|^{p-2}uv) \\
  \nonumber &&\qquad\qquad\,\,\, =2p\int_{\mathbb{R}^{2}}\int_{\mathbb{R}^{2}} \log(1+|x-y|)|u|^p(x)|u|^{p-2}(y)u(y)v(y)dxdy
\end{eqnarray}
for any $v\in X$. We have the following estimate:
\begin{align}\label{2-20}
\big| \left \langle V_1^{\prime}(u),v \right \rangle \big| &=2p\left| \int_{\mathbb{R}^{2}}\int_{\mathbb{R}^{2}} \log(1+|x-y|)|u|^p(x)|u|^{p-2}(y)u(y)v(y)dxdy  \right|\\
   \nonumber &\leq 2p \int_{\mathbb{R}^{2}}\int_{\mathbb{R}^{2}} \log(1+|x-y|)|u|^p(x)|u|^{p-1}(y)|v|(y)dxdy  \\
   \nonumber &\leq 2p \int_{\mathbb{R}^{2}}\int_{\mathbb{R}^{2}} \log(1+|x|)|u|^p(x)|u|^{p-1}(y)|v|(y)dxdy \\
   \nonumber &\quad +2p \int_{\mathbb{R}^{2}}\int_{\mathbb{R}^{2}} \log(1+|y|)|u|^p(x)|u|^{p-1}(y)|v|(y)dxdy \\
   \nonumber &\leq 2p \|u\|_{*}^{p}\| u\|_{L^p}^{p-1}\|v\|_{L^p}+2p\|u\|_{L^p}^p\|u\|_{*}^{p-1}\|v\|_{*} \\
   \nonumber &\leq C\|u\|_{X}^{2p-1}\|v\|_X,
\end{align}
thus $V_1^{\prime}(u)\in X'$ and $\|V_1^{\prime}(u)\|_{X'}\leq C\|u\|_{X}^{2p-1}$.

Now, assume $u_n\to u$ in $X$ as $n\rightarrow+\infty$. For each $v\in X$, we have
\begin{flalign}
\nonumber &\big| \left \langle V_1^{\prime}(u_n)- V_1^{\prime}(u),v \right \rangle \big|=2p\left|B_1(|u_n|^p,|u_n|^{p-2}u_nv)-B_1(|u|^p,|u|^{p-2}uv)\right| \\
  \nonumber \leq& 2p \int_{\mathbb{R}^{2}}\int_{\mathbb{R}^{2}} \log(1+|x-y|)\left||u_n|^p(x)|u_n|^{p-2}(y)u_n(y)v(y)-|u|^p(x)|u|^{p-2}(y)u(y)v(y)\right| dxdy\\
  \nonumber \leq& 2p \int_{\mathbb{R}^{2}}\int_{\mathbb{R}^{2}} \log(1+|x|)\left||u_n|^p(x)|u_n|^{p-2}(y)u_n(y)v(y)-|u|^p(x)|u|^{p-2}(y)u(y)v(y)\right| dxdy\\
  \nonumber &+2p \int_{\mathbb{R}^{2}}\int_{\mathbb{R}^{2}} \log(1+|y|)\left||u_n|^p(x)|u_n|^{p-2}(y)u_n(y)v(y)-|u|^p(x)|u|^{p-2}(y)u(y)v(y)\right| dxdy\\
  \nonumber =:& 2p\left(II^{1}_{n}+II^{2}_{n}\right).
\end{flalign}
For the sequence $II^{1}_{n}$, one has, as $n\rightarrow+\infty$,
\begin{flalign}\label{2-22}
II^{1}_{n}&\leq \int_{\mathbb{R}^{2}}\int_{\mathbb{R}^{2}} \log(1+|x|)|u_n|^p(x)\left| |u_n|^{p-2}(y)u_n(y)
-|u|^{p-2}(y)u(y)\right| |v|(y)dxdy\\
 \nonumber &\quad+ \int_{\mathbb{R}^{2}}\int_{\mathbb{R}^{2}} \log(1+|x|) \left| |u_n|^p(x) -|u|^p(x) \right| |u|^{p-1}(y)|v|(y) dxdy \\
 \nonumber &\leq \|u_n\|_*^p\int_{\mathbb{R}^{2}} \left| |u_n|^{p-2}(y)u_n(y)
     -|u|^{p-2}(y)u(y)\right| |v|(y)dy \\
 \nonumber &\quad+\|u\|_{L^p}^{p-1}\|v\|_{L^{p}}\int_{\mathbb{R}^{2}} \log(1+|x|) \left||u_n|^p(x)-|u|^p(x)\right| dx \\
 \nonumber &\leq \|u_n\|_{*}^{p}\int_{\mathbb{R}^{2}}|u_n|^{p-2}|u_n-u| |v|
     + \left| |u_n|^{p-2} -|u|^{p-2}\right||u||v| dy \\
 \nonumber & \quad+ p\|u\|_{L^p}^{p-1}\|v\|_{L^{p}}\left(\|u_{n}\|_{*}^{p-1}+\|u\|_{*}^{p-1}\right)\|u_{n}-u\|_{*} \\
 \nonumber &\leq \|u_n\|_{*}^{p}\bigg(\|u_n\|_{L^p}^{p-2}\|u_n-u\|_{L^p} \|v\|_{L^p}
       +\int_{\mathbb{R}^{2}}\left| |u_n|^{p-2}-|u|^{p-2}\right||u||v|dy\bigg) \\
 \nonumber &\quad + p\|u\|_{L^p}^{p-1}\|v\|_{L^{p}}\left(\|u_{n}\|_{*}^{p-1}+\|u\|_{*}^{p-1}\right)\|u_{n}-u\|_{*}.
\end{flalign}
Therefore, if $2\leq p\leq3$, we derive that, as $n\rightarrow+\infty$,
\begin{eqnarray}\label{2-23}
  II^{1}_{n}&\leq&\|u_n\|_{*}^{p}\bigg(\|u_n\|_{L^p}^{p-2}\|u_n-u\|_{L^p}+sgn(p-2)\|u_{n}-u\|_{L^{p}}^{p-2}\|u\|_{L^{p}}\bigg)\|v\|_{X} \\
      \nonumber &&+p\|u\|_{L^p}^{p-1}\left(\|u_{n}\|_{*}^{p-1}+\|u\|_{*}^{p-1}\right)\|u_{n}-u\|_{*}\|v\|_{X}\rightarrow0;
\end{eqnarray}
if $3<p<+\infty$, we derive that, as $n\rightarrow+\infty$,
\begin{eqnarray}\label{2-24}
  II^{1}_{n}&\leq&\|u_n\|_{*}^{p}\bigg(\|u_n\|_{L^p}^{p-2}\|u_n-u\|_{L^p}+(p-2)\left(\|u_{n}\|_{L^{p}}^{p-3}+\|u\|_{L^{p}}^{p-3}\right)\|u_{n}-u\|_{L^{p}} \\
      \nonumber && \|u\|_{L^{p}}\bigg)\|v\|_{X}+p\|u\|_{L^p}^{p-1}\left(\|u_{n}\|_{*}^{p-1}+\|u\|_{*}^{p-1}\right)\|u_{n}-u\|_{*}\|v\|_{X}\rightarrow0.
\end{eqnarray}
Similar estimates as \eqref{2-23} and \eqref{2-24} can also be obtained for $II^{2}_{n}$ and we can deduce that $II^{2}_{n}\rightarrow0$ as $n\rightarrow+\infty$. Hence, $V_1^{\prime}(u_n) \to V_1^{\prime}(u)$ in $X'$ and the Gateaux derivative $V_1^{\prime}$ is continuous, that is, $V_{1}\in C^{1}(X,\mathbb{R})$.

From \eqref{2-12}, we have known that $V_{2}$ is well-defined on $L^{\frac{4p}{3}}(\mathbb{R}^{2})$ and hence on $X$. We will show that $V_{2}\in C^{1}(L^{\frac{4p}{3}}(\mathbb{R}^{2}),\mathbb{R})$. To this end, letting $u_{n}\rightarrow u$ in $L^{\frac{4p}{3}}(\mathbb{R}^{2})$ as $n\rightarrow+\infty$, by Hardy-Littlewood-Sobolev inequality, we get
\begin{equation}\label{2-25}
\begin{aligned}&\left|V_{2}\left(u_{n}\right)-V_{2}(u)\right| \\
\leq & \int_{\mathbb{R}^{2}}\int_{\mathbb{R}^{2}} \log\left(1+\frac{1}{|x-y|}\right)|u_{n}|^{p}(x)\left||u_{n}|^{p}(y)-|u|^{p}(y)\right| dxdy \\ &+\int_{\mathbb{R}^{2}}\int_{\mathbb{R}^{2}} \log \left(1+\frac{1}{|x-y|}\right)\left||u_{n}|^{p}(x)-|u|^{p}(x)\right||u|^{p}(y) dx dy \\ \leq & \int_{\mathbb{R}^{2}}\int_{\mathbb{R}^{2}}\frac{p}{|x-y|} |u_{n}|^{p}(x)\left(|u_{n}|^{p-1}(y)+|u|^{p-1}(y)\right)\left|u_{n}(y)-u(y)\right|dxdy \\ &+\int_{\mathbb{R}^{2}}\int_{\mathbb{R}^{2}}\frac{p}{|x-y|}\left(|u_{n}|^{p-1}(x)+|u|^{p-1}(x)\right)\left|u_{n}(x)-u(x)\right||u|^{p}(y) dxdy \\ \leq & C\left(\|u_{n}\|_{L^{\frac{4p}{3}}}^{p}+\|u\|_{L^{\frac{4p}{3}}}^{p}\right)\left(\|u_{n}\|_{L^{\frac{4p}{3}}}^{p-1}+\|u\|_{L^{\frac{4p}{3}}}^{p-1}\right)
\|u_{n}-u\|_{L^{\frac{4p}{3}}}\rightarrow0, \end{aligned}
\end{equation}
thus $V_{2}\in C(L^{\frac{4p}{3}}(\mathbb{R}^{2}),\mathbb{R})$. The Gateaux derivative of $V_{2}$ at $u\in L^{\frac{4p}{3}}(\mathbb{R}^{2})$ is given by
\begin{eqnarray}\label{2-26}
  && \left\langle V_2^{\prime}(u),v \right\rangle=2p B_2(|u|^p,|u|^{p-2}uv) \\
 \nonumber &=&2p\int_{\mathbb{R}^{2}}\int_{\mathbb{R}^{2}} \log\left(1+\frac{1}{|x-y|}\right)|u|^p(x)|u|^{p-2}(y)u(y)v(y)dxdy
\end{eqnarray}
for any $v\in L^{\frac{4p}{3}}(\mathbb{R}^{2})$. Thus we have
\begin{equation}\label{2-27}
\begin{aligned}\left|\left\langle V_2^{\prime}(u),v \right\rangle\right| & \leq 2p\int_{\mathbb{R}^{2}}\int_{\mathbb{R}^{2}}\frac{1}{|x-y|} |u|^{p}(x)|u(y)|^{p-1}|v(y)|dxdy \\ & \leq C\left\||u|^{p}\right\|_{L^{\frac{4}{3}}}\left\||u|^{p-1}|v|\right\|_{L^{\frac{4}{3}}}\leq C\|u\|_{L^{\frac{4p}{3}}}^{2p-1}\|v\|_{L^{\frac{4p}{3}}}, \end{aligned}
\end{equation}
thus $V_2^{\prime}(u)\in\left(L^{\frac{4p}{3}}(\mathbb{R}^{2})\right)'$ and $\|V_2^{\prime}(u)\|\leq C\|u\|_{L^{\frac{4p}{3}}}^{2p-1}$.

Now, assume $u_n\to u$ in $L^{\frac{4p}{3}}(\mathbb{R}^{2})$ as $n\rightarrow+\infty$. For each $v\in L^{\frac{4p}{3}}(\mathbb{R}^{2})$, we have
\begin{flalign}\label{2-28}
&\big| \left \langle V_2^{\prime}(u_n)- V_2^{\prime}(u),v \right \rangle \big|=2p\left|B_2(|u_n|^p,|u_n|^{p-2}u_nv)-B_2(|u|^p,|u|^{p-2}uv)\right| \\
  \nonumber \leq& 2p \int_{\mathbb{R}^{2}}\int_{\mathbb{R}^{2}} \log\left(1+\frac{1}{|x-y|}\right)\left||u_n|^p(x)-|u|^{p}(x)\right||u_n|^{p-1}(y)|v|(y)dxdy\\
  \nonumber &+2p\int_{\mathbb{R}^{2}}\int_{\mathbb{R}^{2}}\log\left(1+\frac{1}{|x-y|}\right)|u|^p(x)\left||u_n|^{p-2}(y)u_n(y)-|u|^{p-2}(y)u(y)\right||v|(y)dxdy\\
  \nonumber \leq& \int_{\mathbb{R}^{2}}\int_{\mathbb{R}^{2}} \frac{2p}{|x-y|}\left(|u_n|^{p-1}(x)+|u|^{p-1}(x)\right)\left|u_n(x)-u(x)\right||u_n|^{p-1}(y)|v|(y)dxdy \\
 \nonumber &+\int_{\mathbb{R}^{2}}\int_{\mathbb{R}^{2}}\frac{2p}{|x-y|}|u|^{p}(x)\left|u_n(y)-u(y)\right||u_n|^{p-2}(y)|v|(y)dxdy \\
 \nonumber &+\int_{\mathbb{R}^{2}}\int_{\mathbb{R}^{2}}\frac{2p}{|x-y|}|u|^{p}(x)\left||u_n|^{p-2}(y)-|u|^{p-2}(y)\right||u|(y)|v|(y)dxdy.
\end{flalign}
Therefore, by Hardy-Littlewood-Sobolev inequality, if $2\leq p\leq3$, we can derive that, as $n\rightarrow+\infty$,
\begin{eqnarray}\label{2-29}
 && \big| \left \langle V_2^{\prime}(u_n)- V_2^{\prime}(u),v \right \rangle \big| \\
  \nonumber &\leq& C\|u\|_{L^{\frac{4p}{3}}}^{p}\bigg(\|u_n\|_{L^{\frac{4p}{3}}}^{p-2}\|u_n-u\|_{L^{\frac{4p}{3}}}+sgn(p-2)\|u_{n}-u\|_{L^{\frac{4p}{3}}}^{p-2}\|u\|_{L^{\frac{4p}{3}}}\bigg)
  \|v\|_{L^{\frac{4p}{3}}} \\
  \nonumber &&+C\|u_{n}\|_{L^{\frac{4p}{3}}}^{p-1}\left(\|u_{n}\|_{L^{\frac{4p}{3}}}^{p-1}+\|u\|_{L^{\frac{4p}{3}}}^{p-1}\right)\|u_{n}-u\|_{L^{\frac{4p}{3}}}
  \|v\|_{L^{\frac{4p}{3}}}\rightarrow0;
\end{eqnarray}
if $3<p<+\infty$, we can derive that, as $n\rightarrow+\infty$,
\begin{eqnarray}\label{2-30}
 && \quad \big| \left \langle V_2^{\prime}(u_n)- V_2^{\prime}(u),v \right \rangle \big| \\
  \nonumber &&\leq C\|u\|_{L^{\frac{4p}{3}}}^{p}\bigg(\|u_n\|_{L^{\frac{4p}{3}}}^{p-2}\|u_n-u\|_{L^{\frac{4p}{3}}}+(p-2)
  \left(\|u_{n}\|_{L^{\frac{4p}{3}}}^{p-3}+\|u\|_{L^{\frac{4p}{3}}}^{p-3}\right)\|u_{n}-u\|_{L^{\frac{4p}{3}}}\|u\|_{L^{\frac{4p}{3}}}\bigg)
  \\
  \nonumber &&\quad \times\|v\|_{L^{\frac{4p}{3}}}+C\|u_{n}\|_{L^{\frac{4p}{3}}}^{p-1}\left(\|u_{n}\|_{L^{\frac{4p}{3}}}^{p-1}+\|u\|_{L^{\frac{4p}{3}}}^{p-1}\right)\|u_{n}-u\|_{L^{\frac{4p}{3}}}
  \|v\|_{L^{\frac{4p}{3}}}\rightarrow0.
\end{eqnarray}
Therefore, $V_2^{\prime}(u_n) \to V_2^{\prime}(u)$ in $\left(L^{\frac{4p}{3}}(\mathbb{R}^{2})\right)'$ and the Gateaux derivative $V_2^{\prime}$ is continuous, that is, $V_{2}\in C^{1}(L^{\frac{4p}{3}}(\mathbb{R}^{2}),\mathbb{R})$. As a immediate consequence, we infer that $V_{2}\in C^{1}(X,\mathbb{R})$ and $\left\langle V_2^{\prime}(u),v \right\rangle=2p B_2(|u|^p,|u|^{p-2}uv)$ for any $u,v\in X$.

It follows immediately that $V_{0}\in C^{1}(X,\mathbb{R})$ and $\left\langle V_0^{\prime}(u),v \right\rangle=2p B_0(|u|^p,|u|^{p-2}uv)$ for any $u,v\in X$. Thus we deduce that the energy functional
\begin{equation}\label{2-31}
  I(u)=\frac{1}{2}\|u\|_{H^1(\mathbb{R}^{2})}^2+\frac{\gamma}{4p\pi}V_{0}(u)-\frac{b}{q}\|u\|_{L^q(\mathbb{R}^{2})}^q
\end{equation}
is $C^1$ on $X$, and
\begin{eqnarray}\label{2-32}
  &&\left\langle I^{\prime}(u),v \right\rangle=\left\langle u,v \right\rangle_{H^{1}(\mathbb{R}^{2})}+\frac{\gamma}{2\pi}B_0(|u|^p,|u|^{p-2}uv)-\int_{\mathbb{R}^{2}}b{|u|}^{q-2}uvdx \\
 \nonumber &=&\int_{\mathbb{R}^{2}}\left(-\Delta u+a(x)u+\frac{\gamma}{2\pi}\left(\log(|\cdot|)*|u|^p\right){|u|}^{p-2}u-b{|u|}^{q-2}u\right)v(x)dx
\end{eqnarray}
for any $u,v\in X$. So $u\in X$ is a critical point of $I$ if and only if it is a weak solution of \eqref{S-N} in the sense that the RHS of \eqref{2-32} vanishes for every $v\in X$.

\emph{Proof of (3):} Let sequence $\{u_n\}\subset H^1(\mathbb{R}^{2})$, $u_n \rightharpoonup u $ in $H^1(\mathbb{R}^{2})$, hence $\{u_{n}\}$ is bounded in $H^1(\mathbb{R}^{2})$ and strongly converge to $u$ in $L_{loc}^r(\mathbb{R}^{2})$ for any $1\leq r<+\infty$. Therefore, for any $R>0$, by Fatou's lemma, we have
\begin{eqnarray}\label{2-37}
  &&\liminf_{n\rightarrow+\infty}\int_{B_R(0)}\int_{B_R(0)} \log(1+|x-y|) |u_n|^p(x)|u_n|^{p}(y) dxdy \\
  \nonumber &\geq& \int_{B_R(0)}\int_{B_R(0)}\log(1+|x-y|)|u|^{p}(x)|u|^{p}(y) dxdy.
\end{eqnarray}
Letting $R\rightarrow+\infty$, by monotone convergence theorem, we derive
\begin{equation}\label{2-33}
  \liminf_{n \to +\infty} V_1(u_n)\geq V_1(u),
\end{equation}
that is, $V_{1}$ is w.l.s.c. on $H^{1}(\mathbb{R}^{2})$.

Now we consider the energy functional
\begin{equation}\label{2-34}
  I(u)=\frac{1}{2}\|u\|_{H^1(\mathbb{R}^{2})}^2+\frac{\gamma}{4p\pi}\big(V_1(u)-V_2(u)\big)-\frac{b}{q}\|u\|_{L^q(\mathbb{R}^{2})}^q.
\end{equation}
Let sequence $\{u_n\}\subset X$ s.t. $u_n \rightharpoonup u$ in $X$ as $n\rightarrow+\infty$. By the compactness of $X \hookrightarrow \hookrightarrow L^s(\mathbb{R}^{2})$ for any $p\leq s< \infty $, one has, up to a subsequence, $u_n \to u$ in $L^{\frac{4p}{3}}(\mathbb{R}^{2})$ as $n\rightarrow+\infty$. Thus $V_2(u_n) \to V_2(u)$ and we get
\begin{equation}\label{2-35}
  \liminf_{n \to +\infty} I(u_n) \geq I(u)
\end{equation}
that is, $I$ is w.l.s.c. on $X$. Since the functional
\begin{equation}\label{2-36}
  u \rightarrow I(u)-\frac{\gamma}{4p\pi}V_1(u)=\frac{1}{2}\|u\|_{H^1(\mathbb{R}^{2})}^2-\frac{\gamma}{4p\pi}V_2(u)- \frac{b}{q}\|u\|_{L^q(\mathbb{R}^{2})}^q
\end{equation}
is continuous on $H^1(\mathbb{R}^{2})$ and $I(u)=\left(I(u)-\frac{\gamma}{4p\pi}V_1(u)\right)+\frac{\gamma}{4p\pi}V_1(u)$, we deduce that $I$ is l.s.c on $H^{1}(\mathbb{R}^{2})$. This completes our proof of Lemma \ref{lemma0}.
\end{proof}

\begin{lem}\label{lemma1}
Assume $p\geq2$, $q\geq2$ and $u\in X$ is a weak solution of the Euler-Lagrange equation \eqref{S-N}, i.e., a critical point of the energy functional $I$. Then $u$ satisfies the following properties: \\
(1)\, The potential function defined by $ w(x):=\frac{1}{2\pi}\int_{\mathbb{R}^{2}}\log(|x-y|)|u|^p(y)dy\in C^{3}(\mathbb{R}^{2})$, moreover, it satisfies
\begin{equation*}
  \Delta w=|u|^p \quad \text{in} \,\, \mathbb{R}^{2}, \qquad w(x)-\frac{1}{2\pi}\log|x|\int_{\mathbb{R}^{2}}|u|^pdx\rightarrow0, \quad \text{as} \,\, |x| \to \infty;
\end{equation*}
(2)\, $u$ decays faster than exponential functions, i.e., for some $A>0$ and $C>0$,
\begin{equation*}
  |u(x)|\leq C e^{-A|x|}, \quad \text{as} \,\, |x| \to \infty;
\end{equation*}
(3)\, $u\in W^{2,r}(\mathbb{R}^2)$ ($\forall\, 1\leq r<+\infty$) is a strong solution of the Euler-Lagrange equation \eqref{S-N}. Moreover, if $a$ is H\"{o}lder continuous, then $u\in C_{loc}^{2,\alpha}(\mathbb{R}^{2})$ for some $\alpha\in(0,1)$, and hence $u\in C^{2}(\mathbb{R}^{2})$.
\end{lem}
\begin{proof}
\emph{Proof of (1):} We first show that the potential function $w(x):=\frac{1}{2\pi}\int_{\mathbb{R}^{2}}\log(|x-y|)|u|^p(y)dy\in L_{loc}^{\infty} $, and hence it is well-defined.

For each $R>1$ and $x \in B_R(0)$,
\begin{align}\label{2-38}
2\pi|w(x)| &=\left|\int_{\mathbb{R}^{2}} \log(|x-y|)|u|^p(y)dy \right|\\
      \nonumber &\leq \int_{\mathbb{R}^{2}} \left|\log(|x-y|)\right| |u|^{p}(y)dy \\
      \nonumber &\leq \int_{|y-x|<2R}\left|\log(|x-y|)\right| |u|^p(y)dy+\int_{|y-x| \geq 2R}\left|\log|x-y|\right| |u|^p(y)dy.
\end{align}
The first term in \eqref{2-38} is a convolution form with the kernel $\left|\log|x|\right|\chi_{B_{2R}(0)}\in L^r(\mathbb{R}^{2})$ for every $1\leq r <\infty $, so
\begin{equation}\label{2-39}
  \int_{|y-x|<2R}\left|\log(|x-y|)\right| |u|^p(y)dy\leq C_{R}\|u\|_{L^{2p}(\mathbb{R}^{2})}^{p}\leq C_R\|u\|_{H^1(\mathbb{R}^{2})}^{p}.
\end{equation}
As to the second term in \eqref{2-38}, we have
\begin{align}\label{2-40}
\int_{|y-x| \geq 2R}\left|\log(|x-y|)\right| |u|^p(y)dy &\leq \int_{|y-x| \geq 2R}\left| \log(R+|y|) \right| |u|^p(y)dy \\
\nonumber &\leq 2\int_{|y-x| \geq 2R}\log(1+|y|)|u|^p(y)dy\\
\nonumber & \leq 2\|u\|_{*}^{p}.
\end{align}
Thus one has $|w(x)|\leq C_R(\|u\|_{H^1(\mathbb{R}^{2})}^p +\|u\|_{*}^{p})$ and $w$ is well-defined.

By the definition of the function $w$, we have, for any $|x|>2$,
\begin{align}\label{2-41}
&2\pi w(x)-\log|x|\int_{\mathbb{R}^{2}}|u|^{p}dx=\int_{\mathbb{R}^{2}} \left(\log(|x-y|)-\log|x|\right)|u|^{p}(y)dy \\
   \nonumber =:&\int_{|y-x|\geq\frac{|x|}{2}}\log\left(\frac{|x-y|}{|x|}\right)|u|^{p}(y)dy+\int_{|y-x|<\frac{|x|}{2}}\left(\log(|x-y|)-\log|x|\right)|u|^{p}(y)dy.
\end{align}
Note that for any $|y-x|\geq\frac{|x|}{2}$, one has
\begin{equation}\label{2-42}
  \log\frac{1}{2}\leq\log\left(\frac{|x-y|}{|x|}\right)\leq\log\left(1+\frac{|y|}{|x|}\right)\leq\log\left(1+|y|\right).
\end{equation}
Since $u\in X$, we have $|u|^{p}\in L^{1}(\mathbb{R}^{2})$ and $\log(1+|x|)|u|^{p}\in L^{1}(\mathbb{R}^{2})$, thus it follows from Lebesgue's dominated convergence theorem that
\begin{equation}\label{2-43}
  \int_{|y-x|\geq\frac{|x|}{2}}\log\left(\frac{|x-y|}{|x|}\right)|u|^{p}(y)dy\rightarrow0, \quad \text{as} \,\,\, |x|\rightarrow+\infty.
\end{equation}
Since one has
\begin{eqnarray}\label{2-45}
 && \int_{|y-x|<1}\left|\log(|x-y|)\right||u|^{p}(y)dy \\
 \nonumber &\leq&\left(\int_{|y-x|<1}\left|\log(|x-y|)\right|^{2}dy\right)^{\frac{1}{2}}
  \left(\int_{|y-x|<1}|u|^{2p}(y)dy\right)^{\frac{1}{2}} \\
 \nonumber &\leq& C\left(\int_{|y-x|<1}|u|^{2p}(y)dy\right)^{\frac{1}{2}}\rightarrow0, \qquad \text{as} \,\,\, |x|\rightarrow+\infty,
\end{eqnarray}
and
\begin{eqnarray}\label{2-46}
 && \int_{1\leq|y-x|<\frac{|x|}{2}}\log(|x-y|)|u|^{p}(y)dy+\int_{|y-x|<\frac{|x|}{2}}\log|x||u|^{p}(y)dy \\
 \nonumber &\leq& \int_{1\leq|y-x|<\frac{|x|}{2}}\log(|y|)|u|^{p}(y)dy+\int_{|y-x|<\frac{|x|}{2}}\log\left(2|y|\right)|u|^{p}(y)dy \\
 \nonumber &\leq& \int_{|y|>\frac{|x|}{2}}\left(2\log\left(1+|y|\right)+\log2\right)|u|^{p}(y)dy\rightarrow0, \qquad \text{as} \,\,\, |x|\rightarrow+\infty,
\end{eqnarray}
thus we arrive at
\begin{eqnarray}\label{2-44}
  &&\left|\int_{|y-x|<\frac{|x|}{2}}\left(\log(|x-y|)-\log|x|\right)|u|^{p}(y)dy\right|\leq\int_{|y-x|<1}\left|\log(|x-y|)\right||u|^{p}(y)dy \\
 \nonumber && +\int_{1\leq|y-x|<\frac{|x|}{2}}\log(|x-y|)|u|^{p}(y)dy+\int_{|y-x|<\frac{|x|}{2}}\log|x||u|^{p}(y)dy\rightarrow0, \qquad \text{as} \,\,\, |x|\rightarrow+\infty.
\end{eqnarray}
Combining \eqref{2-41} and \eqref{2-44} yields that
\begin{equation}\label{2-47}
  w(x)-\frac{1}{2\pi}\log|x|\int_{\mathbb{R}^{2}}|u|^{p}dx\rightarrow0, \qquad \text{as} \,\,\, |x|\rightarrow+\infty,
\end{equation}
and hence the asymptotic property in (1) has been established.

Then, by the Agmon's theorem (see \cite{A}), we know $u$ decays faster than exponential functions, that is, asymptotic property (2) holds. Thus elliptic regularity theory implies that $u\in W^{2,r}(\mathbb{R}^{2})$ for any $r\in[1,+\infty)$, and hence $u$ is a strong solution of \eqref{S-N}. Moreover, from Sobolev embeddings,  we can infer that $u\in C_{loc}^{1,\beta}(\mathbb{R}^{2})$ for any $0\leq\beta<1$. As a consequence, $w\in C_{loc}^{3,\beta}(\mathbb{R}^{2})$ and hence $w\in C^{3}(\mathbb{R}^{2})$ and satisfies $\Delta w=|u|^{p}$ in $\mathbb{R}^{2}$. This proves property (1).

Finally, if $a$ is H\"{o}lder continuous, then $u$ satisfies an equation of the form $-\Delta u=f$ with (locally) H\"{o}lder continuous $f$, thus $u\in C^{2,\alpha}_{loc}(\mathbb{R}^{2})$ for some $\alpha\in(0,1)$ and $u\in C^{2}(\mathbb{R}^{2})$. This proves property (3) and concludes our proof of Lemma \ref{lemma1}.
\end{proof}

\begin{lem}[Pohozaev type identity]\label{Pohozaev}
Assume $u\in X$ is a weak solution to \eqref{S-N} with $a\in C^{2}(\mathbb{R}^{2})$, then the following identity holds:
\begin{eqnarray}\label{pohozaev}
&&P(u):=\frac{\gamma}{4\pi p}\left(\int_{\mathbb{R}^{2}}|u|^{p}dx\right)^{2}+\frac{\gamma}{\pi p}\int_{\mathbb{R}^{2}}\int_{\mathbb{R}^{2}}
\log(|x-y|)|u|^{p}(x)|u|^{p}(y)dxdy \\
\nonumber &&\qquad\quad\,\,\,\, -\frac{2b}{q}\int_{\mathbb{R}^{2}}|u|^{q}dx+\int_{\mathbb{R}^{2}}a(x)|u|^{2}dx \\
\nonumber && \qquad\,\,\,\, =\frac{\gamma}{4\pi p}\|u\|_{L^{p}(\mathbb{R}^{2})}^{2p}+\frac{\gamma}{\pi p}V_{0}(u)-\frac{2b}{q}\|u\|_{L^{q}(\mathbb{R}^{2})}^{q}+\int_{\mathbb{R}^{2}}a(x)|u|^{2}dx=0.
\end{eqnarray}
Consequently, one has
\begin{equation}\label{J}
  J_{k}(u)=k\left\langle I^{\prime}(u), u\right\rangle-P(u)=0, \qquad k=1,2,
\end{equation}
where the auxiliary functionals $J_{k}$ are defined by \eqref{af}.
\end{lem}
\begin{proof}
From Lemma \ref{lemma1}, we know that $u\in C^{2}(\mathbb{R}^{2})$ and $u$ decays faster than exponential functions, i.e., for some $A>0$ and $C>0$,
\begin{equation}\label{p-1}
  |u(x)|\leq C e^{-A|x|}, \quad \text{as} \,\, |x| \to+\infty.
\end{equation}
Moreover, the potential function $w(x):=\frac{1}{2\pi}\int_{\mathbb{R}^{2}}\log\left(|x-y|\right)|u|^{p}(y)dy\in C^{3}(\mathbb{R}^{2})$ and satisfies
\begin{equation}\label{p-2}
  \|\nabla w\|_{L^{\infty}(\mathbb{R}^{2})}\leq C, \qquad w(x)-\frac{1}{2\pi}\log|x|\int_{\mathbb{R}^{2}}|u|^pdx\rightarrow0, \quad \text{as} \,\, |x| \to +\infty.
\end{equation}

Now we define the following two functions on $\mathbb{R}^{2}$:
$$g(u):=b|u|^{q-2}u-a(x)u\in C^{2}(\mathbb{R}^{2}), \quad G(u):=\int_{0}^{u}g(s)ds=\frac{b}{q}|u|^{q}-\frac{a(x)}{2}u^{2}\in C^{2}(\mathbb{R}^{2}),$$
which also decay exponentially as $|x|\rightarrow+\infty$. For any $R>0$, we will multiply the equation \eqref{S-N} by $x\cdot\nabla u$ and integrate by parts on $B_{R}(0)$. Using the following formula (see e.g. Page 136 in \cite{W})
\begin{equation}\label{p-3}
\Delta u(x \cdot \nabla u)=\nabla \cdot\left(\nabla u(x \cdot \nabla u)-x \frac{|D u|^{2}}{2}\right),
\end{equation}
we get
\begin{equation}\label{p-4}\begin{aligned}
\int_{B_{R}(0)}-\Delta u(x \cdot \nabla u) d x &=\int_{B_{R}(0)}-\nabla \cdot\left(\nabla u(x \cdot \nabla u)-x \frac{|D u|^{2}}{2}\right)dx \\
&=\int_{\partial B_{R}(0)} \frac{R}{2}|\nabla u|^{2}d\sigma-\int_{\partial B_{R}(0)} \frac{1}{R}|x \cdot \nabla u|^{2}d\sigma.
\end{aligned}\end{equation}
By direct calculations, one has
\begin{equation}\label{p-5}\begin{aligned}
g(u)(x \cdot \nabla u) &=G^{\prime}(u) x^{i} D_{i} u=D_{i}(G(u)) x^{j} \\
&=x \cdot \nabla(G(u))=\nabla \cdot(x G(u))-G(u) \nabla \cdot x \\
&=\nabla \cdot(x G(u))-2 G(u),
\end{aligned}\end{equation}
and hence the divergence theorem implies
\begin{eqnarray}\label{p-6}
\int_{B_{R}(0)} g(u)(x \cdot \nabla u)dx&=&\int_{B_{R}(0)}\left(\nabla \cdot(x G(u))-2 G(u)\right)dx \\
\nonumber &=&\int_{\partial B_{R}(0)} R G(u) d\sigma-\int_{B_{R}(0)} 2 G(u)dx.
\end{eqnarray}
Moreover, note that
\begin{equation}\label{p-7}
w|u|^{p-2} u(x \cdot \nabla u)=\frac{1}{p}\left[\nabla\cdot\left(w|u|^{p} x\right)-|u|^{p} x \cdot \nabla w-2 w|u|^{p}\right],
\end{equation}
thus we have
\begin{equation}\label{p-8}\begin{aligned}
&\quad \int_{B_{R}(0)} w|u|^{p-2} u(x \cdot \nabla u)dx=\frac{1}{p} \int_{B_{R}(0)}\left[\nabla\cdot\left(w|u|^{p} x\right)-|u|^{p} x \cdot \nabla w-2 w|u|^{p}\right]dx \\
&=\frac{1}{p}\int_{\partial B_{R}(0)} w|u|^{p} R d\sigma-\frac{1}{p} \int_{B_{R}(0)}|u|^{p} x \cdot \nabla wdx-\frac{2}{p} \int_{B_{R}(0)} w|u|^{p}dx.
\end{aligned}\end{equation}

Therefore, multiplying the equation \eqref{S-N} by $x\cdot\nabla u$ and integrating by parts on $B_{R}(0)$, it follows from \eqref{p-4}, \eqref{p-6} and \eqref{p-8} that
\begin{equation}\label{p-9}\begin{aligned}
0 &=\int_{B_{R}(0)}\left[-\Delta u-g(u)+\gamma w|u|^{p-2} u\right](x\cdot\nabla u)dx \\
&=\int_{\partial B_{R}(0)} \frac{R}{2}|\nabla u|^{2} d\sigma-\int_{\partial B_{R}(0)} \frac{1}{R}|x \cdot \nabla u|^{2} d\sigma \\
&\quad -\int_{\partial B_{R}(0)} R G(u) d\sigma+\int_{B_{R}(0)} 2 G(u)dx \\
&\quad +\frac{\gamma}{p} \int_{\partial B_{R}(0)} R w|u|^{p}d\sigma-\frac{\gamma}{p} \int_{B_{R}(0)}|u|^{p} x \cdot \nabla w dx-\frac{2\gamma}{p} \int_{B_{R}(0)} w|u|^{p}dx.
\end{aligned}\end{equation}
As a consequence, we arrive at
\begin{eqnarray}\label{p-10}
&&\int_{B_{R}(0)}\left[\frac{\gamma}{p}|u|^{p} x \cdot \nabla w+\frac{2\gamma}{p} w|u|^{p}-2 G(u)\right]dx \\
\nonumber &=&\int_{\partial B_{R}(0)}\left[\frac{R}{2}|\nabla u|^{2}-\frac{1}{R}|x \cdot \nabla u|^{2}-R G(u)+\frac{\gamma}{p} R w|u|^{p}\right]d\sigma.
\end{eqnarray}

Next, following the idea in \cite{BL}, we will show that the boundary terms in \eqref{p-10} converges to zero along a sequence $R_{k}\rightarrow+\infty$ as $k\rightarrow+\infty$, that is,
\begin{equation}\label{p-11}
  \lim_{k\rightarrow+\infty}R_{k}\int_{\partial B_{R_{k}}(0)}f(x)dx=0,
\end{equation}
where the function
\begin{equation}\label{p-12}
  f(x):=\frac{1}{2}|\nabla u|^{2}+\frac{\gamma}{p}w|u|^{p}-\frac{|x \cdot \nabla u|^{2}}{|x|^{2}}-G(u).
\end{equation}
To this end, it is suffices for us to show that $f\in L^{1}(\mathbb{R}^{2})$. Indeed, from the standard elliptic regularity theory and the exponential decay of $u$, one can deduce that $\nabla u$ also decays exponentially. Consequently, it follows immediately from \eqref{p-1} and \eqref{p-2} that $f\in L^{1}(\mathbb{R}^{2})$, and hence \eqref{p-11} holds.

Furthermore, \eqref{p-1} and \eqref{p-2} also implies that $w|u|^{p}\in L^{1}(\mathbb{R}^{2})$, $G(u)\in L^{1}(\mathbb{R}^{2})$ and $|u|^{p} x \cdot \nabla w\in L^{1}(\mathbb{R}^{2})$, and
\begin{equation}\label{p-13}\begin{aligned}
\int_{\mathbb{R}^{2}}|u|^{p} x \cdot \nabla w dx&=\frac{1}{2\pi}\int_{\mathbb{R}^{2}}\int_{\mathbb{R}^{2}}\frac{|x|^{2}-x \cdot y}{|x-y|^{2}}|u|^{p}(x)|u|^{p}(y) d x d y \\
&=\frac{1}{4\pi}\left(\int_{\mathbb{R}^{2}}|u|^{p}dx\right)^{2}.
\end{aligned}\end{equation}

Taking $R=R_{k}$ in \eqref{p-10} and letting $k\rightarrow+\infty$, \eqref{p-11} and \eqref{p-13} yields that
\begin{eqnarray}\label{p-14}
  && P(u)=\int_{\mathbb{R}^{2}}\left[\frac{\gamma}{p}|u|^{p} x \cdot \nabla w+\frac{2\gamma}{p} w|u|^{p}-2 G(u)\right]dx \\
 \nonumber &&\qquad\,\, =\lim_{k\rightarrow+\infty}\int_{B_{R_{k}}(0)}\left[\frac{\gamma}{p}|u|^{p} x \cdot \nabla w+\frac{2\gamma}{p} w|u|^{p}-2 G(u)\right]dx \\
 \nonumber &&\qquad\,\, =\lim_{k\rightarrow+\infty}R_{k}\int_{\partial B_{R_{k}}(0)}f(x)dx=0.
\end{eqnarray}
Therefore, for every weak solution $u\in X$ to \eqref{S-N}, the auxiliary functionals
\begin{equation}\label{p-15}
  J_{k}(u)=k\left\langle I^{\prime}(u), u\right\rangle-P(u)=0, \qquad k=1,2.
\end{equation}
This finishes our proof of Lemma \ref{Pohozaev}.
\end{proof}

\begin{lem}\label{lemma2}
Assume $p\geq2$ and $q>2$. There exists a $\alpha>0 $ such that, for all $0<\beta\leq\alpha$,
\begin{equation}\label{2-49}
  \inf_{\|u\|_{H^1}=\beta}I(u)>0, \qquad  \inf_{\|u\|_{H^1}=\beta}\left\langle I'(u),u \right\rangle>0.
\end{equation}
\end{lem}
\begin{proof}
From the definition of the energy functional $I$, \eqref{2-12} and Sobolev embeddings, we get
\begin{align}\label{2-48}
I(u)&=\frac{1}{2}\|u\|_{H^1}^2 + \frac{\gamma}{4p\pi}V_0(u) - \frac{b}{q}\|u\|_{L^q}^q\\
\nonumber &\geq \frac{1}{2}\|u\|_{H^1}^2 - \frac{\gamma}{4p\pi}V_2(u) - \frac{b}{q}\|u\|_{L^q}^q\\
\nonumber &\geq \frac{1}{2}\|u\|_{H^1}^2 -C_1 \|u\|_{\frac{4p}{3}}^{2p} - C_{2}\|u\|_{H^1}^q\\
\nonumber &\geq \|u\|_{H^1}^2\left(\frac{1}{2}-C_1\|u\|_{H^1}^{2p-2} - C_{2}\|u\|_{H^1}^{q-2}\right).
\end{align}
Since $p\geq 2$ and $q>2$, we derive that if $\alpha$ is small enough, then for any $0<\beta\leq\alpha$, the first inequality in \eqref{2-49} holds. By direct calculations, we also have
\begin{align}\label{2-50}
\left\langle I'(u),u \right\rangle&=\|u\|_{H^1}^2 + \frac{\gamma}{2\pi}V_0(u)-b\|u\|_{L^q}^q\\
\nonumber &\geq \|u\|_{H^1}^2-\frac{\gamma}{2\pi}V_2(u)-b\|u\|_{L^q}^q\\
\nonumber &\geq \|u\|_{H^1}^2\left(1-C_3\|u\|_{H^1}^{2p-2}-C_{4}\|u\|_{H^1}^{q-2}\right),
\end{align}
which implies that if $\alpha$ is small enough, then for any $0<\beta\leq\alpha$, the second inequality in \eqref{2-49} holds. This finishes our proof of Lemma \ref{lemma2}.
\end{proof}

\begin{lem}\label{lemma11}
Suppose $a(x)=a>0$, $p\geq2$, $q\geq2$. Let $k=1$ \emph{if and only if} $p\geq3$ and $k=2$ \emph{if and only if} $2\leq p<3$. Let $u\in X\setminus\{0\}$ and $u_{t,k}\in X\setminus\{0\}$ be defined by $u_{t,k}(x):=t^{k}u(tx)$ for any $t>0$. Then we have
\begin{equation}\label{eq-a3}
  \lim_{t\rightarrow+\infty}I(u_{t,k})=-\infty.
\end{equation}
\end{lem}
\begin{proof}
Let $u\in X\setminus\{0\}$. Then we have
\begin{equation}\label{63}\begin{aligned}
I(u_{t,k})=& \frac{t^{2k}}{2}\int_{\mathbb{R}^{2}}|\nabla u|^{2} \mathrm{d} x+\frac{t^{2(k-1)}}{2} \int_{\mathbb{R}^{2}}au^{2} \mathrm{d} x-\frac{t^{2(kp-2)}\gamma}{4p\pi}\log t\left(\int_{\mathbb{R}^{2}}|u|^{p} \mathrm{d} x\right)^{2} \\
&+\frac{t^{2(kp-2)}\gamma}{4p\pi}\int_{\mathbb{R}^{2}} \int_{\mathbb{R}^{2}} \log (|x-y|)|u|^{p}(x)|u|^{p}(y) \mathrm{d} x \mathrm{d} y-\frac{t^{kq-2}b}{q}\int_{\mathbb{R}^{2}}|u|^{q}\mathrm{d} x,
\end{aligned}\end{equation}
and hence $I(u_{t,k})\rightarrow-\infty$ as $t\rightarrow+\infty$. This finishes the proof of Lemma \ref{lemma11}.
\end{proof}

\begin{lem}\label{lemma3}
Assume $p\geq2$, $q\geq2p$ and $u\in X\setminus\{0\}$. Then there are only two different possibilities for the function $\varphi_{u}(t)=I(tu)$:\\
(1) \, there exists a unique $t_u\in(0,+\infty)$ such that, $\varphi'_{u}(t)>0$ on $(0,t_u)$ and $\varphi'_{u}(t)<0$ on $(t_u,+\infty)$, moreover, $\varphi_{u}(t)\rightarrow -\infty$ as $t\rightarrow+\infty$;  \\
(2) \, $\varphi'_u(t)>0$ on $(0,+\infty)$ and $\varphi_{u}(t)\rightarrow +\infty$ as $t\rightarrow+\infty$.
\end{lem}
\begin{proof}
The proof of Lemma \ref{lemma3} is obvious. We only need to observe that
\begin{equation}\label{2-51}
  \frac{\varphi_u^{\prime}(t)}{t}=\|u\|_{H^1}^2 +\frac{\gamma}{2\pi}t^{2p-2}V_0(u) -bt^{q-2}\|u\|_{L^q}^q
\end{equation}
with $p\geq 2$, $q\geq 2p$.
\end{proof}
\begin{rem}\label{rem3}
From the proof of Lemma \ref{lemma3}, it is clear that, if $u\in X\setminus\{0\}$ such that $V_{0}(u)<0$, then
\begin{equation}\label{2-52}
  0<\sup_{\mathbb{R}u}I:=\sup_{t\in\mathbb{R}}I(tu)<+\infty.
\end{equation}
\end{rem}

\section{A compactness condition and quantitative deformation Lemma}
In the following, we assume $a\in C(\mathbb{R}^{2})$ and $\mathbb{Z}^{2}$-periodic with $\inf_{\mathbb{R}^{2}}a>0$. For any function $u:\,\mathbb{R}^{2}\rightarrow\mathbb{R}$ and $x\in\mathbb{R}^{2}$, the action of translation is denoted by
\begin{equation}\label{3-1}
  x*u:\, \mathbb{R}^2 \longrightarrow \mathbb{R}, \qquad x*u(y)=u(y-x) \quad \forall\, y\in\mathbb{R}^{2}.
\end{equation}

\subsection{Cerami compactness condition of translation invariant version}
The main result of this sub-section is the following theorem.
\begin{thm}\label{Cerami}
Assume $p\geq2$ and $q\geq2p$. Let sequence $\{u_n\}\subset X$ satisfy
\begin{equation}\label{3-2}
  I(u_n) \longrightarrow d>0, \qquad \|I'(u_n)\|_{X'}\left(1+\|u_n\|_X\right)\longrightarrow 0, \quad \text{as} \,\, n\rightarrow\infty.
\end{equation}
Then, after passing to a subsequence, there exist points $x_n\in \mathbb{Z}^2$ ($n\in\mathbb{N}$), such that
\begin{equation}\label{3-3}
  x_n*u_n \longrightarrow u \quad  \text{strongly in} \quad X, \quad \text{as} \quad  n\rightarrow\infty
\end{equation}
for some critical points $u\in X\setminus\{0\}$ of $I$.
\end{thm}

In order to prove Theorem \ref{Cerami}, we need some useful lemmas.

\begin{lem}\label{lemma4}
Assume $p\geq2$. Let $\{u_n\}$ be a sequence in $L^p(\mathbb{R}^2)$ s.t. $u_n \to u \in L^p(\mathbb{R}^2)\setminus\{0 \}$ a.e. on $\mathbb{R}^2$. Assume $\{v_n\}$ is a bounded sequence in $L^p(\mathbb{R}^{2})$ s.t.
\begin{equation*}
  \sup_{n\in\mathbb{N}} B_1(|u_n|^p,|v_n|^p)<+\infty.
\end{equation*}
Then, there exists $n_0\in\mathbb{Z}$ and $C>0$ s.t. $\|v_n\|_{*}\leq C $ for any $n\geq n_0$. Furthermore, if $B_1(|u_n|^p,|v_n|^p) \to 0$ and $\|v_n\|_{L^p} \to 0$, then $\|v_n\|_{*}\to 0$ as $n\rightarrow\infty$.
\end{lem}
\begin{proof}
Since $u\not\equiv 0$, one has $\lim_{\delta\to 0+}\left|\{|u|>\delta\}\right|=\left|\{|u|>0\}\right|>0$. Thus we can choose $R$ large enough and $\delta$ small enough s.t.
\begin{equation}\label{3-4}
  |E|:=\left|\{|u|>\delta\}\cap B_R(0)\right|>0.
\end{equation}
By applying Egorov's theorem on $E$, we can derive a subset $A \subset E$ with $|A|>0$ and $n_{0}\in\mathbb{N}$ large enough such that, $|u_{n}|\geq\frac{\delta}{2}$ for any $n\geq n_{0}$ on the set $A$.

Note that for any $x \in B_R(0)$ and $y\in \mathbb{R}^2\setminus B_{2R}(0)$,
\begin{equation}\label{3-5}
  1+|x-y|\geq 1+\frac{|y|}{2}\geq\sqrt{1+|y|},
\end{equation}
and hence, we deduce from the definition of $B_{1}$ that, for any $n\geq n_{0}$,
 \begin{align}\label{3-6}
 B_1(|u_n|^p,|v_n|^p)&\geq \int\limits_{R^2 \setminus B_{2R}(0)} \int \limits_{A}\log(1+|x-y|) |u_n|^p(x)|v_n|^p(y)dxdy \\
   \nonumber &\geq \frac{|A|}{2}\left(\frac{\delta}{2}\right)^{p}\int\limits_{R^2 \setminus B_{2R}(0)}\log(1+|y|)|v_n|^p(y)dy \\
   \nonumber &\geq \frac{|A|}{2}\left(\frac{\delta}{2}\right)^{p}\left(\|v_n\|_{*}^{p}-\log(1+2R)\|v_n\|_{L^p}^p\right).
 \end{align}
Thus if $\|v_n\|_{L^p}$ and $ B_1(|u_n|^p,|v_n|^p)$ are bounded, then we get from \eqref{3-6} that $\|v_n\|_{*}$ are bounded for any $n\geq n_{0}$. Moreover, if $B_1(|u_n|^p,|v_n|^p)\to 0$ and $\|v_n\|_{L^p}\to 0$ as $n\rightarrow\infty$, then \eqref{3-6} yields that $\|v_n\|_{*}\to 0$.
\end{proof}

\begin{lem}\label{lemma5}
Assume $p\geq2$. Let $\{\widetilde{u_n}\}$ be a bounded sequence in $X$ such that $\widetilde {u_n}\rightharpoonup u$ weakly in $X$. Then we have
\begin{equation}\label{3-7}
  B_1\left(|\widetilde {u_n}|^p,|\widetilde{u_n}|^{p-2}u(\widetilde{u_n}-u)\right)\longrightarrow 0, \quad \text{as} \,\, n\rightarrow\infty.
\end{equation}
\end{lem}
\begin{proof}
Since $\widetilde {u_n}\rightharpoonup u$ in $X$, by Lemma \ref{lemma0} and the definition of $B_{1}$, we have
\begin{align}\label{3-8}
&\quad \left|B_1\left(|\widetilde {u_n}|^p,|\widetilde{u_n}|^{p-2}u(\widetilde {u_n}-u)\right)\right| \\
\nonumber &\leq \int_{\mathbb{R}^{2}}\int_{\mathbb{R}^{2}} \log(1+|x|)|\widetilde {u_n}|^p(x)  |\widetilde {u_n}|^{p-2}(y) |u|(y) |\widetilde {u_n}-u| (y) dxdy\\
\nonumber &\quad + \int_{\mathbb{R}^{2}}\int_{\mathbb{R}^{2}} \log(1+|y|)|\widetilde {u_n}|^p(x)  |\widetilde {u_n}|^{p-2}(y) |u|(y)|\widetilde {u_n}-u|(y)dxdy\\
\nonumber &\leq C\int_{\mathbb{R}^{2}}|\widetilde {u_n}|^{p-2} (y) |u|(y) |\widetilde {u_n}-u|(y)dy \\
\nonumber &\quad +C\int_{\mathbb{R}^{2}}\log(1+|y|)|\widetilde {u_n}|^{p-2}(y)|u|(y)|\widetilde {u_n}-u|(y)dy\\
\nonumber &\leq o_{n}(1)+C\int_{\mathbb{R}^{2}}\log(1+|y|)|\widetilde {u_n}|^{p-2}(y)|u|(y)|\widetilde{u_n}-u|(y).
\end{align}
For arbitrary $R>0$, we have, as $n\rightarrow\infty$,
\begin{eqnarray}\label{3-9}
  &&\quad \int_{B_R(0)}\log(1+|y|)|\widetilde{u_n}|^{p-2}(y)|u|(y)|\widetilde {u_n}-u|(y)dy \\
\nonumber &&\leq \log(1+R)\|\widetilde{u_n}\|^{p-2}_{L^{p}}\|u\|_{L^{p}}\|\widetilde {u_n}-u\|_{L^{p}}\longrightarrow 0.
\end{eqnarray}
Combining \eqref{3-8} and \eqref{3-9}, we deduce that, for any $R>0$,
\begin{align}\label{3-10}
&\lim_{n\rightarrow\infty}\left|B_1\left(|\widetilde {u_n}|^p,|\widetilde{u_n}|^{p-2}u(\widetilde {u_n}-u)\right)\right| \\
\nonumber \leq& C\sup_{n\in\mathbb{N}}\int_{\mathbb{R}^2\setminus B_R(0)} \log(1+|y|)|\widetilde {u_n}|^{p-2}(y)|u|(y)|\widetilde {u_n}-u|(y)dy \\
\nonumber \leq& C\left(\int_{\mathbb{R}^2\setminus B_R(0)}\log(1+|y|)|u|^p(y)dy\right)^{\frac{1}{p}}=o_{R}(1),
\end{align}
as $ R\to+\infty $. By letting $R\to+\infty$ in \eqref{3-10}, we get the desired conclusion.
\end{proof}

Now we will carry out the proof of Theorem \ref{Cerami} by splitting it into several lemmas. From now on, we will assume the sequence $\{u_n\}\subset X$ satisfies the assumption \eqref{3-2} in Theorem \ref{Cerami}.

\begin{lem}\label{lemma6}
Assume $p\geq2$ and $q\geq2p$. If $\{t_n\}\subset[0,+\infty)$ is a bounded sequence, then $I(t_n u_n)\leq I(u_n)+o_{n}(1)$, as $n \to \infty $. Furthermore, if $t_n \to 0$, then $\liminf_{n \to \infty}I(t_n u_n)\geq 0$.
\end{lem}
\begin{proof}
Recall that
\begin{equation}\label{3-11}
  I(u)=\frac{1}{2}\|u\|_{H^1(\mathbb{R}^{2})}^2 + \frac{\gamma}{4p\pi}V_0(u)-\frac{b}{q}\|u\|_{L^q(\mathbb{R}^{2})}^q,
\end{equation}
\begin{equation}\label{3-12}
  \big\langle I'(u),u\big \rangle =\|u\|_{H^1(\mathbb{R}^{2})}^2+\frac{\gamma}{2\pi}V_0(u)-b\|u\|_{L^q(\mathbb{R}^{2})}^q,
\end{equation}
we have
\begin{equation}\label{3-13}
  I(t_n u_n)-I(u_n)=\frac{t_n^2-1}{2}\|u_n\|_{H^1}^2+\frac{\gamma}{4p\pi}(t_n^{2p}-1)V_0(u_n)-\frac{b}{q}(t_n^q-1)\|u_n\|_{L^q}^q.
\end{equation}
By the assumption \eqref{3-2} and \eqref{3-12}, we have
\begin{eqnarray}\label{3-14}
  \frac{\gamma}{2\pi}V_0(u_n) &=& \big\langle I^{\prime}(u_n),u_n\big \rangle - \|u_n\|_{H^1}^2 + b\|u_n\|_{L^q}^q \\
 \nonumber &=&o_{n}(1)-\|u_n\|_{H^1}^2 + b\|u_n\|_{L^q}^q
\end{eqnarray}
Substituting \eqref{3-14} into \eqref{3-13}, since $p\geq2$, $q\geq2p$ and $(t_{n})_{n}$ is bounded, we get
\begin{align}\label{3-15}
&\quad I(t_nu_n)-I(u_n) \\
\nonumber &= -\bigg( \frac{t_n^{2p}-1}{2p}-\frac{t_n^2-1}{2} \bigg) \|u_n\|_{H^1}^2-b\bigg( \frac{t_n^q-1}{q}-\frac{t_n^{2p}-1}{2p}  \bigg)\|u_n\|_{L^q}^q
+o_{n}(1)\\
\nonumber & \leq o_{n}(1),
\end{align}
this shows $ I(t_n u_n)\leq I(u_n)+o_{n}(1)$. Moreover, if $t_n \to 0$, using the assumption \eqref{3-2} and \eqref{3-14} again, we get
\begin{equation}\label{3-16}
  I(t_nu_n)=\left(\frac{t_n^2}{2}-\frac{t_n^{2p}}{2p}\right)\|u_n\|_{H^1}^2+b\left(\frac{t_n^{2p}}{2p}-\frac{t_n^q}{q}\right)\|u_n\|_{L^q}^q+o_{n}(1)\geq o_{n}(1),
\end{equation}
as $n\rightarrow+\infty$. This finishes the proof of Lemma \ref{lemma6}.
\end{proof}

\begin{lem}\label{lemma7}
Assume $p\geq2$ and $q\geq2p$. Suppose sequence $\{u_n\}$ satisfies the assumption \eqref{3-2}, then $\{u_n\}$ is bounded in $H^1(\mathbb{R}^{2})$.
\end{lem}
\begin{proof}
We will prove Lemma \ref{lemma7} by contradiction arguments. Suppose on the contrary that, after passing to a subsequence, then we have $\|u_{n}\|_{H^1(\mathbb{R}^{2})}\to+\infty$.

Define $v_n:=\frac{u_n}{\|u_n\|_{H^{1}}}$, then we have $\|v_n\|_{H^1}=1$. In the following, we will carry out the proof by three steps.

\emph{Step 1}. We will show the following property:
\begin{equation}\label{3-17}
  \inf_{n\in\mathbb{N}}\sup_{x\in \mathbb{Z}^2} \int\limits_{B_2(x)} v_n^2(y)dy>0.
\end{equation}
Suppose it is not true, then by Lemma \ref{vanish}, after passing to a subsequence, $v_n \to 0$ in $L^s(\mathbb{R}^{2})$ for every $2<s<+\infty$. Thus $|V_2(tv_n)|\leq C\|tv_n\|_{L^{\frac{4p}{3}}}^{2p}\to 0$ and $\|tv_n\|_{L^q}^{q} \to 0 $ for each fixed $t\geq0$. Hence
\begin{align}\label{3-18}
 I(tv_n) &=\frac{t^2}{2}\|v_n\|_{H^1}^2 + \frac{\gamma}{4p\pi}\bigg( V_1(tv_n)-V_2(tv_n) \bigg) - \frac{b}{q}\|tv_n\|_{L^q}^q \\
 \nonumber & \geq \frac{t^2}{2}+o_{n}(1)
\end{align}
for each fixed $t\geq0$. But, on the other hand, by Lemma \ref{lemma6}, one has, for every fixed $t\geq 0$,
\begin{equation}\label{3-19}
  I(tv_n)=I\left(\left(\frac{t}{\|u_n\|_{H^1}}\right)u_n\right)\leq I(u_n)+o_{n}(1)=d+o_{n}(1),
\end{equation}
which is absurd if we take $t$ large enough (say, $t>\sqrt{2d}$). So we have obtained the property \eqref{3-17}.

\emph{Step 2}. By \eqref{3-17}, there exists a sequence $\{x_n\}\subset\mathbb{Z}^2$ s.t.
\begin{equation}\label{3-20}
  \liminf_{n\rightarrow\infty}\int_{B_2(x_n)}v_n^2(y)dy>0,
\end{equation}
and hence $w_n:=(-x_n)*v_n$ ($n\in\mathbb{N}$) satisfy
\begin{equation}\label{3-21}
  \liminf_{n\rightarrow\infty}\int_{B_2(0)} w_n^2(y)dy>0.
\end{equation}
Note that $\|w_n\|_{H^1(\mathbb{R}^{2})}=\|v_n\|_{H^1(\mathbb{R}^{2})}=1$, one has $w_n \rightharpoonup w$ for some $w\in H^1(\mathbb{R}^{2})$. Then, Rellich's compact embedding theorem and \eqref{3-21} yield that $w\not\equiv 0$. By passing to a subsequence, we can also assume $w_n \longrightarrow w$ a.e. on $\mathbb{R}^2$.

By the translation invariant in $\mathbb{Z}^2$ of $I$ and taking $t =1$ in \eqref{3-19}, we have
\begin{equation}\label{3-22}
  I(w_n)=I(v_n)\leq d+o_{n}(1).
\end{equation}
Hence, by Sobolev embeddings, we arrive at
\begin{align}\label{3-23}
\frac{\gamma}{4p\pi}V_1(w_n)&=I(w_n)-\frac{1}{2}\|w_n\|_{H^1}^2+\frac{\gamma}{4p\pi}V_2(w_n)+\frac{b}{q}\|w_n\|_{L^q}^q  \\
                    \nonumber &\leq d+o_{n}(1)-\frac{1}{2}+\frac{\gamma}{4p\pi}V_2(w_n)+\frac{b}{q}\|w_n\|_{L^q}^q \\
                    \nonumber &\leq C.
\end{align}

\emph{Step 3}. In this step, we will distinguish two different cases separately.

\emph{Case 1}: $b>0$ and $q>2p$. In this case, by Fatou's lemma, after passing to a subsequence, we have
\begin{align}\label{3-24}
 I(tv_n)=I(tw_n)&=\frac{t^2}{2}\|w_n\|_{H^1}^2+\frac{\gamma t^{2p}}{4p\pi}\bigg(V_1(w_n)-V_2(w_n)\bigg)-\frac{bt^q}{q}\|w_n\|_{L^q}^q \\
 \nonumber &\leq\frac{t^2}{2}+C t^{2p}-\frac{bt^q}{q}\|w\|_{L^q}^q+o_{n}(1).
\end{align}
Thus there exist $t_0$ and $n_{0}$ sufficiently large such that, for any $n\geq n_{0}$,
\begin{equation}\label{3-25}
  I(t_0 v_n)=I(t_0 w_n)\leq -1.
\end{equation}
However, since $\frac{t_0}{\|u_n\|_{H^1}} \to 0$, \eqref{3-25} contradicts with Lemma \ref{lemma6}.

\emph{Case 2}: $b=0$ or $q=2p$. In this case, we have
\begin{align}\label{3-26}
I(tv_n)=I(tw_n) &=\frac{t^2}{2}\|w_n\|_{H^1}^2+\frac{\gamma t^{2p}}{4p\pi}V_0(w_n)-\frac{bt^q}{q}\|w_n\|_{L^q}^q \\
               \nonumber &:=\frac{t^2}{2}+\frac{\gamma t^{2p}}{4p\pi}\rho_n
\end{align}
with $\rho_n:=V_0(w_n)-\frac{4p\pi b}{\gamma q}\|w_n\|_{L^q}^q$. We will show that
\begin{equation}\label{3-27}
  \rho:=\limsup_{n\to\infty}\rho_n<0.
\end{equation}
If not, then $\limsup_{n\to\infty}\rho_n\geq0$. After passing to a subsequence, we have, if we choose $t_0$ large enough (say, $ t_0\geq\sqrt{4d}$), then
\begin{align}\label{3-28}
I(t_0 v_n)=I(t_0 w_n)&=\frac{t_0^2}{2}+\frac{\gamma t_0^{2p}}{4p\pi}\rho_n\geq\frac{t_0^2}{2}+o_{n}(1)\\
               \nonumber &\geq 2d+o_{n}(1).
\end{align}
However, by the assumption \eqref{3-2} and Lemma \ref{lemma6}, one has
\begin{align}\label{3-27}
 d+o(1)=I(u_n) &\geq I\left(\left(\frac{t_0}{\|u_n\|_{H^1}}\right)u_n\right)+o_{n}(1)\\
   \nonumber &=I(t_0 v_n)+o_{n}(1)\geq 2d +o_{n}(1),
\end{align}
which is a contradiction. Thus $\rho=\limsup_{n\to\infty}\rho_n<0$, and hence there exists a $n_{0}\in\mathbb{N}$ large enough such that $\rho_n\leq -{\epsilon}_0<0$ for any $n\geq n_{0}$. Then, we have, for any $n\geq n_{0}$,
\begin{equation}\label{3-28}
  I(tv_n)=\frac{t^2}{2}+\frac{\gamma t^{2p}}{4p\pi}\rho_n\leq\frac{t^2}{2}-\frac{\gamma t^{2p}}{4p\pi}{\epsilon}_0.
\end{equation}
Thus there exists $t_0$ sufficiently large such that, for any $n\geq n_{0}$, $I(t_0 v_n)\leq -1$, which contradicts with Lemma \ref{lemma6} again. Therefore, $\{u_n\}$ is bounded in $H^1(\mathbb{R}^{2})$.
\end{proof}

\begin{lem}\label{lemma8}
Assume $p\geq2$ and $q>2$. Suppose $\{u_n\}$ satisfies the assumption \eqref{3-2}. Then we have
\begin{equation}\label{3-29}
  \liminf_{n\rightarrow\infty}\sup_{x\in\mathbb{Z}^2}\int_{B_2(x)}u_n^2(y)dy>0.
\end{equation}
\end{lem}
\begin{proof}
Suppose it is false. Then by Lemma \ref{vanish}, after passing to a subsequence, we have $u_n \to 0$ in $L^s(\mathbb{R}^{2})$ for every $s>2$. Since
\begin{equation}\label{3-30}
  \big\langle I^{\prime}(u_n),u_n\big\rangle=\|u_n\|_{H^1}^2+\frac{\gamma}{2\pi}\left(V_1(u_n)-V_2(u_n)\right)-b\|u_n\|_{L^q}^q,
\end{equation}
thus we get
\begin{equation}\label{3-31}
  \|u_n\|_{H^1}^2+\frac{\gamma}{2\pi}V_1(u_n)= \big\langle I^{\prime}(u_n),u_n\big\rangle+\frac{\gamma}{2\pi}V_2(u_n)+b\|u_n\|_{L^q}^q \longrightarrow 0,
\end{equation}
as $n\rightarrow\infty$. It follows immediately that $\|u_n\|_{H^1}^2\to 0$, $V_1(u_n)\to 0$, and hence
\begin{equation}\label{3-32}
  I(u_n)=\frac{1}{2}\|u_n\|_{H^1}^2+\frac{\gamma}{4p\pi}\bigg(V_1(u_n)-V_2(u_n)\bigg)-\frac{b}{q}\|u_n\|_{L^q}^q\longrightarrow 0,
\end{equation}
which is absurd, because $I(u_n)\to d>0$. This finishes our proof of Lemma \ref{lemma8}.
\end{proof}

By Lemma \ref{lemma8}, there exists a sequence $\{x_n\}\subset\mathbb{Z}^2$ s.t.
\begin{equation}\label{3-33}
  \liminf_{n\rightarrow\infty}\int_{B_2(x_n)}u_n^2(y)dy>0,
\end{equation}
and hence $\widetilde{u_n}:=(-x_n)*u_n\in X$ ($n\in\mathbb{N}$) satisfy
\begin{equation}\label{3-34}
  \liminf_{n\rightarrow\infty}\int_{B_2(0)}\widetilde{u_n}^2(y)dy>0.
\end{equation}
Note that $\|\widetilde{u_n}\|_{H^{1}}=\|u_{n}\|_{H^{1}}$ are bounded, one has $\widetilde{u_n}\rightharpoonup u$ for some $u\in H^1(\mathbb{R}^{2})$. Then, Rellich's compact embedding theorem and \eqref{3-34} yield that $u\not\equiv 0$. By passing to a subsequence, we can also assume $\widetilde{u_n} \longrightarrow u$ a.e. on $\mathbb{R}^2$.

Now we are ready to complete the proof of Theorem \ref{Cerami}.
\begin{proof}[Proof of Theorem \ref{Cerami} (completed)]
We will complete the proof of Theorem \ref{Cerami} by the following three steps.

\emph{Step 1}: We are to show that $\{\widetilde{u_n}\}$ is bounded in $X$. Indeed, by assumption \eqref{3-2}, we have
\begin{align}\label{3-35}
\frac{\gamma}{2\pi}V_1(\widetilde {u_n})&=\big\langle I^{\prime}(\widetilde {u_n}) ,\widetilde {u_n}\big \rangle-\|\widetilde {u_n}\|_{H^1}^2
+\frac{\gamma}{2\pi}V_2(\widetilde {u_n})+b\|\widetilde {u_n}\|_{L^q}^q \\
\nonumber &=\big\langle I^{\prime}(u_n) ,u_n \big \rangle-\|u_n\|_{H^1}^2+\frac{\gamma}{2\pi}V_2(u_n)+b\|u_n\|_{L^q}^q \\
\nonumber &=o_{n}(1)-\|u_n\|_{H^1}^2+\frac{\gamma}{2\pi}V_2(u_n)+b\|u_n\|_{L^q}^q\leq C.
\end{align}
By Lemma \ref{lemma4}, we get from \eqref{3-35} that $\|\widetilde {u_n}\|_*$ are bounded, and hence $\{\widetilde {u_n}\}$ is bounded in $X$. We may assume that, after passing to a subsequence if necessary, $\widetilde {u_n}\longrightarrow u$ weakly in $X$, so that $u\in X$. By the compact embeddings of $X \hookrightarrow\hookrightarrow L^s(\mathbb{R}^{2})$ for all $p\leq s<+\infty$, we have $\widetilde {u_n}\longrightarrow u$ strongly in $L^s(\mathbb{R}^{2})$ for all $p\leq s<+\infty$.

\emph{Step 2}: Our goal is to prove $\widetilde {u_n} \longrightarrow u$ strongly in $X$. First, we will show that
\begin{equation}\label{3-36}
  \big\langle I^{\prime}(\widetilde{u_n}),\widetilde{u_n}-u\big\rangle\longrightarrow 0, \quad \text{as} \,\, n \to \infty.
\end{equation}
In fact, by the $\mathbb{Z}^2$-translation invariance, we have
\begin{equation}\label{3-37}
   \bigg|\big\langle I^{\prime}(\widetilde{u_n}),\widetilde{u_n}-u\big\rangle\bigg|=\bigg|\big\langle I^{\prime}(u_n),u_n-x_{n}*u\big\rangle\bigg|
   \leq\|I^{\prime}(u_n)\|_{X'}\bigg(\|u_n\|_{X}+\|x_{n}*u\|_{X}\bigg).
\end{equation}
For the last term in \eqref{3-37}, we have
\begin{align}\label{3-38}
\|x_{n}*u\|_{*}&=\left(\int_{\mathbb{R}^{2}}\log(1+|x|)|u|^p(x-x_n)dx\right)^{\frac{1}{p}}=\left(\int_{\mathbb{R}^{2}}\log(1+|x+x_n|)|u|^{p}dx\right)^{\frac{1}{p}} \\
\nonumber &\leq \left(\int_{\mathbb{R}^{2}}\log(1+|x|)|u|^{p}dx\right)^{\frac{1}{p}}+\left(\int_{\mathbb{R}^{2}}\log(1+|x_n|)|u|^pdx\right)^{\frac{1}{p}} \\
\nonumber &\leq \|u\|_{*}+\left(\log(1+|x_n|)\right)^{\frac{1}{p}}\|u\|_{L^{p}(\mathbb{R}^{2})},
\end{align}
and hence
\begin{equation}\label{3-39}
  \|x_{n}*u\|_{X}\leq C_1+C_2\left(\log(1+|x_n|)\right)^{\frac{1}{p}}.
\end{equation}
On the other hand, for any $n\in\mathbb{N}$ such that $|x_{n}|\geq4$, one can infer from \eqref{3-34} that
\begin{align}\label{3-40}
\|u_n\|_{*} &=\left(\int_{\mathbb{R}^{2}}\log(1+|x-x_n|)|\widetilde{u_n}|^{p}dx\right)^{\frac{1}{p}} \\
\nonumber &\geq \left(\int_{B_{2}(0)}\log\left(1+\frac{|x_n|}{2}\right)|\widetilde{u_n}|^{p}dx\right)^{\frac{1}{p}}\geq \left(\int_{B_{2}(0)}\log\left(\sqrt{1+|x_n|}\right)|\widetilde{u_n}|^{p}dx\right)^{\frac{1}{p}} \\
\nonumber &\geq C_{3}\left(\log(1+|x_n|)\right)^{\frac{1}{p}},
\end{align}
and it follows immediately that
\begin{equation}\label{3-41}
  \|u_n\|_{X}\geq C_{4}+C_{3}\left(\log(1+|x_n|)\right)^{\frac{1}{p}}.
\end{equation}
Therefore, by combining \eqref{3-39} and \eqref{3-41}, we have
\begin{equation}\label{3-42}
  \|x_{n}*u\|_X\leq C\|u_n\|_X.
\end{equation}
Then, by the assumption \eqref{3-2} and \eqref{3-37}, we have
\begin{equation}\label{3-43}
  \big\langle I'(\widetilde {u_n}),\widetilde {u_n}-u \big\rangle\longrightarrow 0, \quad \text{as} \,\, n\rightarrow\infty,
\end{equation}
this proves the property \eqref{3-36}.

As a consequence, we get
\begin{align}\label{3-44}
o_{n}(1)&=\big\langle I^{\prime}(\widetilde {u_n}) ,\widetilde {u_n}-u\big\rangle=\big\langle I^{\prime}(\widetilde {u_n}) ,\widetilde {u_n} \big \rangle -\big\langle I^{\prime}(\widetilde {u_n}) , u  \big \rangle\\
\nonumber &= \|\widetilde {u_n}\|_{H^1}^2 -\|u\|_{H^1}^2 +o_{n}(1)+\frac{\gamma}{4p\pi}\bigg\langle V'_{0}(\widetilde {u_n}),\widetilde {u_n}-u \bigg\rangle-b\int |\widetilde {u_n}|^{q-2} \widetilde {u_n} (\widetilde {u_n}-u).
\end{align}
By compact embeddings $X\hookrightarrow\hookrightarrow L^s(\mathbb{R}^{2})$ for any $p\leq s<\infty$ and Hardy-Littlewood-Sobolev inequalities, we have
\begin{equation}\label{3-45}
  b\int |\widetilde {u_n}|^{p-2} \widetilde {u_n} (\widetilde {u_n}-u) \longrightarrow 0, \qquad
 \frac{\gamma}{4p\pi}\bigg\langle V_2^{\prime}(\widetilde {u_n}),\widetilde {u_n}-u \bigg\rangle \longrightarrow 0.
\end{equation}
At the same time, one can also infer from Lemma \ref{lemma5} that
\begin{align}\label{3-46}
&\quad \frac{\gamma}{4p\pi}\bigg\langle V'_{1}(\widetilde {u_n}),\widetilde {u_n}-u \bigg\rangle
=\frac{\gamma}{2\pi}B_1\bigg( |\widetilde {u_n}|^{p},|\widetilde {u_n}|^{p-2}\widetilde {u_n}(\widetilde {u_n}-u)\bigg) \\
\nonumber &=\frac{\gamma}{2\pi}B_1 \bigg(|\widetilde {u_n}|^{p},|\widetilde {u_n}|^{p-2}\left((\widetilde {u_n}-u)^2+u (\widetilde {u_n}-u)\right)\bigg)\\
\nonumber &=\frac{\gamma}{2\pi}B_1\bigg(|\widetilde {u_n}|^{p},|\widetilde {u_n}|^{p-2}|\widetilde{u_n}-u|^2\bigg)
+\frac{\gamma}{2\pi}B_1\bigg(|\widetilde {u_n}|^{p},|\widetilde {u_n}|^{p-2}u (\widetilde{u_n}-u)\bigg) \\
\nonumber &=\frac{\gamma}{2\pi}B_1\bigg(|\widetilde {u_n}|^{p},|\widetilde {u_n}|^{p-2}|\widetilde{u_n}-u|^2\bigg)+o_{n}(1).
\end{align}
Define $|v_n|^p:=|\widetilde{u_n}|^{p-2}|\widetilde{u_n}-u|^2$ for every $n\in\mathbb{N}$, then
\begin{equation}\label{3-47}
  B_1 \bigg( |\widetilde {u_n}|^{p},|\widetilde {u_n}|^{p-2}|\widetilde{u_n}-u|^2\bigg)=B_1\bigg(|\widetilde {u_n}|^{p}, |v_n|^{p}\bigg)\geq 0,
\end{equation}
and hence, we get from \eqref{3-44}, \eqref{3-45}, \eqref{3-46} and \eqref{3-47} that
\begin{align}\label{3-48}
o_{n}(1) &=\big\langle I^{\prime}(\widetilde {u_n}) ,\widetilde {u_n}-u\big \rangle\\
\nonumber &=o_{n}(1)+\|\widetilde {u_n}\|_{H^1}^2 -\|u\|_{H^1}^2  + B_1 \big( |\widetilde {u_n}|^{p}, |v_n|^{p} \big)\\
\nonumber &\geq o_{n}(1)+\|\widetilde {u_n}\|_{H^1}^2 -\|u\|_{H^1}^2,
\end{align}
which implies $\|\widetilde {u_n}\|_{H^1}^2\to\|u\|_{H^1}^2$ and $B_1\big( |\widetilde {u_n}|^{p},|v_n|^{p}\big)\to 0$. Thus we derive $\|\widetilde {u_n}-u\|_{H^{1}(\mathbb{R}^{2})}\rightarrow 0$ as $n\rightarrow\infty$. Again, by the compact embedding $X\hookrightarrow\hookrightarrow L^p(\mathbb{R}^{2})$, we get $\|v_n\|_{L^{p}(\mathbb{R}^{2})}\to 0$. Thus by Lemma \ref{lemma4}, we arrive at $\|v_n\|_{*}\to 0$.

Now, we apply Lemma \ref{epsilon} (take $\epsilon=\frac{1}{2}$ therein) and obtain
\begin{align}\label{3-49}
&\quad o_{n}(1)=\|v_n\|_{*}^{p}=\int_{\mathbb{R}^{2}}\log(1+|x|)|\widetilde {u_n}|^{p-2}|\widetilde{u_n}-u|^2 dx \\
    \nonumber &=\int_{\mathbb{R}^{2}}\log(1+|x|)\bigg(|\widetilde{u_n}-u|^{p-2}+|\widetilde {u_n}|^{p-2}-|\widetilde{u_n}-u|^{p-2}\bigg)|\widetilde{u_n}-u|^2 dx \\
    \nonumber &=\int_{\mathbb{R}^{2}}\log(1+|x|)|\widetilde {u_n}-u|^{p}dx+\int_{\mathbb{R}^{2}}\log(1+|x|)\bigg(|\widetilde{u_n}|^{p-2}-
    |\widetilde{u_n}-u|^{p-2}\bigg)|\widetilde{u_n}-u|^2 dx \\
    \nonumber &\geq\frac{1}{2}\int_{\mathbb{R}^{2}}\log(1+|x|)|\widetilde {u_n}-u|^{p}dx-C\int_{\mathbb{R}^{2}}\log(1+|x|)|u|^{p-2}|\widetilde{u_n}-u|^2 dx.
\end{align}
By method similar to the proof of Lemma \ref{lemma5}, we can also deduce
\begin{equation}\label{3-50}
  B_{1}(|\widetilde{u_n}|^{p},|u|^{p-2}|\widetilde{u_n}-u|^2)\rightarrow0, \qquad \text{as} \,\, n\rightarrow\infty.
\end{equation}
As a consequence, if we define $|\widetilde{v_n}|^p:=|u|^{p-2}|\widetilde{u_n}-u|^2$ for every $n\in\mathbb{N}$, then Lemma \ref{lemma4} implies
\begin{equation}\label{3-51}
  \|\widetilde{v_n}\|_{*}=\left(\int_{\mathbb{R}^{2}}\log(1+|x|)|u|^{p-2}|\widetilde{u_n}-u|^2 dx\right)^{\frac{1}{p}}\rightarrow0, \qquad \text{as} \,\, n\rightarrow\infty.
\end{equation}
Combining \eqref{3-49} and \eqref{3-51} yields that $\|\widetilde {u_n}-u \|_{*}\to 0$. Combining with the $H^1(\mathbb{R}^{2})$ strong convergence, we finally derive the desired strong convergence in $X$:
\begin{equation}\label{3-52}
  \|\widetilde {u_n}-u\|_X=\|(-x_{n})*u-u\|_X\rightarrow 0, \qquad \text{as} \,\, n\rightarrow\infty.
\end{equation}

\emph{Step 3}: We will prove $u\in X\setminus\{0\}$ is a critical point of $I$, i.e., $I^{\prime}(u)=0$.

For any given $v \in X$, similar to the proof of \eqref{3-39}, we can deduce that
\begin{equation}\label{3-53}
  \|x_{n}*v\|_{X}\leq C_{5}+C_{6}\log(1+|x_n|),
\end{equation}
and hence, by \eqref{3-41}, we have
\begin{equation}\label{3-54}
  \|x_{n}*v\|_{X}\leq C\|u_n\|_{X}.
\end{equation}
By assumption \eqref{3-2} and \eqref{3-54}, we get
\begin{align}\label{3-55}
\bigg|\big\langle I'(\widetilde {u_n}),v\big\rangle\bigg|&=\big|\langle I'(u_n),x_{n}*v\rangle\big|\leq\|I'(u_n)\|_{X'}\|x_{n}*v\|_{X} \\
\nonumber & \leq C\|I'(u_n)\|_{X'}\|u_n\|_{X}\longrightarrow 0.
\end{align}
Note that $\big\langle I'(u),v\big\rangle=\lim_{n\to\infty}\big\langle I'(\widetilde{u_n}),v\big\rangle$, thus we have $ \big \langle I^{\prime}(u),v  \big \rangle =0 $. This concludes our proof of Theorem \ref{Cerami}.
\end{proof}

\subsection{Quantitative Deformation Lemma}
In this subsection, we will construct the nonincreasing flow based on the Cerami compactness condition in Theorem \ref{Cerami}.

First, we introduce the some standard notations. For any given $c>0$, we denote the set of critical points at the energy level $c$ by
\begin{equation}\label{3+1}
  K_c:=\{u\in X: \, I'(u)=0, \, I(u)=c\}.
\end{equation}
For any $\rho>0$, we define the set
\begin{equation}\label{3+2}
  A_{c,\rho}:=\{u\in X: \, \|u-v\|_{H^1}\leq\rho \,\,\, \text{for some} \,\, v\in K_c\}.
\end{equation}
For any $c\geq 0$, the sub-level set is denoted by
\begin{equation}\label{3+3}
  I^{c}:=\{u\in X: \, I(u)\leq c\},
\end{equation}
and, in particular,
\begin{equation}\label{3+4}
  D:=I^{0}=\{u\in X: \, I(u)\leq 0\}.
\end{equation}

We have the following lemma on important properties of the set $A_{c,\rho}$.
\begin{lem}\label{lemma9}
Assume $p\geq2$ and $q\geq2p$. Then $A_{c,\rho}$ satisfies the following two properties: \\
1. \, For any $c>0$ and $\rho>0$, $A_{c,\rho}$ is symmetric w.r.t. the reflection $u\mapsto -u$ and invariant under $\mathbb{Z}^2$-translations; \\
2. \, For every given $c>0$, there exists a $\rho_{0}(c)>0$ such that, $A_{c,\rho}\cap D=\emptyset$, $\forall \, \rho\in(0,\rho_{0})$.
\end{lem}
\begin{proof}
\emph{Proof of 1}: This is an immediate consequence from the symmetry w.r.t. the reflection $u\mapsto -u$ and $\mathbb{Z}^2$-translations invariance of the energy functional $I$.

\emph{Proof of 2}: We will prove it by contradiction arguments. Assume that there exist sequences $\rho_n\to 0$ and $\{u_{n}\}\subset X$ s.t. $u_n\in A_{c,\rho_{n}}\cap D$ for each $n\in\mathbb{N}$. Taking $v_n\in K_c$ such that $\|u_n-v_n\|_{H^{1}}\leq\rho_n\to 0$. By Theorem \ref{Cerami}, after passing to a subsequence, there exists a sequence $\{x_n\}\subset \mathbb{Z}^2 $ s.t. $x_n*v_n \to v \in K_c$ strongly in $X$. By the $\mathbb{Z}^2$-translations invariance, $x_n*u_n \in A_{c,\rho_{n}}\cap D$, and we have
\begin{align}\label{3+5}
\| x_n*u_n -v\|_{H^{1}} &\leq \|x_n*u_n -x_n*v_n \|_{H^{1}} + \| x_v*v_n -v \|_{H^{1}} \\
               \nonumber &= \| u_n -v_n \|_{H^{1}}+ \|x_n*v_n -v\|_{H^{1}}\\
               \nonumber &\leq \rho_n+o_{n}(1)\longrightarrow 0,
\end{align}
as $n\rightarrow\infty$. However, since $x_n*u_n \in D$ and $I$ is l.s.c on $H^1(\mathbb{R}^{2})$, we get
\begin{equation}\label{3+6}
  0<c=I(v)\leq\liminf_{n\to\infty}I(x_n*u_n)\leq 0,
\end{equation}
which is a contradiction. Hence we infer that, for each $c>0$, there is a $\rho_0(c)>0$ s.t., for all $0<\rho<\rho_{0}(c)$, $A_{c,\rho}\cap D=\emptyset$.
\end{proof}

\begin{lem}[Quantitative deformation lemma]\label{deformation}
Assume $p\geq2$ and $q\geq2p$. Let $c>0$. Then, for any $\rho\in(0,\rho_{0}(c))$, there exists a $\epsilon =\epsilon(c,\rho)>0$ and an odd continuous map
$\varphi: \, I^{c+\epsilon}\setminus A_{c,\rho} \longrightarrow I^{c-\epsilon}$ such that $\varphi|_{D}=id_{D}$.
\end{lem}
\begin{proof}
We fix $\rho\in(0,\rho_{0}(c))$ ($\rho_{0}(c)$ is the same as in Lemma \ref{lemma9}) arbitrarily, and consider the sets
\begin{equation}\label{3+7}
  S:=X \setminus A_{c,\rho} \quad \text{and} \quad \widetilde{S_{\delta}}:= \{u\in X: \, \|u-v\|_{H^1}\leq\delta \,\,\, \text{for some} \,\, v\in S\}
\end{equation}
for $\delta>0$. We are to show that, for $\delta>0$ small enough,
\begin{equation}\label{3+8}
 \|I'(u)\|_{X'}(1+\|u\|_X )\geq 8\delta, \quad \forall u \in \widetilde{S_{2\delta}}\cap I^{-1}([c-2\delta^2 ,c+2\delta^2]).
\end{equation}
Suppose on contrary that there exist sequences $\delta_{n}>0$ ($n\in\mathbb{N}$) and $\{u_{n}\}\in\widetilde{S_{2\delta_{n}}}$ such that $\delta_{n}\rightarrow0$ as $n\rightarrow\infty$ and
\begin{equation}\label{3+9}
  \|I'(u_{n})\|_{X'}(1+\|u_{n}\|_X)<8\delta_{n}, \qquad c-2\delta_{n}^2\leq I(u_{n})\leq c+2\delta_{n}^2
\end{equation}
for every $n\in\mathbb{N}$. By Theorem \ref{Cerami}, after passing to a subsequence, there exist points $x_{n}\in\mathbb{Z}^{2}$ ($n\in\mathbb{N}$) and $u\in K_{c}$ such that
\begin{equation}\label{3+10}
  \|x_{n}*u_{n}-u\|_{X}\rightarrow 0, \qquad \text{as} \,\, n\rightarrow\infty.
\end{equation}
By definition \eqref{3+7}, there is a sequence $\{v_{n}\}\subset S$ such that
\begin{equation}\label{3+11}
  \|u_{n}-v_{n}\|_{H^{1}(\mathbb{R}^{2})}\leq 2\delta_{n}, \qquad \forall \, n\in\mathbb{N}.
\end{equation}
Then it follows that
\begin{eqnarray}\label{3+12}
\left\|x_{n}*v_{n}-u\right\|_{H^{1}}&\leq&\left\|x_{n} * v_{n}-x_{n} * u_{n}\right\|_{H^{1}}+\left\|x_{n} * u_{n}-u\right\|_{H^{1}} \\
\nonumber &\leq&\left\|v_{n}-u_{n}\right\|_{H^{1}}+\left\|x_{n}*u_{n}-u\right\|_{H^{1}} \rightarrow 0, \qquad \text { as } n \rightarrow \infty.
\end{eqnarray}
Since, by $\mathbb{Z}^{2}$-translations invariance, $x_{n}*v_{n}\in S$, \eqref{3+12} implies that
\begin{equation}\label{3+13}
  \|u-v\|_{H^{1}(\mathbb{R}^{2})}\geq\rho, \qquad \forall \, v\in K_{c},
\end{equation}
which is absurd since $u\in K_{c}$. Thus \eqref{3+8} holds for $\delta>0$ sufficiently small. It is clear that \eqref{3+8} still holds if $\widetilde{S_{2\delta}}$ is replaced by the subset
\begin{equation}\label{3+14}
  S_{2\delta}:= \{u\in X: \, \|u-v\|_{X}\leq2\delta \,\,\, \text{for some} \,\, v\in S\}.
\end{equation}

Now, we may fix a $\delta>0$ small enough s.t. \eqref{3+8} holds and
\begin{equation}\label{3+15}
  \varepsilon:=\delta^2<\frac{c}{2}.
\end{equation}
Since $I(u)$ is even w.r.t. $u$, by Lemma 2.6 in \cite{LW}, there exists a continuous function $\eta:[0,1] \times X \longrightarrow X $ such that
\begin{align*}
1) \quad & \eta(t,u)=u \quad \text{if} \: t=0 \: \text{or} \: u \notin  S_{2\delta }  \cap  I^{-1}([c-2\varepsilon ,c+2\varepsilon]); \\
2) \quad &\eta(1,I^{c+\varepsilon}\cap S)\subset I^{c-\varepsilon}; \\
3) \quad & \eta(t,\cdot) \,\, \text{is a homeomorphism of} \,\, X, \,\, \forall \, t\in[0,1]; \\
4) \quad & t\mapsto I(\eta(t,u)) \;\text{is nonincreasing for all}\; u \in X; \\
5) \quad & \eta(t,-u)=-\eta(t,u), \,\, \forall \, t\in[0,1], u\in X.
\end{align*}
As a consequence, since $ D \cap I^{-1}([c-2\varepsilon ,c+2\varepsilon])=\emptyset$, if we let $\varphi(u):=\eta(1,u)$, then $\varphi$ is an odd continuous map satisfying
\begin{equation}\label{3+16}
  \varphi: \, I^{c+\epsilon}\setminus A_{c,\rho}\longrightarrow I^{c-\epsilon} \quad \text{and} \quad \varphi|_{D}=id_{D}.
\end{equation}
This completes our proof of Lemma \ref{deformation}.
\end{proof}

\section{Proof of Theorem \ref{thm1}}
In this section, using ideas from \cite{CW}, we will derive the existence of ground state solutions to Schr\"{o}dinger-Newton equations \eqref{S-N} by constructing critical value for $I$, and hence complete the proof of Theorem \ref{thm1}.

In the following, we assume $p\geq2$, $q\geq2p$, $a\in C(\mathbb{R}^{2})$ and $\mathbb{Z}^{2}$-periodic with $\inf_{\mathbb{R}^{2}}a>0$. We use $\gamma(A)\in \mathbb{N}\cup\{0,\infty\}$ to denote the Krasnoselski genus of a closed and symmetric subset $A\subset X$ (for definition and properties of $\gamma$, please refer to Chapter II. 5 in \cite{S}, see also \cite{CW}). Next, let us recall the notion of relative genus $\gamma_D(A)$ for closed and symmetric subset $A \subset X$.
\begin{defn}[relative genus]
Let $D \subset Y$ be a closed symmetric subset of $X$. We define the genus of $Y$ relative to $D$, denoted by $\gamma_D(A)$, as the smallest number $k$ such that $Y$ can be covered by closed and symmetric subsets $U,\,V$ with the following properties:
\begin{align*}
1) \quad & D \subset U, \, \text{and there exists an odd continuous map} \, \chi: \, U \to D \,\, \text{such that} \,\, \chi(u)=u, \,\, \forall \, u\in D; \\
2) \quad & \gamma(V)\leq k.
\end{align*}
If no such covering exists, we set $ \gamma_D(A):=\infty$.
\end{defn}
For properties of relative genus $\gamma_{D}$, please refer to \cite{CP}, see also \cite{CW}.
\begin{lem}\label{property}(\cite{CP})
Let $D$, $Y$ and $Z$ be closed symmetric subsets of $X$ with $D\subset Y$. Then we have: \\
(1) (Subadditivity) \, $\gamma_{D}(Y\cup Z)\leq\gamma_{D}(Y)+\gamma_{D}(Z)$. \\
(2) \, If $D\subset Z$, and if there exists an odd continuous map $\varphi: \, Y\rightarrow Z$ with $\varphi(u)=u$ for any $u\in D$, then $\gamma_{D}(Y)\leq \gamma_{D}(Z)$.
\end{lem}

Now we construct various possible critical values by
\begin{equation}\label{4-0}
  c_{1}:=\inf\{c>0: \, \gamma_D(I^c)\geq 1\}, \qquad  c_{mm}:=\inf_{u\in X\setminus\{0\}}\sup_{t\in\mathbb{R}}I(tu),
\end{equation}
\begin{equation}\label{4-1}
  c_{g}:=\inf\{I(u): \, u\in X\setminus\{0\}, \, I'(u)=0\}, \qquad c_{\mathcal{N}}=\inf_{u\in\mathcal{N}}I(u),
\end{equation}
where the Nehari manifold is given by
\begin{equation}\label{4-2}
  \mathcal{N}=\{u\in X\setminus\{0\}: \, \left\langle I'(u),u\right\rangle=0\}.
\end{equation}

We have the following crucial lemma on properties of the above possible critical values.
\begin{lem}\label{lemma10}
Under the same assumptions as in Theorem \ref{thm1}, the following properties hold: \\
(1) \, $0<c_{1}=c_g=c_{\mathcal{N}}=c_{mm}<+\infty$. \\
(2) \, $c_1$ is a critical value of energy functional $I$.
\end{lem}
\begin{proof}
We will prove Lemma \ref{lemma10} by the following five steps.

\emph{Step 1}. Since $u\in\mathcal{N}$ if and only if $\varphi'_{u}(1)=0$, it follows from Lemma \ref{lemma3} that $c_{\mathcal{N}}=c_{mm}$. As a consequence of Lemma \ref{lemma2}, we can also deduce $c_{\mathcal{N}}=c_{mm}>0$. By definition, it is clear that $c_{g}\geq c_{\mathcal{N}}$, thus we have $0<c_{mm}=c_{\mathcal{N}}\leq c_{g}$.

\emph{Step 2}. We will show that $c_1\geq c_\mathcal{N}$. Suppose it is false, then we can take a number $c>0$ such that $c_1<c<c_\mathcal{N}$. Defining the sets
\begin{equation}\label{4-3}
  \mathcal{N}^{+}:=\{u\in X: \, \big\langle I'(u),u \big\rangle>0\} \quad \text{and} \quad \mathcal{N}^{-}:=\{u\in X: \, \big\langle I'(u),u \big\rangle<0\}
\end{equation}
and the function
\begin{equation}\label{4-4}
  \tau: \, \mathcal{N}^- \to [1,+\infty), \qquad  \tau(u):=\inf\{t\geq 1: \, I(tu)\leq 0\}.
\end{equation}
By Lemma \ref{lemma3}, $\tau$ is well-defined and it is an even and continuous function on $\mathcal{N}^-$. Since $c>0$ and $c_1<c<c_{\mathcal{N}}$, the closed symmetric subset $I^c\subset\mathcal{N}^-\bigcup\mathcal{N}^+\bigcup\{0\}$. Now we define an odd map
\begin{equation}\label{4-5}
  \chi: \, I^c\to X, \qquad \chi(u)=\tau(u)u, \quad \text{if} \,\, u \in\mathcal{N}^-; \qquad \chi(u)=0, \quad \text{if} \,\, u\in\mathcal{N}^+\bigcup\{0\}.
\end{equation}
Since $\mathcal{N}^{\pm}$ are open subsets and Lemma \ref{lemma2} yields that $B_{\alpha}(0)\setminus\{0\}\subseteq \mathcal{N}^{+}$ for some $\alpha>0$, thus $\chi$ is well-defined and it is a continuous map on $I^{c}$. If $u\in D$, then $I(u)\leq 0$, Lemma \ref{lemma3} implies that $u\in\mathcal{N}^-$ and $\tau(u)=1$, and hence we have $\chi|_D=id_D$, which gives us $\gamma_D(I^c)=0$ by the definition of the relative genus. However, since $c>c_1$, by the definition of $c_{1}$ and Lemma \ref{property}, one has $\gamma_D(I^c)\geq\gamma_D(I^{c_{1}})\geq1$, which is a contradiction. Thus we must have $c_1\geq c_\mathcal{N}$.

\emph{Step 3}. We will show $c_1\leq c_{mm}<+\infty$. For arbitrarily given $u\in X\setminus\{0\}$ satisfying $0<\sup_{W}I<+\infty$ with $W:=\mathbb{R}u$, by Lemma \ref{lemma3}, we can choose a $R>0$ large enough such that $I(v)\leq 0$ for all $v\in W$ with $\|v\|_{H^{1}}\geq R$, i.e., $\{v\in W: \, \|v\|_{H^{1}}\geq R\}\subset D$. In order to prove $c_1\leq\sup_{W}I$, by definition of $c_{1}$ and Lemma \ref{property}, we only need to show $\gamma_D(W \bigcup D)\geq 1$, since we have $I(v)\leq\sup_{W}I$ on $W\bigcup D$. Suppose on contrary that $\gamma_D(W\bigcup D)=0$, by the definition of the relative genus, there exists an odd continuous map $\psi: \, W\bigcup D\to D$ such that $\psi(v)=v$ for every $v\in D$. Since $\psi(0)=0$, $\psi(v)=v$ for any $v\in W$ with $\|v\|_{H^{1}}\geq R$, by continuity, there exists a $v_{\alpha}\in W$ such that $\|\psi(v_{\alpha})\|_{H^{1}}=\alpha$, where $\alpha>0$ is the same constant given in Lemma \ref{lemma2}. Hence, by Lemma \ref{lemma2}, we get $\psi(v_{\alpha})\notin D$, which is a contradict with definition of the map $\psi: \, W\bigcup D\to D$. Thus we arrive at $c_1\leq c_{mm}$.

Next, we show that $c_{mm}<+\infty$. To this end, we choose arbitrarily a function $u\in X\setminus\{0\}$ with $supp \, u\subseteq B_{\frac{1}{4}}(0)\subset \mathbb{R}^2$. Then, note that $|x-y|<1$ for every $x,y\in B_{\frac{1}{4}}(0)$, we have
\begin{equation}\label{4-6}
  V_0(u)=\int_{\mathbb{R}^{2}}\int_{\mathbb{R}^{2}}\log(|x-y|)|u|^p(x)|u|^p(y)dxdy<0.
\end{equation}
It follows easily from the proof of Lemma \ref{lemma3} (1) that
\begin{equation}\label{4-7}
  c_{mm}\leq \sup_{t\in\mathbb{R}}I(tu)<+\infty.
\end{equation}

\emph{Step 4}. Now we will prove (2). Assume $c_1$ is not a critical value, then $K_{c_1}=\emptyset$, and hence $ A_{c_1,\rho}=\emptyset $ for every $\rho>0$. Thus the quantitative deformation lemma (Lemma \ref{deformation}) implies that, there exist a $0<\varepsilon<\frac{c_{1}}{2}$ and an odd continuous map $\varphi: \, I^{c_1+\varepsilon}\longrightarrow I^{c_1-\varepsilon}$ s.t. $\varphi|_D=id_{D}$. This combining with Lemma \ref{property} yield that $\gamma_{D}(I^{c_1-\varepsilon})\geq1$ , which leads to a contradiction with the definition of $c_{1}$. This finishes our proof of (2).

\emph{Step 5}. By Steps 1-3, we have already shown $0<c_{1}=c_{\mathcal{N}}=c_{mm}<+\infty$ and $c_{1}=c_{\mathcal{N}}=c_{mm}\leq c_{g}$. Then, by Step 4 and the definition of $c_{g}$, we have $c_{g}\leq c_{1}$, thus we finally arrive at $0<c_{1}=c_{g}=c_{\mathcal{N}}=c_{mm}<+\infty$. This completes our proof of (1) and Lemma \ref{lemma10}.
\end{proof}

Now we are ready to complete our proof of Theorem \ref{thm1}.
\begin{proof}[Proof of Theorem \ref{thm1} (completed)]
By Lemma \ref{lemma10}, we derive existence of critical points $\pm u\in X\setminus\{0\}$ of $I$ such that
\begin{equation}\label{4-8}
  I(u)=c_g \quad \text{and} \quad c_g=c_{1}=\inf_{\mathcal{N}}I=\inf_{u\in X\setminus\{0\}}\sup_{t\in\mathbb{R}}I(tu)>0,
\end{equation}
and hence $\pm u$ are ground state solutions for the Schr\"{o}dinger-Newton equation \eqref{S-N}.

Next, let $u\in\mathcal{N}$ be an arbitrary minimizer of $I|_\mathcal{N}$, i.e., $I(u)=\inf_{\mathcal{N}}I$. We are to show that $u$ is a critical point of $I$. If not, suppose there exists $v \in X$ s.t. $\big\langle I'(u),v\big\rangle<0$. Since $I'$ is continuous in $X$, there exist $\eta>0$ and $\delta>0$ small enough such that
\begin{equation}\label{4-9}
  \big\langle I'(t(u+sv)),v \big\rangle<0, \qquad \forall \,\, t\in[1-\delta,1+\delta], \,\, s\in[-\eta,\eta].
\end{equation}
Since $u\in\mathcal{N}$, by Lemma \ref{lemma3}, we can deduce
\begin{equation}\label{4-10}
  \big\langle I'((1-\delta)u),(1-\delta)u\big\rangle>0>\big\langle I'((1+\delta)u),(1+\delta)u\big\rangle.
\end{equation}
Then, by continuity of $I'$, there exists a $\hat{s}\in(0,\eta)$ small enough s.t.
\begin{equation}\label{4-11}
  \big\langle I'((1-\delta)(u+\hat{s}v)),(1-\delta)(u+\hat{s}v) \big\rangle>0>\big\langle I'((1+\delta)(u+\hat{s}v)),(1+\delta)(u+\hat{s}v)\big\rangle,
\end{equation}
which implies that there exists some $\hat{t}\in(1-\delta,1+\delta)$ such that $\hat{t}(u+\hat{s}v)\in\mathcal{N}$. However, by Lemma \ref{lemma3} and \eqref{4-9}, we have
\begin{equation}\label{4-12}
  I(\hat{t}(u+\hat{s}v))-I(u)\leq I(\hat{t}(u+\hat{s}v))-I(\hat{t}u)=-\hat{t}\int_0^{\hat{s}}\big\langle I'(\hat{t}(u+sv)),v\big\rangle ds<0,
\end{equation}
which yields a contradiction with $I(u)=\inf_{\mathcal{N}}I$. Thus $u$ must be a critical point of $I$.

Finally, if $u\in\mathcal{N}$ is a minimizer of $I|_{\mathcal{N}}$, then $|u|\in\mathcal{N}$ is also a minimizer of $I|_{\mathcal{N}}$. Hence $|u|$ is a critical point of $I$ and satisfies the Euler-Lagrange equation \eqref{S-N}. By Lemma \ref{lemma1}, we have $|u|\in W^{2,r}(\mathbb{R}^{2})$ ($\forall \, 1\leq r<+\infty$) is a strong solution to elliptic equation of the form $-\Delta |u|+\zeta(x)|u|=0$ with $\zeta\in L_{loc}^{\infty}(\mathbb{R}^{2})$. Since $|u|\not\equiv0$, by Harnack inequality (see \cite{GT}), we have $|u|>0$ on $\mathbb{R}^{2}$ and hence $u$ does not change sign. This concludes our proof of Theorem \ref{thm1}.
\end{proof}

\section{Proof of Theorem \ref{thm3}}
In this section, we will prove Theorem \ref{thm3}. We first prove that mountain pass energy level $c_{mp}>0$ and there exists a critical point $u\in X\setminus\{0\}$ of $I$ such that $I(u)=c_{mp}$ (for the definition of $c_{mp}$, see \eqref{mp}). To this end, we need the following general minimax principle from Proposition 2.8 in \cite{LW}, which will lead to Cerami sequences instead of Palais-Smale sequences.

\begin{prop}[\cite{LW}]\label{prop1}
Assume $X$ is a Banach space. Let $M_{0}$ be a closed subspace of the metric space $M$ and $\Gamma_{0}\subset C(M_{0},X)$. Define
\begin{equation*}
  \Gamma:=\{\gamma\in C(M,X):\, \gamma|_{M_{0}}\in\Gamma_{0}\}.
\end{equation*}
If $\phi\in C^{1}(X,\mathbb{R})$ satisfies
\begin{equation*}
 +\infty>c:=\inf _{\gamma \in \Gamma} \sup _{t\in M}\phi(\gamma(t))>\sigma:=\sup _{\gamma_{0} \in \Gamma_{0}} \sup _{t\in M_{0}} \phi\left(\gamma_{0}(t)\right),
\end{equation*}
then, for every $\epsilon\in\left(0,\frac{c-\sigma}{2}\right)$, $\delta>0$ and $\gamma\in\Gamma$ with $\sup_{t\in M} \phi(\gamma(t)) \leqslant c+\epsilon$, there exists $u\in X$ such that \\
\noindent (i) $c-2\epsilon\leq \phi(u)\leq c+2\epsilon$, \\
\noindent (ii) $dist(u,\gamma(M))\leq 2\delta$, \\
\noindent (iii) $\left(1+\|u\|_{X}\right)\|\phi'(u)\|_{X'}\leq\frac{8\epsilon}{\delta}$.
\end{prop}

Lemmas \ref{lemma2} and \ref{lemma11} implies that
\begin{equation}\label{4}
  0<\inf_{\|u\|_{H^{1}}=\alpha}I(u)\leq c_{mp}<+\infty,
\end{equation}
and hence the functional $I$ has a mountain pass geometry. Now we will use Proposition \ref{prop1} to show the existence of a Cerami sequence $\{u_{n}\}\subset X$ at the mountain pass energy level $c_{mp}$ with a key additional property $J_{k}(u_{n})\rightarrow0$ ($k=1$ or $2$) as $n\rightarrow+\infty$.

\begin{lem}\label{lemma12}
Assume $p\geq2$ and $q>2$. Then, for any given $k=1$ or $2$, there exists a sequence $\{u^{k}_{n}\}\subset X$ such that, as $n\rightarrow+\infty$,
\begin{equation}\label{25}
I\left(u^{k}_{n}\right) \longrightarrow c_{m p}, \quad\left\|I^{\prime}(u^{k}_{n})\right\|_{X^{\prime}}\left(1+\|u^{k}_{n}\|_{X}\right) \longrightarrow 0, \quad J_{k}\left(u^{k}_{n}\right) \longrightarrow 0,
\end{equation}
where the auxiliary functional $J_{k}$ is defined by \eqref{af}.
\end{lem}
\begin{proof}
Let $k=1$ or $2$ be fixed. Following the ideas from \cite{J} (see also \cite{DW,HIT,MS}), we define the Banach space
\begin{equation}\label{0}
\tilde{X}:=\mathbb{R} \times X, \qquad \|(s, v)\|_{\tilde{X}}^{2}:=|s|^{2}+\|v\|_{X}^{2}, \quad \forall \,\, (s, v) \in \mathbb{R} \times X,
\end{equation}
and the continuous map
\begin{equation}\label{1}
\rho_{k}: \tilde{X} \longrightarrow X, \qquad \rho_{k}(s, v)(\cdot)=e^{ks}v\left(e^{s}\cdot\right), \quad \forall \,\, (s, v) \in \mathbb{R} \times X,
\end{equation}
which is linear with respect to $v\in X$ for any fixed $s\in\mathbb{R}$. Moreover, we define a functional on $\tilde{X}$ by $\phi_{k}(s,v):=I(\rho_{k}(s,v))$, then we have
\begin{equation}\label{2}
\begin{aligned}
\phi_{k}(s, v) &=\frac{1}{2} e^{2ks} \int_{\mathbb{R}^{2}}|\nabla v|^{2}dx+\frac{1}{2} e^{2(k-1)s} \int_{\mathbb{R}^{2}}a|v|^{2}dx \\
&+\frac{\gamma}{4p\pi} e^{2(kp-2)s} \int_{\mathbb{R}^{2}}\int_{\mathbb{R}^{2}}\log(|x-y|)|v|^{p}(x)|v|^{p}(y)dxdy \\
&-\frac{\gamma}{4p\pi}se^{2(kp-2)s}\left(\int_{\mathbb{R}^{2}}|v|^{p}dx\right)^{2}-\frac{b}{q} e^{(kq-2)s}\int_{\mathbb{R}^{2}}|v|^{q}dx.
\end{aligned}
\end{equation}
Through direct calculations, we get
\begin{equation}\label{3}\begin{aligned}
&\quad \partial_{s} \phi_{k}(s, v) \\
&=ke^{2ks}\int_{\mathbb{R}^{2}}|\nabla v|^{2}dx+(k-1)e^{2(k-1)s}\int_{\mathbb{R}^{2}}a|v|^{2}dx
-\frac{\left(kq-2\right)b}{q}e^{(kq-2)s}\int_{\mathbb{R}^{2}}|v|^{q}dx \\
&\quad +\frac{\gamma}{2p\pi}(kp-2)e^{2(kp-2)s} \int_{\mathbb{R}^{2}}\int_{\mathbb{R}^{2}}\log(|x-y|)|v|^{p}(x)|v|^{p}(y)dxdy \\
&\quad -\frac{\gamma}{4p\pi}\left[1+2s\left(kp-2\right)\right]e^{2(kp-2)s}\left(\int_{\mathbb{R}^{2}}|v|^{p}dx\right)^{2} \\
&=k\int_{\mathbb{R}^{2}}|\nabla \rho_{k}|^{2}dx+(k-1)\int_{\mathbb{R}^{2}}a|\rho_{k}|^{2}dx-\frac{(kq-2)b}{q}\int_{\mathbb{R}^{2}}|\rho_{k}|^{q}dx \\
&\quad +\frac{\gamma}{2p\pi}\left(kp-2\right)\int_{\mathbb{R}^{2}}\int_{\mathbb{R}^{2}}\log (|x-y|)|\rho_{k}|^{p}(x)|\rho_{k}|^{p}(y)dxdy-\frac{\gamma}{4p\pi}\left(\int_{\mathbb{R}^{2}}|\rho_{k}|^{p}dx\right)^{2} \\
&=J_{k}(\rho_{k}(s, v)),
\end{aligned}\end{equation}
where the auxiliary functional $J_{k}$ is defined by \eqref{af}. Moreover, one has
\begin{equation}\label{4}\begin{aligned}
\left\langle\partial_{v} \phi_{k}(s, v), w\right\rangle &=\left.\frac{d}{d t}\phi_{k}(s, v+t w)\right|_{t=0}=\left.\frac{d}{d t}I(\rho_{k}(s, v+t w))\right|_{t=0} \\
&=\left.\frac{d}{d t} I(\rho_{k}(s, v)+t \rho_{k}(s, w))\right|_{t=0}=\left\langle I^{\prime}(\rho_{k}(s, v)), \rho_{k}(s, w)\right\rangle.
\end{aligned}\end{equation}
Thus the functional $\phi_{k}\in C^{1}(\tilde{X},\mathbb{R})$ and for any $s\in\mathbb{R}$ and $v,w\in X$,
\begin{eqnarray}\label{5}
\left\langle\phi_{k}^{\prime}(s, v),(h, w)\right\rangle &=& \partial_{s} \phi_{k}(s, v) h+\left\langle\partial_{v}\phi_{k}(s,v),w\right\rangle \\
\nonumber &=& J_{k}(\rho_{k}(s, v))h+\left\langle I^{\prime}(\rho_{k}(s, v)), \rho_{k}(s, w)\right\rangle.
\end{eqnarray}

Now we define the following mountain pass value for $\tilde{X}$ and $\phi_{k}$:
\begin{equation}\label{6}
c_{m p,k}^{\ast}:=\inf_{\tilde{\gamma}\in\tilde{\Gamma}_{k}}\max_{t\in[0,1]}\phi_{k}(\tilde{\gamma}(t)),
\end{equation}
where
\begin{equation}\label{7}
  \tilde{\Gamma}_{k}=\left\{\tilde{\gamma} \in C([0,1],\tilde{X})\mid\tilde{\gamma}(0)=(0,0), \phi_{k}(\tilde{\gamma}(1))<0\right\}.
\end{equation}
One can easily observe that $\Gamma=\left\{\rho_{k}\circ\tilde{\gamma}\mid\tilde{\gamma}\in\tilde{\Gamma}_{k}\right\}$ (for the definition of $\Gamma$, see \eqref{mp}), and hence $c_{mp}=c^{\ast}_{mp,k}$. By the definition of $c_{mp}$, for any $n\in\mathbb{N}$, there exists a $\gamma_n\in\Gamma$ such that
\begin{equation}\label{8}
\max _{t \in[0,1]} I\left(\gamma_{n}(t)\right)\leq c_{m p}+\frac{1}{n^{2}}.
\end{equation}
Then we define $\tilde{\gamma}_{n}\in\tilde{\Gamma}_{k}$ by $\tilde{\gamma}_{n}(t):=(0,\gamma_{n}(t))$ and obtain that
\begin{equation}\label{9}
\max _{t \in[0,1]}\phi_{k}\left(\tilde{\gamma}_{n}(t)\right)=\max _{t \in[0,1]} I\left(\gamma_{n}(t)\right)\leq c_{m p}+\frac{1}{n^{2}}.
\end{equation}
For any $n\in\mathbb{N}$, by taking $X=\tilde{X}$, $\Gamma=\tilde{\Gamma}_{k}$, $M=[0,1]$, $M_{0}=\{0,1\}$, $\gamma=\tilde{\gamma}_{n}$, $\epsilon=\frac{c_{mp}}{4n^{2}}$ and $\delta=\frac{1}{n}$ in Proposition \ref{prop1}, we can deduce from Proposition \ref{prop1} that, there exists a sequence $\{(s^{k}_{n}, v^{k}_{n})\}\in\tilde{X}$ such that
\begin{equation}\label{10}
c_{mp}-\frac{c_{mp}}{2n^{2}} \leq \phi_{k}\left(s^{k}_{n},v^{k}_{n}\right) \leq c_{m p}+\frac{c_{mp}}{2n^{2}},
\end{equation}
\begin{equation}\label{11}
\operatorname{dist}\left(\left(s^{k}_{n},v^{k}_{n}\right),\{0\} \times \gamma_{n}([0,1])\right) \leq \frac{2}{n},
\end{equation}
\begin{equation}\label{12}
\left(1+\left\|\left(s^{k}_{n},v^{k}_{n}\right)\right\|_{\tilde{X}}\right)\left\|\phi_{k}^{\prime}\left(s^{k}_{n},v^{k}_{n}\right)\right\|_{\tilde{X}^{\prime}}\leq \frac{2c_{mp}}{n}.
\end{equation}
It follow directly from \eqref{11} that $\lim_{n\rightarrow+\infty}s^{k}_{n}=0$. Moreover, by taking $h=1$ and $w=0$ in \eqref{5}, we get
\begin{equation}\label{13}
\left|\left\langle\phi_{k}^{\prime}\left(s^{k}_{n}, v^{k}_{n}\right),(1,0)\right\rangle\right|=\left|J_{k}\left(\rho_{k}\left(s^{k}_{n}, v^{k}_{n}\right)\right)\right|\leq\left\|\phi_{k}^{\prime}\left(s^{k}_{n}, v^{k}_{n}\right)\right\|_{\tilde{X}^{\prime}} \longrightarrow 0.
\end{equation}
Combining this with \eqref{10} yields that, for $u^{k}_{n}:=\rho_{k}(s^{k}_{n},v^{k}_{n})\in X$,
\begin{equation}\label{14}
I\left(u^{k}_{n}\right)\longrightarrow c_{m p}, \qquad J_{k}(u^{k}_{n})\rightarrow0, \qquad \text{as} \,\, n\rightarrow+\infty.
\end{equation}
Furthermore, we can show that
\begin{equation}\label{15}
\left\|u^{k}_{n}\right\|_{X}\leq2\left(1+o_{n}(1)\right)\left\|v^{k}_{n}\right\|_{X}\leq2\left(1+o_{n}(1)\right)\left\|\left(s^{k}_{n}, v^{k}_{n}\right)\right\|_{\tilde{X}}.
\end{equation}
Indeed, through direct calculations, we get
\begin{eqnarray}\label{16}
\|u^{k}_{n}\|^{2}_{H^{1}(\mathbb{R}^{2})}&:=&\int_{\mathbb{R}^{2}}\left(|\nabla u^{k}_{n}|^{2}+a|u^{k}_{n}|^{2}\right)dx \\
\nonumber &=&e^{2ks^{k}_{n}}\int_{\mathbb{R}^{2}}\left|\nabla v^{k}_{n}\right|^{2}dx+e^{2(k-1)s^{k}_{n}}\int_{\mathbb{R}^{2}}a|v^{k}_{n}|^{2}dx \\
\nonumber &=&(1+o_{n}(1))\int_{\mathbb{R}^{2}}\left(|\nabla v^{k}_{n}|^{2}+a|v^{k}_{n}|^{2}\right)dx=(1+o_{n}(1))\|v^{k}_{n}\|^{2}_{H^{1}(\mathbb{R}^{2})},
\end{eqnarray}
\begin{equation}\label{17}\begin{aligned}
|u^{k}_{n}|^{p}_{\ast}&:=\left\|u^{k}_{n}\right\|_{L^{p}(d\mu)}^{p}=\int_{\mathbb{R}^{2}}\log(1+|x|)\left|u^{k}_{n}\right|^{p}(x)dx \\
&=\int_{\mathbb{R}^{2}}\log(1+|x|)\left|e^{ks^{k}_{n}}v^{k}_{n}\right|^{p}\left(e^{s^{k}_{n}}x\right)dx \\
&=e^{(kp-2)s^{k}_{n}}\int_{\mathbb{R}^{2}}\log\left(1+e^{-s^{k}_{n}}|x|\right)\left|v^{k}_{n}\right|^{p}(x)dx \\
&=\left(1+o_{n}(1)\right)\int_{\mathbb{R}^{2}}\log(1+|x|)\left|v^{k}_{n}\right|^{p}dx
-s^{k}_{n}\left(1+o_{n}(1)\right)\int_{\mathbb{R}^{2}}\left|v^{k}_{n}\right|^{p}(x)dx \\
&=\left(1+o_{n}(1)\right)|v^{k}_{n}|^{p}_{\ast}+o_{n}(1)\|v^{k}_{n}\|^{p}_{H^{1}(\mathbb{R}^{2})} \\
&\leq\left(1+o_{n}(1)\right)\|v^{k}_{n}\|^{p}_{X}.
\end{aligned}\end{equation}
Consequently, we arrive at
\begin{equation}\label{18}
  \left\|u^{k}_{n}\right\|_{X}=\left\|u^{k}_{n}\right\|_{H^{1}(\mathbb{R}^{2})}+|u^{k}_{n}|_{\ast}\leq 2\left(1+o_{n}(1)\right)\left\|v^{k}_{n}\right\|_{X},
\end{equation}
and hence \eqref{15} holds.

Next, we will show that $ \left\|I^{\prime}(u^{k}_{n})\right\|_{X^{\prime}}\left(1+\|u^{k}_{n}\|_{X}\right)\longrightarrow 0$. To this end, for any $w\in X$, let us define $w^{k}_{n}:=e^{-ks^{k}_{n}}w(e^{-s^{k}_{n}}\cdot)\in X$, then we have
\begin{eqnarray}\label{19}
\|w^{k}_{n}\|^{2}_{H^{1}(\mathbb{R}^{2})}&=&e^{-2ks^{k}_{n}}\int_{\mathbb{R}^{2}}\left|\nabla w\right|^{2}dx+e^{-2(k-1)s^{k}_{n}}\int_{\mathbb{R}^{2}}a w^{2}dx \\
\nonumber &=&(1+o_{n}(1))\int_{\mathbb{R}^{2}}\left(|\nabla w|^{2}+aw^{2}\right)dx=(1+o_{n}(1))\|w\|^{2}_{H^{1}(\mathbb{R}^{2})},
\end{eqnarray}
\begin{equation}\label{20}\begin{aligned}
|w^{k}_{n}|^{p}_{\ast}&=\int_{\mathbb{R}^{2}}\log(1+|x|)\left|e^{-ks^{k}_{n}}w\right|^{p}\left(e^{-s^{k}_{n}} x\right) d x \\
&=e^{-(kp-2)s^{k}_{n}}\int_{\mathbb{R}^{2}}\log \left(1+e^{s^{k}_{n}}|x|\right)\left|w\right|^{p}(x) d x \\
&=\left(1+o_{n}(1)\right)\int_{\mathbb{R}^{2}}\log(1+|x|)\left|w\right|^{p}dx+s^{k}_{n}\left(1+o_{n}(1)\right)\int_{\mathbb{R}^{2}}\left|w\right|^{p}(x)dx \\
&=\left(1+o_{n}(1)\right)|w|^{p}_{\ast}+o_{n}(1)\|w\|^{p}_{H^{1}(\mathbb{R}^{2})} \\
&\leq\left(1+o_{n}(1)\right)\|w\|^{p}_{X},
\end{aligned}\end{equation}
and hence
\begin{equation}\label{21}
\left\|w^{k}_{n}\right\|_{X}\leq2\left(1+o_{n}(1)\right)\left\|w\right\|_{X}.
\end{equation}
By \eqref{5} with $h=0$ and \eqref{21}, we can infer that
\begin{equation}\label{22}\begin{aligned}
\left|\left\langle I^{\prime}\left(u^{k}_{n}\right), w\right\rangle\right| &=\left|\left\langle I^{\prime}\left(\rho_{k}\left(s^{k}_{n}, v^{k}_{n}\right)\right), \rho_{k}\left(s^{k}_{n}, w^{k}_{n}\right)\right\rangle\right| \\
&=\left|\left\langle\partial_{v}\phi_{k}\left(s^{k}_{n},v^{k}_{n}\right), w^{k}_{n}\right\rangle\right|=\left|\left\langle\phi_{k}'\left(s^{k}_{n}, v^{k}_{n}\right), (0,w^{k}_{n})\right\rangle\right|\\
& \leq\left\|\phi_{k}^{\prime}\left(s^{k}_{n}, v^{k}_{n}\right)\right\|_{\tilde{X}^{\prime}}\left\|w^{k}_{n}\right\|_{X}
\leq2\left(1+o_{n}(1)\right)\left\|\phi_{k}^{\prime}\left(s^{k}_{n},v^{k}_{n}\right)\right\|_{\tilde{X}^{\prime}}\left\|w\right\|_{X},
\end{aligned}\end{equation}
which implies that
\begin{equation}\label{23}
\left\|I^{\prime}(u^{k}_{n})\right\|_{X^{\prime}}\leq C\left\|\phi^{\prime}\left(s^{k}_{n}, v^{k}_{n}\right)\right\|_{\tilde{X}^{\prime}}.
\end{equation}
Combining this with \eqref{12} and \eqref{15} yields that
\begin{equation}\label{24}\begin{aligned}
\left\|I^{\prime}(u^{k}_{n})\right\|_{X^{\prime}}\left(1+\|u^{k}_{n}\|_{X}\right) & \leq C\left\|\phi_{k}^{\prime}\left(s^{k}_{n}, v^{k}_{n}\right)\right\|_{\tilde{X}^{\prime}}\left(1+\left\|\left(s^{k}_{n}, v^{k}_{n}\right)\right\|_{\tilde{X}}\right) \rightarrow 0,
\end{aligned}\end{equation}
as $n\rightarrow+\infty$. This concludes our proof of Lemma \ref{lemma12}.
\end{proof}

Next, we will prove the following key lemma which indicates that any sequence $\{u_{n}\}$ satisfying \eqref{25} (with $k=1$ and $p\geq3$ \emph{or} $k=2$ and $2\leq p<3$) is bounded in $H^{1}(\mathbb{R}^{2})$.
\begin{lem}\label{lemma13}
Assume $p\geq2$, $q\geq 2p-2$ and $q>2$. If $p\geq3$, let $\{u_{n}\}\subset X$ satisfies
\begin{equation}\label{eq-a0}
  c:=\sup_{n\in\mathbb{N}}I(u_{n})<+\infty, \qquad J_{1}(u_{n})\rightarrow0;
\end{equation}
if $2\leq p<3$, assume further that $q\geq2p-1$ or $q<p+1$ and let $\{u_{n}\}\subset X$ satisfies
\begin{equation}\label{eq-a1}
  c:=\sup_{n\in\mathbb{N}}I(u_{n})<+\infty, \qquad \|I'(u_{n})\|_{X'}(1+\|u_{n}\|_{X})\rightarrow 0, \quad J_{2}(u_{n})\rightarrow0,
\end{equation}
as $n\rightarrow+\infty$. Then $\{u_{n}\}$ is bounded in $H^{1}(\mathbb{R}^{2})$.
\end{lem}
\begin{proof}
We first consider the cases $p\geq3$. By \eqref{eq-a0}, we have
\begin{equation}\label{26'}\begin{aligned}
c+o(1)&\geq I\left(u_{n}\right)-\frac{1}{2(p-2)}J_{1}\left(u_{n}\right) \\
&=\left(\frac{1}{2}-\frac{1}{2(p-2)}\right)\int_{\mathbb{R}^{2}}\left|\nabla u_{n}\right|^{2}dx+\frac{a}{2}\int_{\mathbb{R}^{2}}u_{n}^{2}dx \\
&+\frac{\gamma}{8\pi p(p-2)}\left(\int_{\mathbb{R}^{2}}\left|u_{n}\right|^{p}dx\right)^{2}
+\frac{b}{q}\left(\frac{q-2}{2(p-2)}-1\right)\int_{\mathbb{R}^{2}}\left|u_{n}\right|^{q}dx.
\end{aligned}\end{equation}
If $p>3$, it follows immediately from \eqref{26'} that $\{u_{n}\}$ is bounded in $H^{1}(\mathbb{R}^{2})$ due to $q\geq2p-2$. If $p=3$ and $q>2p-2=4$, then \eqref{26'} implies that $\{u_{n}\}$ is bounded in $L^{2}(\mathbb{R}^{2})$, $L^{3}(\mathbb{R}^{2})$, and $L^{q}(\mathbb{R}^{2})$ if $b>0$. Thus we infer from \eqref{2-12} and the Gagliardo-Nirenberg inequality that
\begin{eqnarray}\label{27'}
\left\|u_{n}\right\|_{H^{1}(\mathbb{R}^{2})}^{2}&=&2I\left(u_{n}\right)+\frac{\gamma}{6\pi}\left(V_{2}\left(u_{n}\right)-V_{1}\left(u_{n}\right)\right)
+\frac{2b}{q}\left\|u_{n}\right\|_{L^{q}(\mathbb{R}^{2})}^{q} \\
\nonumber &\leq& 2c+C\|u_{n}\|^{6}_{L^{4}(\mathbb{R}^{2})}+\frac{2b}{q}\left\|u_{n}\right\|_{L^{q}(\mathbb{R}^{2})}^{q} \\
\nonumber &\leq& 2c+C\left\|u_{n}\right\|_{L^{3}(\mathbb{R}^{2})}^{\frac{9}{2}}\left\|\nabla u_{n}\right\|_{L^{2}(\mathbb{R}^{2})}^{\frac{3}{2}}
+\frac{2b}{q}\left\|u_{n}\right\|_{L^{q}(\mathbb{R}^{2})}^{q} \\
\nonumber &\leq& C_{1}+C_{2}\left\|u_{n}\right\|_{H^{1}(\mathbb{R}^{2})}^{\frac{3}{2}},
\end{eqnarray}
and hence $\{u_{n}\}$ is bounded in $H^{1}(\mathbb{R}^{2})$. In the case $p=3$ and $q=4$, it follows from \eqref{26'} that $\{u_{n}\}$ is bounded in $L^{2}(\mathbb{R}^{2})$ and $L^{3}(\mathbb{R}^{2})$. Thus we deduce from \eqref{2-12} and the Gagliardo-Nirenberg inequality that
\begin{eqnarray}\label{27'+}
\left\|u_{n}\right\|_{H^{1}(\mathbb{R}^{2})}^{2}&=&2I\left(u_{n}\right)+\frac{\gamma}{6\pi}\left(V_{2}\left(u_{n}\right)-V_{1}\left(u_{n}\right)\right)
+\frac{b}{2}\left\|u_{n}\right\|_{L^{4}(\mathbb{R}^{2})}^{4} \\
\nonumber &\leq& 2c+C\|u_{n}\|^{6}_{L^{4}(\mathbb{R}^{2})}+\frac{b}{2}\left\|u_{n}\right\|_{L^{4}(\mathbb{R}^{2})}^{4} \\
\nonumber &\leq& 2c+C\left\|u_{n}\right\|_{L^{3}(\mathbb{R}^{2})}^{\frac{9}{2}}\left\|\nabla u_{n}\right\|_{L^{2}(\mathbb{R}^{2})}^{\frac{3}{2}}
+C\left\|u_{n}\right\|_{L^{3}(\mathbb{R}^{2})}^{3}\left\|\nabla u_{n}\right\|_{L^{2}(\mathbb{R}^{2})} \\
\nonumber &\leq& C_{1}+C_{2}\left\|u_{n}\right\|_{H^{1}(\mathbb{R}^{2})}^{\frac{3}{2}}+C_{3}\left\|u_{n}\right\|_{H^{1}(\mathbb{R}^{2})},
\end{eqnarray}
and hence $\{u_{n}\}$ is bounded in $H^{1}(\mathbb{R}^{2})$.

Next, we consider the cases $2\leq p<3$. From \eqref{eq-a1}, we get
\begin{equation}\label{26}\begin{aligned}
c+o(1)&\geq I\left(u_{n}\right)-\frac{1}{4(p-1)}J_{2}\left(u_{n}\right) \\
&=\left(\frac{1}{2}-\frac{1}{2(p-1)}\right)\int_{\mathbb{R}^{2}}\left|\nabla u_{n}\right|^{2}dx+\left(\frac{1}{2}-\frac{1}{4(p-1)}\right)\int_{\mathbb{R}^{2}}au_{n}^{2}dx \\
&+\frac{\gamma}{16\pi p(p-1)}\left(\int_{\mathbb{R}^{2}}\left|u_{n}\right|^{p}dx\right)^{2}
+\frac{(q+1-2p)b}{2(p-1)q}\int_{\mathbb{R}^{2}}\left|u_{n}\right|^{q}dx.
\end{aligned}\end{equation}
We will discuss two different cases $q\geq 2p-1$ and $q<2p-1$.

\emph{Case (i)} $q\geq 2p-1$. If $p>2$, it follows immediately from \eqref{26} that $\{u_{n}\}$ is bounded in $H^{1}(\mathbb{R}^{2})$. If $p=2$ and $q>2p-1=3$, then \eqref{26} implies that $\{u_{n}\}$ is bounded in $L^{2}(\mathbb{R}^{2})$ and $L^{q}(\mathbb{R}^{2})$ if $b>0$. Thus, we infer from \eqref{2-12} and the Gagliardo-Nirenberg inequality that
\begin{eqnarray}\label{27}
\left\|u_{n}\right\|_{H^{1}(\mathbb{R}^{2})}^{2}&=&2I\left(u_{n}\right)+\frac{\gamma}{4\pi}\left(V_{2}\left(u_{n}\right)-V_{1}\left(u_{n}\right)\right)
+\frac{2b}{q}\left\|u_{n}\right\|_{L^{q}(\mathbb{R}^{2})}^{q} \\
\nonumber &\leq& 2c+C\|u_{n}\|^{4}_{L^{\frac{8}{3}}(\mathbb{R}^{2})}+\frac{2b}{q}\left\|u_{n}\right\|_{L^{q}(\mathbb{R}^{2})}^{q} \\
\nonumber &\leq& 2c+C\left\|u_{n}\right\|_{L^{2}(\mathbb{R}^{2})}^{3}\left\|\nabla u_{n}\right\|_{L^{2}(\mathbb{R}^{2})}
+\frac{2b}{q}\left\|u_{n}\right\|_{L^{q}(\mathbb{R}^{2})}^{q} \\
\nonumber &\leq& C_{1}+C_{2}\left\|u_{n}\right\|_{H^{1}(\mathbb{R}^{2})},
\end{eqnarray}
and hence $\{u_{n}\}$ is bounded in $H^{1}(\mathbb{R}^{2})$. In the case $p=2$ and $q=3$, it follows from \eqref{26} that $\{u_{n}\}$ is bounded in $L^{2}(\mathbb{R}^{2})$. Thus we deduce from \eqref{2-12} and the Gagliardo-Nirenberg inequality that
\begin{eqnarray}\label{27+}
\left\|u_{n}\right\|_{H^{1}(\mathbb{R}^{2})}^{2}&=&2I\left(u_{n}\right)+\frac{\gamma}{4\pi}\left(V_{2}\left(u_{n}\right)-V_{1}\left(u_{n}\right)\right)
+\frac{2b}{3}\left\|u_{n}\right\|_{L^{3}(\mathbb{R}^{2})}^{3} \\
\nonumber &\leq& 2c+C\|u_{n}\|^{4}_{L^{\frac{8}{3}}(\mathbb{R}^{2})}+\frac{2b}{3}\left\|u_{n}\right\|_{L^{3}(\mathbb{R}^{2})}^{3} \\
\nonumber &\leq& 2c+C\left\|u_{n}\right\|_{L^{2}(\mathbb{R}^{2})}^{3}\left\|\nabla u_{n}\right\|_{L^{2}(\mathbb{R}^{2})}
+C\left\|u_{n}\right\|_{L^{2}(\mathbb{R}^{2})}^{2}\left\|\nabla u_{n}\right\|_{L^{2}(\mathbb{R}^{2})} \\
\nonumber &\leq& C_{1}+C_{2}\left\|u_{n}\right\|_{H^{1}(\mathbb{R}^{2})},
\end{eqnarray}
and hence $\{u_{n}\}$ is bounded in $H^{1}(\mathbb{R}^{2})$.

\emph{Case (ii)} $q<2p-1$. Assume further that $q<p+1$. We will first show that $\|\nabla u_{n}\|_{L^{2}(\mathbb{R}^{2})}$ ($n\in\mathbb{N}$) are bounded.

If not, suppose on the contrary that, after extracting a subsequence, we have $\|\nabla u_{n}\|_{L^{2}(\mathbb{R}^{2})}\rightarrow+\infty$ as $n\rightarrow+\infty$. Let $t_{n}:=\|\nabla u_{n}\|_{L^{2}(\mathbb{R}^{2})}^{-\frac{1}{2}}$ ($n\in\mathbb{N}$), so $t_{n}\rightarrow0$ as $n\rightarrow+\infty$. We define the rescaled functions $v_{n}:=t_{n}^{2}u_{n}(t_{n}\cdot)\in X$ for any $n\in\mathbb{N}$. Then we have
\begin{equation}\label{28}
\left\|\nabla v_{n}\right\|_{L^{2}}=1 \qquad \text { and } \qquad\left\|v_{n}\right\|_{L^{r}}^{r}=t_{n}^{2r-2}\left\|u_{n}\right\|_{L^{r}}^{r}, \quad \forall \,\, n \in \mathbb{N}, \,\, 1 \leqslant r<+\infty.
\end{equation}
From the Gagliardo-Nirenberg inequality, one can deduce that, for any $n \in \mathbb{N}$,
\begin{equation}\label{29}
\left\|v_{n}\right\|_{L^{p}}^{p} \leqslant C\left\|v_{n}\right\|_{L^{2}}^{2}\left\|\nabla v_{n}\right\|_{L^{2}}^{p-2}=C\left\|v_{n}\right\|_{L^{2}}^{2},
\end{equation}
\begin{equation}\label{30}
\left\|v_{n}\right\|_{L^{q}}^{q} \leqslant C\left\|v_{n}\right\|_{L^{p}}^{p}\left\|\nabla v_{n}\right\|_{L^{2}}^{q-p}=C\left\|v_{n}\right\|_{L^{p}}^{p}.
\end{equation}
Multiplying \eqref{26} by $t_{n}^{2p}$, we can infer from \eqref{28}, \eqref{29} and \eqref{30} that
\begin{equation}\label{31}\begin{aligned}
&\quad c t_{n}^{4}+o\left(t_{n}^{4}\right) \\
&\geq\left(\frac{1}{2}-\frac{1}{2(p-1)}\right)t_{n}^{4}\left\|\nabla u_{n}\right\|_{L^{2}}^{2}+a\left(\frac{1}{2}-\frac{1}{4(p-1)}\right)t_{n}^{4}\|u_{n}\|_{L^{2}}^{2} \\
&\quad +\frac{\gamma}{16\pi p(p-1)}t_{n}^{4}\left\|u_{n}\right\|_{L^{p}}^{2p}
-\frac{(2p-1-q)b}{2(p-1)q}t_{n}^{4}\left\|u_{n}\right\|_{L^{q}}^{q} \\
&\geq\frac{1}{2}-\frac{1}{2(p-1)}+Ca\left(\frac{1}{2}-\frac{1}{4(p-1)}\right)t_{n}^{2}\|v_{n}\|_{L^{p}}^{p} \\
&\quad +\frac{\gamma}{16\pi p(p-1)}t_{n}^{4(2-p)}\left\|v_{n}\right\|_{L^{p}}^{2p}
-C\frac{(2p-1-q)b}{2(p-1)q}t_{n}^{6-2q}\left\|v_{n}\right\|_{L^{p}}^{p}.
\end{aligned}\end{equation}
Consequently, we obtain a contradiction when $b=0$. If $b>0$, then \eqref{31} yields that
\begin{equation}\label{32}
\left\|v_{n}\right\|_{L^{p}(\mathbb{R}^{2})}=O\left(t_{n}^{\frac{2(2p-1-q)}{p}}\right).
\end{equation}
In the cases $2<p<3$, \eqref{31} also implies that
\begin{equation}\label{36}
\left\|v_{n}\right\|_{L^{p}(\mathbb{R}^{2})}\geq Ct_{n}^{\frac{2(q-3)}{p}},
\end{equation}
which is in contradiction with \eqref{32} due to $q<p+1$. If $p=2$ and $q<3$, by \eqref{2-12}, \eqref{eq-a1}, \eqref{28}, \eqref{32} and the Gagliardo-Nirenberg inequality, we also have
\begin{eqnarray}\label{33}
&& o_{n}(1)=t_{n}^{4}J_{2}\left(u_{n}\right)=t_{n}^{4}\bigg(2\left\|\nabla u_{n}\right\|_{L^{2}}^{2}+a\left\|u_{n}\right\|_{L^{2}}^{2} \\
\nonumber &&\qquad\quad\, +\frac{\gamma}{2\pi}V_{0}\left(u_{n}\right)-\frac{\gamma}{8\pi}\left\|u_{n}\right\|_{L^{2}}^{4}
-\frac{2(q-1)b}{q}\left\|u_{n}\right\|_{L^{q}}^{q}\bigg) \\
\nonumber &&\qquad\,\,\, =2+at_{n}^{2}\left\|v_{n}\right\|_{L^{2}}^{2}+\frac{\gamma}{2\pi}V_{0}\left(v_{n}\right)+\frac{\gamma}{2\pi}\log t_{n}\left\|v_{n}\right\|_{L^{2}}^{4} \\
\nonumber &&\qquad\quad\, -\frac{\gamma}{8\pi}\left\|v_{n}\right\|_{L^{2}}^{4}-\frac{2(q-1)b}{q}t_{n}^{6-2q}\left\|v_{n}\right\|_{L^{q}}^{q} \\
\nonumber &&\qquad\,\,\, \geq 2-C\|v_{n}\|_{L^{\frac{8}{3}}}^{4}+o_{n}(1) \\
\nonumber &&\qquad\,\,\, \geq 2-C\|v_{n}\|_{L^{2}}^{3}+o_{n}(1)=2+o_{n}(1),
\end{eqnarray}
which is absurd. Therefore, $\|\nabla u_{n}\|_{L^{2}(\mathbb{R}^{2})}$ ($n\in\mathbb{N}$) are bounded provided that $q<p+1$.

Now we can deduce from \eqref{eq-a1} that
\begin{eqnarray}\label{37}
&& \frac{\gamma}{4\pi p}\left\|u_{n}\right\|_{L^{p}}^{2p} \\
\nonumber &=& \left\langle I^{\prime}\left(u_{n}\right), u_{n}\right\rangle-J_{2}\left(u_{n}\right)+\left\|\nabla u_{n}\right\|_{L^{2}}^{2}-\frac{(q-2)b}{q}\left\|u_{n}\right\|_{L^{q}}^{q}+\frac{(p-2)\gamma}{2\pi p}V_{0}(u) \\
\nonumber &=&o_{n}(1)+\left\|\nabla u_{n}\right\|_{L^{2}}^{2}-\frac{(q-2)b}{q}\left\|u_{n}\right\|_{L^{q}}^{q} \\
\nonumber &&+(p-2)\left(2I(u_{n})-\left\|\nabla u_{n}\right\|_{L^{2}}^{2}-a\left\|u_{n}\right\|_{L^{2}}^{2}+\frac{2b}{q}\left\|u_{n}\right\|_{L^{q}}^{q}\right) \\
\nonumber &\leq& o_{n}(1)+2(p-2)c+(3-p)\left\|\nabla u_{n}\right\|_{L^{2}}^{2}-\frac{(q-2p+2)b}{q}\left\|u_{n}\right\|_{L^{q}}^{q}\leq C,
\end{eqnarray}
and hence $\{u_{n}\}$ is bounded in $L^{p}(\mathbb{R}^{2})$. By the Gagliardo-Nirenberg inequality, $\{u_{n}\}$ is bounded in $L^{q}(\mathbb{R}^{2})$. As a consequence, it follows from \eqref{26} that $\{u_{n}\}$ is bounded in $L^{2}(\mathbb{R}^{2})$ and hence $\{u_{n}\}$ is bounded in $H^{1}(\mathbb{R}^{2})$. This concludes our proof of Lemma \ref{lemma13}.
\end{proof}

Base on Lemma \ref{lemma13}, we can derive the following compactness property (modulo translations) for the Cerami sequences satisfying \eqref{eq-a0} and \eqref{eq-a1}.
\begin{lem}\label{lemma14}
Assume $p\geq2$, $q\geq 2p-2$ and $q>2$. If $p\geq3$, suppose the sequence $\{u_{n}\}\subset X$ satisfies \eqref{eq-a0} with $\|I'(u_{n})\|_{X'}\left(1+\|u_{n}\|_{X}\right)\rightarrow0$ as $n\rightarrow+\infty$. If $2\leq p<3$, suppose the sequence $\{u_{n}\}\subset X$ satisfies \eqref{eq-a1} and $q\geq2p-1$ or $q<p+1$. Then, by extracting a subsequence, one of the following two conclusions must hold: \\
\noindent (a) $\|u_{n}\|_{H^{1}(\mathbb{R}^{2})}\rightarrow0$ and $I(u_{n})\rightarrow0$ as $n\rightarrow+\infty$. \\
\noindent (b) There exists a sequence $\{x_{n}\}\in\mathbb{R}^{2}$ such that
\begin{equation}\label{eq-a2}
  x_{n}\ast u_{n}\rightarrow u \quad \text{strongly in} \,\, X, \qquad \text{as} \,\, n\rightarrow+\infty,
\end{equation}
where $u\in X\setminus\{0\}$ is a critical point of functional $I$.
\end{lem}
\begin{proof}
From Lemma \ref{lemma13}, we already know that $\{u_{n}\}$ is bounded in $H^{1}(\mathbb{R}^{2})$. Now suppose $(a)$ does not hold for any subsequence of $\{u_{n}\}$, then we will prove that $(b)$ must hold up to a subsequence. To this end, we first show the following non-vanishing property:
\begin{equation}\label{38}
  \liminf_{n\rightarrow+\infty}\sup_{y\in\mathbb{R}^{2}}\int_{B_{2}(y)}u_{n}^{2}(x)dx>0.
\end{equation}
Suppose on the contrary that \eqref{38} does not hold, then after passing to a subsequence, it follows that
\begin{equation}\label{39}
  \sup_{y\in\mathbb{R}^{2}}\int_{B_{2}(y)}u_{n}^{2}\rightarrow0, \qquad \text{as} \,\, n\rightarrow+\infty.
\end{equation}
Thus Lions' vanishing lemma (Lemma \ref{vanish}) implies that $u_{n}\rightarrow0$ in $L^{r}(\mathbb{R}^{2})$ for any $2<r<+\infty$. Consequently, by \eqref{2-12}, we have
\begin{equation}\label{40}
  V_{2}(u_{n})\rightarrow0, \qquad \|u_{n}\|_{L^{q}}\rightarrow0, \qquad \text{as} \,\, n\rightarrow+\infty.
\end{equation}
Note that
\begin{equation}\label{41}
  o_{n}(1)=\left\langle I'(u_{n}),u_{n} \right\rangle=\|u_{n}\|_{H^{1}}^{2}+\frac{\gamma}{2\pi}\left(V_{1}(u_{n})-V_{2}(u_{n})\right)-b\|u_{n}\|_{L^{q}}^{q},
\end{equation}
we get
\begin{equation}\label{42}
  \|u_{n}\|_{H^{1}}^{2}+\frac{\gamma}{2\pi}V_{1}(u_{n})=\left\langle I'(u_{n}),u_{n} \right\rangle+\frac{\gamma}{2\pi}V_{2}(u_{n})+b\|u_{n}\|_{L^{q}}^{q}=o_{n}(1),
\end{equation}
and hence
\begin{equation}\label{43}
  \|u_{n}\|_{H^{1}(\mathbb{R}^{2})}\rightarrow0, \quad V_{1}(u_{n})\rightarrow0,  \qquad \text{as} \,\, n\rightarrow+\infty.
\end{equation}
Therefore, we arrive at
\begin{equation}\label{44}
  I(u_{n})=\frac{1}{2}\|u_{n}\|_{H^{1}}^{2}+\frac{\gamma}{4p\pi}\left(V_{1}(u_{n})-V_{2}(u_{n})\right)-\frac{b}{q}\|u_{n}\|_{L^{q}}^{q}\rightarrow0,
\end{equation}
as $n\rightarrow+\infty$, which contradicts our assumption that $(a)$ does not hold for any subsequence of $\{u_{n}\}$. Thus the non-vanishing property \eqref{38} must hold.

Based on the on-vanishing property \eqref{38}, we can find a sequence $\{x_{n}\}\subset\mathbb{R}^{2}$ such that the sequence of the translated functions $\widetilde{u_{n}}:=(-x_{n})\ast u_{n}\in X$ ($n\in\mathbb{N}$) satisfies
\begin{equation}\label{45}
  \liminf_{n\rightarrow+\infty}\int_{B_{2}(0)}\widetilde{u_{n}}^{2}(x)dx>0
\end{equation}
and $\widetilde{u_{n}}\rightharpoonup u\not\equiv 0$ in $H^{1}(\mathbb{R}^{2})$. By passing to a subsequence, we may also assume that $\widetilde{u_{n}}\rightarrow u$ a.e. in $\mathbb{R}^{2}$. Then, we can show that $\{\widetilde{u_{n}}\}$ is bounded in $X$, and
\begin{equation}\label{46}
  \left\langle I'(\widetilde{u_{n}}),\widetilde{u_{n}}-u \right\rangle\rightarrow0, \qquad \text{as} \,\, n\rightarrow+\infty.
\end{equation}
Once \eqref{46} was established, we can prove that $\widetilde{u_{n}}\rightarrow u$ strongly in $X$ as $n\rightarrow+\infty$ and $I'(u)=0$. The proof is entirely similar to the cases that $a(x)$ is a $\mathbb{Z}^{2}$-periodic function, so we omit it here. For details of the rest of the proof, please see the proof of Theorem \ref{Cerami} in subsection 3.1. This finishes the proof of Lemma \ref{lemma14}.
\end{proof}

Now we are ready to complete the proof of Theorem \ref{thm3}.
\begin{proof}[Proof of Theorem \ref{thm3} (completed)]
From Lemma \ref{lemma12} and Lemma \ref{lemma14}, we have already derived the mountain pass solution to \eqref{S-N}, that is, there exists $u\in X\setminus\{0\}$ such that $I'(u)=0$ and $I(u)=c_{mp}>0$. Thus we have finished the proof of (i) in Theorem \ref{thm3}.

Next, we aim to show the existence of ground state solution to \eqref{S-N}. To this end, we define the set of critical points:
\begin{equation}\label{47}
  \mathcal{K}:=\{u\in X\setminus\{0\}\mid I'(u)=0\},
\end{equation}
which is nonempty. Extract a sequence $\{u_{n}\}\subset\mathcal{K}$ such that
\begin{equation}\label{48}
  \lim_{n\rightarrow+\infty}I(u_{n})=c_{g}:=\inf_{u\in\mathcal{K}}I(u)\in[-\infty,c_{mp}].
\end{equation}
Therefore, by the definition of $\mathcal{K}$ and Lemma \ref{Pohozaev}, the sequence $\{u_{n}\}$ satisfies both \eqref{eq-a0} and \eqref{eq-a1}. Moreover, from \eqref{2-49} in Lemma \ref{lemma2}, we infer that
\begin{equation}\label{49}
  \liminf_{n\rightarrow+\infty}\|u_{n}\|_{H^{1}(\mathbb{R}^{2})}\geq\alpha>0.
\end{equation}
Therefore, by Lemma \ref{lemma14}, there exists a sequence $\{x_{n}\}\in\mathbb{R}^{2}$ and a critical point $u\in X\setminus\{0\}$ of $I$ such that, by extracting a subsequence,
\begin{equation}\label{50}
  x_{n}\ast u_{n}\rightarrow u \qquad \text{strongly in} \, X, \qquad \text{as} \,\, n\rightarrow+\infty.
\end{equation}
It follows that $u\in\mathcal{K}$ and
\begin{equation}\label{51}
  I(u)=\lim_{n\rightarrow+\infty}I(x_{n}\ast u_{n})=\lim_{n\rightarrow+\infty}I(u_{n})=c_{g}>-\infty,
\end{equation}
and hence $u$ is the ground state solution to \eqref{S-N}. This completes the proof of (ii) in Theorem \ref{thm3}.

Finally, we will prove (iii) in Theorem \ref{thm3}. In the following, we take $k=1$ \emph{if and only if} $p\geq3$ and $k=2$ \emph{if and only if} $2\leq p<3$. Let us define the sets
\begin{equation}\label{52}
  \mathcal{M}_{k}:=\{u\in X\setminus\{0\}\mid J_{k}(u)=0\}, \qquad k=1,2.
\end{equation}
It follows from Lemma \ref{Pohozaev} that $\mathcal{K}\subset \mathcal{M}_{k}$ for $k=1,2$. For any $u\in X\setminus\{0\}$ and any $t>0$, let $Q_{k}(t,u):=u_{t,k}\in X\setminus\{0\}$ ($k=1,2$), i.e.,
\begin{equation}\label{53}
  Q_{k}(t,u)(x):=u_{t,k}(x)=t^{k}u(tx), \qquad \forall \, x\in\mathbb{R}^{2}.
\end{equation}
We define the minimal energy value on $\mathcal{M}_{k}$ ($k=1,2$) by
\begin{equation}\label{55}
  c_{\mathcal{M}_{k}}:=\inf_{u\in\mathcal{M}_{k}}I(u),
\end{equation}
then it satisfies
\begin{equation}\label{56}
  c_{\mathcal{M}_{k}}\leq c_{g}\leq c_{mp}.
\end{equation}
We also define the minimax value
\begin{equation}\label{57}
  c_{mm,k}:=\inf_{u\in X\setminus\{0\}}\sup_{t>0}I(u_{t,k}), \qquad k=1,2.
\end{equation}

We have the following lemma.
\begin{lem}\label{lemma15}
Suppose $a(x)=a>0$, $p\geq2$, $q\geq2p-2$ and $q>2$. Let $k=1$ \emph{if and only if} $p\geq3$ and $k=2$ \emph{if and only if} $2\leq p<3$. Assume further $q\geq 2p-1$ if $2\leq p<3$.

\smallskip

\noindent(i) For any $u\in X\setminus\{0\}$, there exists a unique $t_{u,k}>0$ such that $Q_{k}(t_{u},u)\in\mathcal{M}_{k}$, which is the unique maximum point of the function $I(Q_{k}(t,u))$.

\smallskip

\noindent(ii) The map $X\setminus\{0\}\rightarrow(0,+\infty)$, $u\mapsto t_{u,k}$ is continuous.

\smallskip

\noindent(iii) Every $u\in\mathcal{M}_{k}$ with $I(u)=c_{\mathcal{M}_{k}}$ is a critical point of I which does not change sign in $\mathbb{R}^{2}$.
\end{lem}
\begin{proof}
For any $u\in X\setminus\{0\}$, consider the function $\psi_{u,k}(t):=I(Q_{k}(u,t))$ ($k=1,2$). Through direct calculations, we have
\begin{equation}\label{54}\begin{aligned}
\psi_{u,k}(t)=& \frac{t^{2k}}{2}\int_{\mathbb{R}^{2}}|\nabla u|^{2} \mathrm{d} x+\frac{t^{2(k-1)}}{2} \int_{\mathbb{R}^{2}}au^{2} \mathrm{d} x-\frac{t^{2(kp-2)}\gamma}{4p\pi}\log t\left(\int_{\mathbb{R}^{2}}|u|^{p} \mathrm{d} x\right)^{2} \\
&+\frac{t^{2(kp-2)}\gamma}{4p\pi}\int_{\mathbb{R}^{2}} \int_{\mathbb{R}^{2}} \log (|x-y|)|u|^{p}(x)|u|^{p}(y) \mathrm{d} x \mathrm{d} y-\frac{t^{kq-2}b}{q}\int_{\mathbb{R}^{2}}|u|^{q}\mathrm{d} x.
\end{aligned}\end{equation}

Now we need the following calculus lemma.
\begin{lem}\label{lemma16}
Suppose $p\geq2$, $q\geq2p-2$ and $q>2$. Let $k=1$ \emph{if and only if} $p\geq3$ and $k=2$ \emph{if and only if} $2\leq p<3$. Assume further $q\geq 2p-1$ if $2\leq p<3$. If $C_{1},C_{2},C_{4}>0$, $C_{5}\geq0$, $C_{3}\in\mathbb{R}$, then the function $f_{k}:(0,+\infty)\rightarrow\mathbb{R}$ defined by
\begin{equation}\label{58}
  f_{k}(t):=C_{1}t^{2(k-1)}+C_{2}t^{2k}+C_{3}t^{2(kp-2)}-C_{4}t^{2(kp-2)}\log t-C_{5}t^{kq-2}
\end{equation}
has a unique positive critical point $t_{k}$ such that $f_{k}'(t)>0$ for any $t<t_{k}$ and $f_{k}'(t)<0$ for any $t>t_{k}$.
\end{lem}
\begin{proof}
This lemma can be proved through elementary calculus and direct calculations, so we omit it. For more details in the proof, please refer to \cite{R}.
\end{proof}

By Lemma \ref{lemma16}, the function $\psi_{u,k}$ has a unique critical point $t_{u,k}>0$ such that
\begin{equation}\label{59}
\psi_{u,k}^{\prime}(t)>0 \quad \text { for } t \in\left(0, t_{u,k}\right) \quad \text { and } \quad \psi_{u,k}^{\prime}(t)<0 \quad \text { for } t>t_{u,k}.
\end{equation}
Through direct calculations, we can also obtain $\psi_{u,k}^{\prime}(t)=\frac{J_{k}(Q_{k}(u,t))}{t}$ for any $t>0$, and hence (i) holds. Due to the continuity of the map $X\setminus\{0\}\rightarrow\mathbb{R}$, $u\mapsto\psi'_{u,k}(t)$ for arbitrarily fixed $t>0$ and \eqref{59}, we can infer that the map $X\setminus\{0\}\rightarrow(0,+\infty)$, $u\mapsto t_{u,k}$ is continuous, and hence gives (ii).

Now we only need to prove (iii). Let $u\in\mathcal{M}_{k}$ such that $I(u)=c_{\mathcal{M}_{k}}$. We will show that $u$ is a critical point of $I$ via contradiction arguments. If not, assume that there exists a $v\in X$ such that $\left\langle I'(u),v\right\rangle<0$. Since $I$ is a $C^{1}$-functional on $X$, we could choose an $\epsilon>0$ sufficiently small such that, for any $\tau\in(0,\epsilon)$, every $\omega\in X$ with $\|\omega\|_{X}<\epsilon$ and every $\hat{v}\in X$ with $\|\hat{v}-v\|_{X}<\epsilon$,
\begin{equation}\label{60}
I(u+\omega+\tau\hat{v})\leqslant I(u+\omega)-\epsilon\tau.
\end{equation}
Since $u\in\mathcal{M}_{k}$, one has $t_{u,k}=1$. Combining this with (ii), we may choose $\tau\in(0,\epsilon)$ small enough such that, for $t_{\tau,k}:=t_{u+\tau v,k}$,
\begin{equation}\label{61}
\quad\left\|Q_{k}\left(t_{\tau,k}, u\right)-u\right\|_{X}<\epsilon \quad \text { and } \quad\left\|Q_{k}\left(t_{\tau,k}, v\right)-v\right\|_{X}<\epsilon,
\end{equation}
Let $\omega_{k}=Q_{k}\left(t_{\tau,k}, u\right)-u$ and $\hat{v}_{k}=Q_{k}\left(t_{\tau,k},v\right)$, then \eqref{60} and \eqref{61} implies
\begin{equation}\label{62}\begin{aligned}
I\left(Q_{k}\left(t_{\tau,k},u+\tau v\right)\right)&=I\left(Q_{k}\left(t_{\tau,k}, u\right)+\tau Q_{k}\left(t_{\tau,k}, v\right)\right)
=I(u+\omega_{k}+\tau\hat{v}_{k}) \\
&\leqslant I(u+\omega_{k})-\epsilon\tau<I(u+\omega_{k})=I\left(Q_{k}\left(t_{\tau,k}, u\right)\right) \leqslant I(u)=c_{\mathcal{M}_{k}}.
\end{aligned}\end{equation}
Since $Q_{k}\left(t_{\tau,k},u+\tau v\right)\in\mathcal{M}_{k}$, this contradicts the definition of $c_{\mathcal{M}_{k}}$ and hence $I'(u)=0$.

Finally, if $u\in\mathcal{M}_{k}$ is a minimizer of $I|_{\mathcal{M}_{k}}$, note that $I(u)=I(|u|)$ and $J(u)=J(|u|)$, hence $|u|\in\mathcal{M}_{k}$ is also a minimizer of $I|_{\mathcal{M}_{k}}$. Thus $|u|$ is a critical point of $I$ and satisfies the Euler-Lagrange equation \eqref{S-N}. By Lemma \ref{lemma1}, we have $|u|\in W^{2,r}(\mathbb{R}^{2})$ ($\forall \, 1\leq r<+\infty$) is a strong solution to elliptic equation of the form $-\Delta |u|+\zeta(x)|u|=0$ with $\zeta\in L_{loc}^{\infty}(\mathbb{R}^{2})$. Since $|u|\not\equiv0$, by Harnack inequality (see \cite{GT}), we have $|u|>0$ on $\mathbb{R}^{2}$ and hence $u$ does not change sign. This concludes our proof of Lemma \ref{lemma15}.
\end{proof}

\begin{lem}\label{lemma17}
Suppose $a(x)=a>0$, $p\geq2$, $q\geq2p-2$ and $q>2$. Let $k=1$ \emph{if and only if} $p\geq3$ and $k=2$ \emph{if and only if} $2\leq p<3$. Assume further $q\geq 2p-1$ if $2\leq p<3$. Then, we have $c_{g}=c_{\mathcal{M}_{k}}=c_{mm,k}=c_{mp}$.
\end{lem}
\begin{proof}
From \eqref{56}, we have $c_{\mathcal{M}_{k}}\leq c_{g}\leq c_{mp}$. Lemma \ref{lemma11} yields that $c_{mp}\leq c_{mm,k}$, while Lemma \ref{lemma15} (i) tells us that $c_{\mathcal{M}_{k}}=c_{mm,k}$. Thus we must have $c_{g}=c_{\mathcal{M}_{k}}=c_{mm,k}=c_{mp}$. This finishes the proof of Lemma \ref{lemma17}.
\end{proof}

Consequently, (iii) in Theorem \ref{thm3} follows immediately from (i) and (ii) in Theorem \ref{thm3}, Lemma \ref{lemma15} and Lemma \ref{lemma17}. This concludes our proof of Theorem \ref{thm3}.
\end{proof}

\section{Symmetry of positive solutions}

In this section, we will carry out the proof of Theorem \ref{thm2} and Corollary \ref{cor2}.

\begin{proof}[Proof of Theorem \ref{thm2}]
Assume $p\geq2$. For any given classical solution $(u,w)$ to \eqref{g-S-P} and \eqref{conditions}, by Agmon's theorem (see \cite{A}), there exist some $A>0$ and $C>0$ such that
\begin{equation}\label{5-0}
  0<u(x)\leq Ce^{-A|x|}, \qquad \text{as} \,\, |x|\rightarrow\infty.
\end{equation}
Moreover, by the Liouville theorem, we have, there exists a constant $c\in\mathbb{R}$ such that
\begin{equation}\label{5-1}
w(x)=-\frac{1}{2\pi}\int_{\mathbb{R}^{2}}\log(|x-y|)u^{p}(y)dy+c, \qquad \forall \,\, x \in \mathbb{R}^{2}.
\end{equation}

We will prove Theorem \ref{thm2} by applying the method of moving planes (see \cite{CD,CL,CLO,CW,DFHQW,DFQ,GNN,MZ}). To this end, we need some standard notations. For arbitrary $\lambda\in\mathbb{R}$, we set
\begin{equation}\label{5-2}
\Sigma_{\lambda} :=\left\{x \in \mathbb{R}^{2} : x_{1}>\lambda\right\}, \qquad T_{\lambda}=\partial \Sigma_{\lambda}=\left\{x \in \mathbb{R}^{2} : x_{1}=\lambda\right\}.
\end{equation}
For any $x=(x_{1},x_{2})\in\mathbb{R}^{2}$, let $x^{\lambda}:=(2\lambda-x_{1},x_{2})$ be the reflection of $x$ about the plane $T_{\lambda}$. For any $\lambda \in \mathbb{R}$, we define the following notations:
\begin{equation}\label{5-3}
u^{\lambda}(x)=u\big(x^{\lambda}\big), \qquad w^{\lambda}(x)=w\big(x^{\lambda}\big), \qquad \forall \,\, x \in \mathbb{R}^{2},
\end{equation}
\begin{equation}\label{5-4}
u_{\lambda}(x)=u^{\lambda}(x)-u(x), \qquad w_{\lambda}(x)=w^{\lambda}(x)-w(x), \qquad \forall \,\, x \in \Sigma_{\lambda}.
\end{equation}
Then the difference functions $(u_{\lambda},w_{\lambda})$ satisfy the following system:
\begin{equation}\label{5-5}
\left\{\begin{array}{l}{-\Delta u_{\lambda}+a_{0}u_{\lambda}=\gamma w_{\lambda}\left(u^{\lambda}\right)^{p-1}+\left(\gamma(p-1)w\psi_{\lambda}^{p-2}+\sigma_{\lambda}\right)u_{\lambda} \quad \text { in } \Sigma_{\lambda},} \\ {-\Delta w_{\lambda}=p\xi_{\lambda}^{p-1}u_{\lambda} \quad \text { in } \Sigma_{\lambda},}\end{array}\right.
\end{equation}
where $\psi_{\lambda}(x)$ and $\xi_{\lambda}(x)$ are valued between $u(x^{\lambda})$ and $u(x)$ by mean value theorem, and the function $\sigma_{\lambda}$ is given by
\begin{equation}\label{5-6}
\sigma_{\lambda}(x):=\left\{\begin{array}{ll}{\frac{f\left(u^{\lambda}(x)\right)-f(u(x))}{u_{\lambda}(x)}+a_{0},} & {\text { if } u_{\lambda}(x) \neq 0,} \\ {a_{0},} & {\text { if } u_{\lambda}(x)=0.}\end{array}\right.
\end{equation}
Since $f$ is Lipschitz on $\big(0,\|u\|_{L^{\infty}(\mathbb{R}^{2})}\big]$, there exists a constant $C=C(f,\|u\|_{L^{\infty}(\mathbb{R}^{2})})>0$ such that
\begin{equation}\label{5-7}
\left\|\sigma_{\lambda}\right\|_{L^{\infty}\left(\Sigma_{\lambda}\right)} \leq C, \qquad \forall \,\, \lambda \in \mathbb{R}.
\end{equation}
Furthermore, for $p>2$, the assumption \eqref{monotonicity} on $f$ implies that
\begin{equation}\label{5-13}
  \sigma_{\lambda}(x)\leq0, \qquad \text{if} \,\,\, 0<u^{\lambda}(x)\neq u(x)<\epsilon_{0}.
\end{equation}
From \eqref{5-1}, we derive that
\begin{equation}\label{5-8}
w_{\lambda}(x)=\frac{1}{2\pi}\int_{\Sigma_{\lambda}}\log\left(\frac{\left|x-y^{\lambda}\right|}{|x-y|}\right)\left((u^{\lambda})^{p}(y)-u^{p}(y)\right)dy, \qquad \forall \,\, x \in \Sigma_{\lambda}.
\end{equation}
It follows directly from \eqref{5-8} that $u_{\lambda}\geq0$ in $\Sigma_{\lambda}$ implies immediately $w_{\lambda} \geq 0$ in $\Sigma_{\lambda}$. Moreover, note that
\begin{equation}\label{5-9}
0 \leq \log \frac{\left|x-y^{\lambda}\right|}{|x-y|} \leq \log \left(1+\frac{\left|y-y^{\lambda}\right|}{|x-y|}\right) \leq \frac{\left|y-y^{\lambda}\right|}{|x-y|}=\frac{2\left(y_{1}-\lambda\right)}{|x-y|}, \qquad \forall \,\, x, y \in \Sigma_{\lambda},
\end{equation}
we can infer from the integral formula \eqref{5-8} that, for any $x \in \Sigma_{\lambda}$,
\begin{equation}\label{5-10}
w_{\lambda}^{-}(x) \geq \frac{1}{2\pi}\int_{\Sigma^{-}_{\lambda}}\frac{2\left(y_{1}-\lambda\right)}{|x-y|}\left((u^{\lambda})^{p}(y)-u^{p}(y)\right)dy \geq \frac{p}{\pi}\int_{\Sigma^{-}_{\lambda}}\frac{y_{1}-\lambda}{|x-y|}u^{p-1}(y)u_{\lambda}^{-}(y)dy,
\end{equation}
where $v^{-}:=\min\{v,0\}$ denotes the negative part of a function $v$ and the set
\begin{equation}\label{5-11}
  \Sigma^{-}_{\lambda}:=\{x\in\Sigma_{\lambda}:\,u_{\lambda}(x)<0\}.
\end{equation}
Thus, by Hardy-Littlewood-Sobolev inequality, we get
\begin{equation}\label{5-12}
\left\|w_{\lambda}^{-}\right\|_{L^{2}\left(\Sigma_{\lambda}\right)}\leq C\left(\int_{\Sigma^{-}_{\lambda}}\left(y_{1}-\lambda\right)^{2}u^{2(p-1)}(y)dy\right)^{\frac{1}{2}}\left\|u_{\lambda}^{-}\right\|_{L^{2}\left(\Sigma_{\lambda}\right)}, \qquad \forall \,\, \lambda \in \mathbb{R}.
\end{equation}
By \eqref{conditions}, \eqref{5-0}, \eqref{5-1}, \eqref{5-7} and \eqref{5-13}, we conclude that there exists a $\overline{\lambda}>0$ sufficiently large such that, for any $\lambda\geq\overline{\lambda}$,
\begin{equation}\label{5-14}
  \gamma(p-1)w\psi_{\lambda}^{p-2}+\sigma_{\lambda}\leq0 \qquad \text{in} \,\, \Sigma^{-}_{\lambda}.
\end{equation}
Now, multiplying the first equation in system \eqref{5-5} by $u^{-}_{\lambda}$ and integrating, we can deduce from \eqref{5-12} and \eqref{5-14} that
\begin{eqnarray}\label{5-15}
a_{0}\left\|u_{\lambda}^{-}\right\|_{L^{2}\left(\Sigma_{\lambda}\right)}^{2} &\leq& \int_{\Sigma^{-}_{\lambda}}\left[\gamma w_{\lambda} \big(u^{\lambda}\big)^{p-1}u_{\lambda}^{-}+\left(\gamma(p-1)w\psi_{\lambda}^{p-2}+\sigma_{\lambda}\right)\left(u_{\lambda}^{-}\right)^{2}\right]dx \\
\nonumber &\leq& \gamma\int_{\Sigma^{-}_{\lambda}} w_{\lambda}^{-}\big(u^{\lambda}\big)^{p-1}u_{\lambda}^{-}dx
\leq\gamma\left\|w_{\lambda}^{-}\right\|_{L^{2}\left(\Sigma_{\lambda}\right)}\big\|u^{\lambda}\big\|^{p-1}_{L^{\infty}\left(\Sigma_{\lambda}\right)}
\left\|u_{\lambda}^{-}\right\|_{L^{2}\left(\Sigma_{\lambda}\right)} \\
\nonumber &\leq& C\left(\int_{\Sigma^{-}_{\lambda}}\left(y_{1}-\lambda\right)^{2}u^{2(p-1)}(y)dy\right)^{\frac{1}{2}}\|u\|^{p-1}_{L^{\infty}\left(\mathbb{R}^{2}\right)}
\left\|u_{\lambda}^{-}\right\|_{L^{2}\left(\Sigma_{\lambda}\right)}^{2}.
\end{eqnarray}
It is clear that there exists a $\lambda_{0}\geq\overline{\lambda}$ large enough such that
\begin{equation}\label{5-16}
 C\left(\int_{\Sigma^{-}_{\lambda}}\left(y_{1}-\lambda\right)^{2}u^{2(p-1)}(y)dy\right)^{\frac{1}{2}}\|u\|^{p-1}_{L^{\infty}\left(\mathbb{R}^{2}\right)}\leq\frac{a_{0}}{2} \qquad \text {for any} \,\, \lambda\geq \lambda_{0},
\end{equation}
and hence, \eqref{5-15} implies that $u^{-}_{\lambda}\equiv0$, i.e., $u_{\lambda}\geq0$ in $\Sigma_{\lambda}$ for any $\lambda\geq\lambda_{0}$.

Now, we move the plane $T_{\lambda}$ to the left as long as $u_{\lambda}\geq0$ in $\Sigma_{\lambda}$ until its limiting position. We define
\begin{equation}\label{5-17}
  \lambda_{1}:=\inf\{\lambda\in\mathbb{R}:\,u_{\lambda}\geq0 \,\,\, \text{in} \,\,\, \Sigma_{\lambda}\}<+\infty.
\end{equation}
It follows from \eqref{5-0} that $\lambda_{1}>-\infty$. By continuity, we have
\begin{equation}\label{5-18}
  u_{\lambda_{1}}\geq0 \quad \text{and} \quad w_{\lambda_{1}}\geq0 \qquad \text{in} \,\,\Sigma_{\lambda_{1}}.
\end{equation}

Next, we will show that
\begin{equation}\label{5-19}
  u_{\lambda_{1}}\equiv w_{\lambda_{1}}\equiv0 \qquad \text{in} \,\,\Sigma_{\lambda_{1}}.
\end{equation}
In fact, suppose \eqref{5-19} does not hold, then we deduce from \eqref{5-8} that $u_{\lambda_{1}}\not\equiv0$ and $w_{\lambda_{1}}>0$ in $\Sigma_{\lambda_{1}}$, and hence Hopf Lemma yields that
\begin{equation}\label{5-20}
\frac{\partial w_{\lambda_{1}}}{\partial x_{1}}=-2\frac{\partial w}{\partial x_{1}}>0 \quad \text { on } \,\, T_{\lambda_{1}}.
\end{equation}
By the first equation in system \eqref{5-5}, we have
\begin{equation}\label{5-21}
-\Delta u_{\lambda_{1}}+\left(a_{0}-\gamma(p-1)w\psi_{\lambda_{1}}^{p-2}-\sigma_{\lambda_{1}}\right)u_{\lambda_{1}}\geq \gamma w_{\lambda_{1}}\big(u^{\lambda_{1}}\big)^{p-1}>0 \quad \text { in } \Sigma_{\lambda_{1}},
\end{equation}
hence the strong maximum principle implies
\begin{equation}\label{5-25}
  u_{\lambda_{1}}>0 \qquad \text{in} \,\, \Sigma_{\lambda_{1}}.
\end{equation}
Thus by Hopf Lemma, we get
\begin{equation}\label{5-22}
\frac{\partial u_{\lambda_{1}}}{\partial x_{1}}=-2 \frac{\partial u}{\partial x_{1}}>0 \quad \text { on } \,\, T_{\lambda_{1}}.
\end{equation}

By \eqref{conditions}, \eqref{5-0}, \eqref{5-1}, \eqref{5-7} and \eqref{5-13}, we can fix a $R>1$ sufficiently large such that
\begin{equation}\label{5-23}
\gamma(p-1)w\psi_{\mu}^{p-2}+\sigma_{\mu}\leq 0 \qquad \text { in } \,\, \Sigma^{-}_{\mu}\setminus B_{R}(0), \quad \forall \,\, \mu \in \mathbb{R},
\end{equation}
and
\begin{equation}\label{5-24}
C\left(\int_{\mathbb{R}^{2}\setminus B_{R}(0)}\left(y_{1}-\mu\right)^{2}u^{2(p-1)}(y)dy\right)^{\frac{1}{2}}\|u\|^{p-1}_{L^{\infty}(\mathbb{R}^{2})}
<\frac{a_{0}}{2}, \qquad \forall \,\, \mu\in[\lambda_{1}-1,\lambda_{1}],
\end{equation}
where the constant is the same as in \eqref{5-12}, \eqref{5-15} and \eqref{5-16}.

By \eqref{5-25}, \eqref{5-22} and the continuity of $u$ and $\frac{\partial u}{\partial x_{1}}$, we can prove that there exists a $0<\varepsilon<1$ small enough such that
\begin{equation}\label{5-26}
u_{\mu}\geq 0 \qquad \text { in } \Sigma_{\mu}\cap B_{R}(0), \quad \forall \,\, \mu \in(\lambda_{1}-\varepsilon,\lambda_{1}].
\end{equation}
Indeed, suppose on contrary that there exists a sequence $\mu_{n}\rightarrow\lambda_{1}$ such that $u_{\mu_{n}}<0$ in $\Sigma_{\mu_{n}}\cap B_{R}(0)$. Then there exists a sequence of points $z_{n}\in\Sigma_{\mu_{n}}\cap B_{R}(0)$ such that
\begin{equation}\label{5-27}
  u_{\mu_{n}}(z_{n})<0 \qquad \text{and} \qquad \frac{\partial u_{\mu_{n}}}{\partial x_{1}}(z_{n})\leq0.
\end{equation}
After passing to a subsequence, we may assume $z_{n}\rightarrow z_{0}\in \overline{\Sigma_{\lambda_{1}}\cap B_{R}(0)}$. By the continuity of $u$ and \eqref{5-25}, we have $z_{0}\in T_{\lambda_{1}}$. Thus by the continuity of $\frac{\partial u}{\partial x_{1}}$, it follows that
\begin{equation}\label{5-28}
  \frac{\partial u_{\lambda_{1}}}{\partial x_{1}}(z_{0})=\lim_{n\rightarrow\infty}\frac{\partial u_{\mu_{n}}}{\partial x_{1}}(z_{n})\leq0,
\end{equation}
which contradicts with \eqref{5-22}. Therefore, \eqref{5-26} must hold.

Similar to the proof of \eqref{5-15}, multiplying the first equation in system \eqref{5-5} by $u^{-}_{\mu}$ and integrating, we can deduce from \eqref{5-12}, \eqref{5-23}, \eqref{5-24} and \eqref{5-26} that, for any $\mu \in(\lambda_{1}-\varepsilon,\lambda_{1}]$,
\begin{eqnarray}\label{5-29}
&&a_{0}\left\|u_{\mu}^{-}\right\|_{L^{2}\left(\Sigma_{\mu}\right)}^{2}\leq\int_{\Sigma^{-}_{\mu}\setminus B_{R}(0)}\left[\gamma w_{\mu} \big(u^{\mu}\big)^{p-1}u_{\mu}^{-}+\left(\gamma(p-1)w\psi_{\mu}^{p-2}+\sigma_{\mu}\right)\left(u_{\mu}^{-}\right)^{2}\right]dx \\
\nonumber &\leq& \gamma\int_{\Sigma^{-}_{\mu}\setminus B_{R}(0)} w_{\mu}^{-}\big(u^{\mu}\big)^{p-1}u_{\mu}^{-}dx
\leq\gamma\left\|w_{\mu}^{-}\right\|_{L^{2}\left(\Sigma_{\mu}\right)}\big\|u^{\mu}\big\|^{p-1}_{L^{\infty}\left(\Sigma_{\mu}\right)}
\left\|u_{\mu}^{-}\right\|_{L^{2}\left(\Sigma_{\mu}\right)} \\
\nonumber &\leq& C\left(\int_{\Sigma^{-}_{\mu}}\left(y_{1}-\mu\right)^{2}u^{2(p-1)}(y)dy\right)^{\frac{1}{2}}\|u\|^{p-1}_{L^{\infty}\left(\mathbb{R}^{2}\right)}
\left\|u_{\mu}^{-}\right\|_{L^{2}\left(\Sigma_{\mu}\right)}^{2}<\frac{a_{0}}{2}\left\|u_{\mu}^{-}\right\|_{L^{2}\left(\Sigma_{\mu}\right)}^{2}.
\end{eqnarray}
As a consequence, \eqref{5-29} implies that $u^{-}_{\mu}\equiv0$, i.e., $u_{\mu}\geq0$ in $\Sigma_{\mu}$ for any $\mu \in(\lambda_{1}-\varepsilon,\lambda_{1}]$. This is a contradiction with the definition of $\lambda_{1}$. Therefore, \eqref{5-19} must hold, that is, $u$ and $w$ are symmetric and strictly decreasing with respect to the hyperplane $\{x\in\mathbb{R}^{2}: \, x_{1}=\lambda_{1}\}$.

Repeating the same moving plane procedure with $x_{1}$-coordinate direction replaced by the second coordinate direction $x_{2}$, we will also find a $\lambda_{2}\in\mathbb{R}$ such that $u$ and $w$ are symmetric and strictly decreasing with respect to the hyperplane $\{x\in\mathbb{R}^{2}: \, x_{2}=\lambda_{2}\}$. By the strictly decreasing property, it is clear that every symmetry hyperplane of $u$ and $w$ must contain the point $x_{0}:=(\lambda_{1},\lambda_{2})\in\mathbb{R}^{2}$. Consequently, by repeating the moving plane procedure in an arbitrary direction in place of the $x_{1}$-coordinate direction, we can obtain that $u$ and $w$ must be symmetric and strictly decreasing with respect to any hyperplane containing the point $x_{0}$, and hence radially symmetric and strictly decreasing with respect to the point $x_{0}$. This concludes our proof of Theorem \ref{thm2}.
\end{proof}

\begin{proof}[Proof of Corollary \ref{cor2}]
Assume $p\geq2$ and $q\geq2$. Let $u\in X$ be a positive solution to the Schr\"{o}dinger-Newton equation \eqref{S-N} with constant $a>0$. By Lemma \ref{lemma1}, $(u,w)$ with
$$w:=-\frac{1}{2\pi}\int_{\mathbb{R}^{2}}\log(|x-y|)u^{p}(y)dy$$
is a classical solution to the Schr\"{o}dinger-Poisson system \eqref{g-S-P} satisfying the condition \eqref{conditions}, where the locally Lipschitz function $f(u):=b|u|^{q-2}u-au$. Therefore, Corollary \ref{cor2} follows immediately from Theorem \ref{thm2}.
\end{proof}

\end{document}